\documentclass[11pt]{article}

\usepackage[utf8]{inputenc}
\usepackage[T1]{fontenc}
\usepackage[british]{babel}
\usepackage[a4paper]{geometry}
\usepackage{amsmath,amssymb,amsthm}
\usepackage[shortlabels]{enumitem}
\usepackage[all]{xy}
\usepackage{tikz}
\usepackage{multicol}
\usepackage{float}
\usepackage{graphicx}
\usepackage{wrapfig}
\usetikzlibrary{automata, arrows, arrows.meta, positioning}

\theoremstyle{plain}
\newtheorem{theo}{Theorem}[section]
\newtheorem{lem}[theo]{Lemma}
\newtheorem{cor}[theo]{Corollary}
\newtheorem{prop}[theo]{Proposition}
\theoremstyle{definition}
\newtheorem{defi}[theo]{Definition}

\newtheorem{rem}[theo]{Remark}

\theoremstyle{plain}

\theoremstyle{definition}

\theoremstyle{plain}
\newtheorem{theorem}{Theorem}

\title{The $\ell$-modular local theta correspondence}
\author{Justin Trias}
\date{}

\begin{document}

\maketitle

\begin{abstract} We study the validity of the local theta correspondence over a non-archimedean local field in the context of modular representation theory \textit{i.e.} for representations with coefficient fields of positive characteristic. For a symplectic-orthogonal or a unitary-unitary dual pair over a $p$-adic field, we obtain a bijective correspondence, as long as the characteristic of the coefficient field is large enough compared to the size of the dual pair, and call it the modular local theta correspondence. \end{abstract}

\section*{Introduction}

The classical theta correspondence for complex representations is at the source of significant applications in number theory and automorphic forms, as it helped establishing the Langlands correspondence for some groups \cite{gan_takeda_gsp4} and has deep relations to special values of local constants \cite{waldspurger_shimura,waldspurger_shimura_quaternions} and formal degree \cite{gan_ichino}. The global theta correspondence pieces together data obtained from the local theta correspondence, coming from representations of real and $p$-adic groups.

On the other hand, the study of modular representation theory of $p$-adic groups originates in \cite{vigneras_gl_2} and was motivated by conjectures of Serre related to congruences between automorphic representations. It has been an active topic of research since then. We examine in this work the validity of a local theta correspondence in the modular setting, since it can lead to new arithmetic applications involving congruences or proving new instances of the local Langlands correspondence in the modular setting.

The local theta correspondence takes place in the following framework. We first need to introduce the Weil representation, which is at the heart of this correspondence.

\paragraph{The Weil representation.} Let $F$ be a $p$-adic field with residual cardinality $q$. The main body of the paper also deals with non-archimedean local fields of odd positive characteristic, but we focus in this introduction on the $p$-adic case for simplicity. Let $W$ be a symplectic space over $F$ of finite dimension. The metaplectic group $\textup{Mp}(W)$ is the unique non-trivial central extension of $\textup{Sp}(W)$ by $\{ \pm 1\}$.

Let $R$ be an algebraically closed field of characteristic $\ell \neq p$. The field $R$ will be our coefficient field in the $\ell$-modular setting -- we always exclude the case $\ell = p$ because it requires completely different techniques. Fix a non-trivial additive character $\psi : F \to R^\times$. We recall some facts about the Weil representation $\omega_{\psi}^R \in \textup{Rep}_R(\textup{Mp}(W))$ with coefficients in $R$ as defined in \cite{trias_theta1}. Over the complex numbers, the Weil representation is usually pinned down (see \cite[Th 3.5]{rao}) using its unitary structure and some positivity condition. Similarly, the complex Fourier transforms are canonically normalised thanks to some ambient positivity (\textit{e.g.} $\sqrt{q} \in \mathbb{R}_+$) that typically determines square roots of positive quantities. For positive characteristic fields, speaking of unitarity does not make sense as there is no meaningful sense for positivity. Moreover, there is no canonical normalisation of Fourier transforms, because we always have to make an arbitrary choice of sign when picking up square roots. In order to solve these issues, the construction in \cite{trias_theta1} provides a more natural normalisation for Fourier transforms by the interplay of the non-normalised Weil factor and builds the Weil representation in a very direct way. A special feature appears when $\ell = 2$ as $\omega_\psi^R \in \textup{Rep}_R(\textup{Sp}(W))$ in this case.

Let $(U(V_1),U(V_2))$ be a type I dual pair in $\textup{Sp}(W)$. Lifting $U(V_1) \times U(V_2) \to \textup{Sp}(W)$ to $H_1 \times H_2 \to \textup{Mp}(W)$ yields a pullback $\omega_{\psi,H_1,H_2}$ of the Weil representation to the dual pair. For the sake of exposition, we will assume the Weil representation is a representation of $U(V_1) \times U(V_2)$ and simply write $\omega_\psi$. This assumption is harmless if neither $H_1$ nor $H_2$ is the metaplectic group and it avoids lengthy technical considerations about splittings in the metaplectic group.

\paragraph{The bijective theta map.} Investigating the local theta correspondence in the modular setting is asking whether a certain set of statements, that we now define, are true or not. Let $\pi_1 \in \textup{Irr}_R(U(V_1))$ and let $(\omega_\psi)_{\pi_1}$ be the biggest $\pi_1$-isotypic quotient of the Weil representation. By definition $(\omega_\psi)_{\pi_1}$ factors any morphism $\omega_\psi \to \pi_1$ in $\textup{Rep}_R(U(V_1))$. This quotient still carries an action of $U(V_2)$ classically denoted by:
$$(\omega_\psi)_{\pi_1} \simeq \pi_1 \otimes_R \Theta(\pi_1) \textup{ where } \Theta(\pi_1) \in \textup{Rep}_R(U(V_2)).$$
We first consider the statement:
\begin{itemize}[align=left]
\item[(Fin)] $\Theta(\pi_1)$ has finite length;
\end{itemize}
If (Fin) holds, the maximal semisimple quotient $\theta(\pi_1)$ of $\Theta(\pi_1)$, also called cosocle, is well-defined. Assuming (Fin), we add the following two statements, where the notation $\theta(\pi_1) \cap \theta(\pi_1')$ means counting common irreducible factors with multiplicities:
\begin{itemize}[align=left]
\item[(Irr)]  \ $\theta(\pi_1)$ is irreducible or $0$;
\item[(Uni)] $\theta(\pi_1) \cap \theta(\pi_1') \neq 0$ implies $\pi_1 \simeq \pi_1'$.
\end{itemize}
We want to study the validity of (Fin)-(Irr)-(Uni) for all $\pi_1 \in \textup{Irr}_R(U(V_1))$ together with the other three reverse statements \reflectbox{(Fin)}-\reflectbox{(Irr)}-\reflectbox{(Uni)} going from $U(V_2)$ to $U(V_1)$. These six statements altogether are redundant though, since \reflectbox{(Irr)} and \reflectbox{(Uni)} can be deduced from the other four. We call $(\Theta / R)$ the collection of these six -- or four -- statements. When they are true, we say that the theta correspondence holds over $R$ and the set $\{ (\pi_1,\pi_2) \ | \ \omega_\psi \twoheadrightarrow \pi_1 \otimes_R \pi_2 \}$ is the graph of a bijection, denoted $\theta$ in both directions, on the subsets $\textup{Irr}_R^\theta(U(V_i)) = \{ \pi_i \ | \ \omega_\psi \twoheadrightarrow \pi_i\}$ of irreducible representations contributing:
$$\textup{Irr}_R^\theta(U(V_1)) \overset{\theta}{\simeq} \textup{Irr}_R^\theta(U(V_2)).$$

\paragraph{Banal theta correspondence.} In this paper, our main goal is to prove:

\begin{theorem} \label{banal_non_explicit_theta_thm_intro} We assume $(U(V_1),U(V_2))$ is not quaternionic. Then for all but finitely many $\ell$, the theta correspondence holds over $R$. \end{theorem}

The case $\ell = 0$ corresponds to the classical theta correspondence and is therefore always included in this ``all but finitely many'' result. In the complex case, or more generally over characteristic zero fields, it has been a continued effort since the late 1970s \cite{kudla_invent,mvw,wald,gt,gan_sun} to prove this result for all type I dual pairs and all residue characteristics $p$. In the modular case, we have just stated that the correspondence is valid for almost all coefficient fields $R$. However, the theta correspondence does not hold for arbitrary $R$ as the following counter-example from Section \ref{counter_example_sec} shows.

\begin{theorem} Consider the dual pair $(\textup{SL}_2(F),O_{2(m+2)}(F))$ and suppose $\ell$ divides $q-1$. Then $\Theta(1_{U(V_2)}) \simeq \pi_1 \oplus \pi_1'$ has length $2$ with $\pi_1 \not\simeq \pi_1'$. \end{theorem}

In this counter-example, both (Uni) and \reflectbox{(Irr)} are failing. It contradicts a key feature of the complex setting that takes place in the so-called stable range, that is when one group is much larger than the other. Roughly speaking, most theta lifts are irreducible in the stable range. This counter-example actually uses the knowledge we have over the complex numbers in the stable range about the reducibility points of parabolic induced representations, by comparing it to reducibility points in the modular setting. The fact that no explicit bound on $\ell$ appears in our Theorem \ref{banal_non_explicit_theta_thm_intro} is deeply related to these irreducibility questions in the stable range: we only obtain positive answers to them in Section \ref{proof_banal_theta_sec} by a generic irreducibility argument \textit{i.e.} over a non-empty (dense) open set of $\textup{Spec}(\mathbb{Z}[1/p])$. We can produce an explicit bound on $\ell$ if we assume some irreducibility in the stable range. We assume again that $(U(V_1),U(V_2))$ is not quaternionic and we denote by $V_1^\square$ the doubled space $V_1 \oplus (-V_1)$.

\begin{theorem} \label{banal_explicit_theta_thm_intro} Suppose $\Theta(1_{U(V_2)}) \in \textup{Rep}_R(U(V_1^\square))$ is irreducible. Then if the order of $q$ modulo $\ell$ is larger than $4 \ \textup{max}(\textup{dim}(V_1),\textup{dim}(V_2))$, the theta correspondence holds over $R$. \end{theorem}

This bound is not the sharpest as we want to avoid the introduction of cumbersome notation, but we refer the reader to Section \ref{proof_banal_theta_sec} for the sharper bound we obtain. The irreducibility assumption we introduced in the theorem is actually quite deep and resulted from a series of papers \cite{kudla_rallis,kudla_sweet,sweet,ban_jantzen,yamana}. It is beyond the scope of the present article to determine for which explicit $\ell$ this irreducibility assumption survives, even though we would expect it to be superfluous with the explicit bound we gave.

\paragraph{Dealing with $(\Theta/R)$.} We will now explain the methods we use in the article, as well as reviewing contributions in the classical setting, and explain which proofs can be carried out easily, and which proofs seriously break down in the modular setting and need replacing, pointing out similarities and obstacles along the way. 

\paragraph{Proving (Fin).} In the complex setting, this is obtained in \cite{kudla_invent} as a consequence of a very important result in the cuspidal case, which asserts $\Theta(\pi_1) = \theta(\pi_1)$ is irreducible or $0$ when $\pi_1$ is cuspidal. This is a crucial and unavoidable step in the proof of finite length over the complex numbers. Unfortunately, the argument relies on the projectivity/injectivity of cuspidal representations, which fails in general over $R$ without further assumptions. There are two ways to resolve this matter.

If the characteristic $\ell$ of $R$ is large enough compared to the size of the dual pair $((U(V_1),U(V_2))$, then cuspidal representations are projective/injective and we can still carry out the proofs from \cite{kudla_invent}. In this situation, we say $\ell$ is banal -- this only makes sense \textit{with respect to} the size of the dual pair -- and the bound is very explicit in terms of the pro-orders of the dual pair. In our Section \ref{Rallis_Kudla_filtrations_consequences_sec}, we make sure to push as far as we can the tools from \cite{kudla_invent} for general $R$, before stating sharper corollaries when $\ell$ is banal. This approach allows us to introduce Witt tower techniques, as well as first occurrence index, over any $R$.

When the characteristic $\ell$ is non-banal with respect to $(U(V_1),U(V_2))$, we describe new ways to obtain finite length. These appear in Appendix \ref{finite_length_app}, to avoid breaking the flow of reading. We actually give two different results and both rely on the so-called finiteness of Hecke algebras. However, this is currently unknown for the metaplectic group, even though our strategy would easily include it if this finiteness result was known, together with the depth decomposition.

The first proof uses the lattice model, therefore $p$ must be odd. We prove some finiteness properties for the Weil representation: it is (locally) finitely generated and admissible, not over $R$, but over larger rings such as the Bernstein centres of either of the two groups in the dual pair. This implies the finite length of isotpyic quotients. The second proof shows that $\Theta(\pi_1)$ has finite length if it is finitely generated. It uses the so-called generalised doubling method and there is no restriction on $p$. The idea is rather similar to the first proof by proving a certain representation is (locally) finitely generated and admissible over some Bernstein centre type rings.

\paragraph{About (Irr)+(Uni).} Following the discussion of the previous paragraph, there remain a few cases in which (Fin) is not known. Therefore, the statements (Irr) and (Uni) are not well-defined. We prefer to replace them by:
\begin{itemize}[align=left]
\item[(Irr')] \ if $\Theta(\pi_1) \twoheadrightarrow \pi_2$ with $\pi_2$ is semsimple, then $\pi_2$ has length at most $1$;
\item[(Uni')] if $\Theta(\pi_1) \twoheadrightarrow \pi_2 \twoheadleftarrow \Theta(\pi_1')$ with $\pi_2$ non-zero semisimple, then $\pi_1 \simeq \pi_1'$.
\end{itemize}
In order to study these statements, we follow \cite{gt} very closely. This approach excludes quaternionic dual pairs however -- due to the use of the MVW involution. Again, we push as much as we can their techniques over general $R$ \textit{i.e.} without assuming $\ell$ is banal. In various places though, their proofs need rewriting for at least two reasons. On the one hand, the lack of (Fin) is an obstacle to the existence of the cosocle. On the other hand, second adjunction is not known for metaplectic groups, so we have to take extra care and make sure second adjunction can be replaced, everywhere it appears, by Casselman duality, which is valid for all admissible representations. 

Our strategy is the same as \cite{gt}, which is a proof by strong induction, and occupies the whole of Section \ref{proof_modular_theta_sec}. The boundary of degenerate principal series, first introduced in \cite{kudla_rallis_boundary}, play a central role in the proof. We refer to Section \ref{boundary_of_I_u_sec} for the precise definition of the boundary.

The representations outside the boundary are the easiest to deal with. We study them in Theorem \ref{rep_outside_boundary_thm} provided an hypothesis related to the doubling method holds. We call this hypothesis (H) in the main body of the paper. It is always satisfied in the complex setting and Lemma \ref{H_is_implied_by_irred_Theta_lem} shows that (H) is weaker than the irreducibility hypothesis of Theorem \ref{banal_explicit_theta_thm_intro}. Regarding representations on the boundary, they are addressed thanks to the key Proposition \ref{theta_lifts_irreducible_socles_prop}, which allows to reduce the size of the dual pair and use the induction hypothesis to prove the correspondence. We show the induction steps for irreducibility and uniqueness in Section \ref{proof_banal_theta_sec}. The final form of our banal theta correspondence appears in two forms, in Theorem \ref{conditional_banal_theta_thm} and Theorem \ref{non_explicit_banal_theta_thm}. They have the same flavour as Theorem \ref{banal_explicit_theta_thm_intro} and Theorem \ref{banal_non_explicit_theta_thm_intro} \textit{i.e.} the first is conditional but explicit whereas the second is unconditional but not explicit.

\paragraph{Organisation of the paper.} The first section consists of preliminaries, recalling usual definitions and results in the framework of $\epsilon$-hermitian spaces, modular representation theory, covering groups and the Weil representation. We do our best to avoid talking about splittings of (the lifts of) dual pairs in the metaplectic group, as addressing precisely these questions requires a lot of technicality. We decide instead to rely on our preliminaries, especially the section on generalised (genuine) regular representations, as well as on our first appendix, to bypass these considerations. The second section focuses on proving the MVW involution exists in the modular setting. According to the early 2025 version of \cite{theta_book}, there might be a gap in the literature regarding the MVW involution for metaplectic groups over function fields. Our method is different from the complex setting though: we need to consider all conjugacy classes of trace characters due to the lack of Harish-Chandra's regularity theorem in the modular setting. We can prove the existence of the MVW involution under a certain tameness assumption, which means in the function field case that the characteristic has to be large enough. We hope our proofs help to partially fill in this potential gap, or to shed some light on the original strategy of \cite{mvw}, as they are valid over general coefficients fields. The third section introduces the classical tools appearing in \cite{kudla_invent} to study the correspondence, such as Witt towers, Rallis' argument, the two filtrations -- named after their authors in this manuscript -- as well as the generalised doubling method. The fourth section follows rather usual strategies based on the tools of the previous section. It deals with the persistence principle, the first occurrence index, the cuspidal case and the preservation of the cuspidal support. The last section is dedicated to the proof of the theta correspondence in the banal setting, including a counter-example in the non-banal setting. Appendix~\ref{finite_length_app} explores new strategies to prove finite length. We very briefly mentioned it, but the goal of Appendix \ref{compatibility_Weil_rep_app} is deeply related to the preliminaries about covering groups and generalised (genuine) regular representations, with the key results that allow us avoid talking about splittings of dual pairs.

\paragraph{Acknowledgement:}  I am very indebted to Alberto M\'inguez for introducing me to this topic during my PhD, and would like to express my gratitude to him, and also to Shaun Stevens, for their constant encouragement and for insightful feedback on earlier drafts. I am also grateful to Wee Teck Gan and Nadir Matringe for their thoughtful reviews of my PhD manuscript, which served as the foundation of this paper. I would also like to thank Raphaël Beuzart-Plessis, Jean-François Dat, Johannes Girsch, Guy Henniart, Rob Kurinczuk, Thomas Lanard and Vincent Sécherre, for useful discussions. This work has benefited from the support of the EPSRC grant EP/V061739/1.

\tableofcontents

\section{General preliminaries}

\subsection{Notations} \label{notations_section}

In the main body of the paper, we always denote by $F$ a non-archimedean local field of characteristic not 2, residual characteristic $p$ and residual cardinality $q$. The usual norm $| \cdot |_F = q^{-\textup{val}_F(\cdot)}$ on $F$ endows it with the locally profinite topology. The latter topology is equivalent to ``locally compact and totally disconnected''.

\paragraph{The coefficient field $R$.} We always denote by $R$ an algebraically closed field of characteristic $\ell \neq p$. So $\ell$ is either zero or a prime number different from $p$. In general, we have to choose a square root $q^{\frac{1}{2}}$ (of the image) of $q$ in $R$ -- mainly to normalise our induction functors. How we choose this square root of $q$ does not matter, but we have to bear in mind a choice has been made. In various places, we do not make a difference between $\{ \pm 1 \}$ and its image $\mu_2(R)$ in $R$ via $\mathbb{Z} \to R$. When $\ell = 2$, it is a small abuse of notations because $\mu_2(R) = \{ 1 \}$, but we prefer to avoid the bulky less explicit $\mu_2(R)$.

\paragraph{The character $\psi$.} A smooth character $F \to R^\times$ is a group morphism with open kernel. In the manuscript, we always denote by $\psi : F \to R^\times$ a non-trivial smooth character. One exists because $R$ has enough $p$-power roots of unity. 

\paragraph{Smooth representations.} Let $G$ be a locally profinite group. Let $A$ be a commutative unitary ring. An $A[G]$-module $V$ is said to be smooth if, for all $v \in V$, the stabiliser of $v$ is open in $G$. We denote by $\textup{Rep}_A(G)$ the category of smooth $A[G]$-modules. We may also call these smooth $A[G]$-modules, smooth representations of $G$, or simply representations. When $H$ is a closed subgroup of $G$, the induction functor $\textup{Ind}_H^G$ associates to $(\sigma,W) \in \textup{Rep}_A(H)$ the representation $\textup{Ind}_H^G(W) \in \textup{Rep}_A(G)$ of locally constant functions on $G$ taking values in $W$ and satisfying $f(h g) = \sigma(h) \cdot f(g)$ for all $g \in G$ and $h \in H$. The compact induction $\textup{ind}_H^G$ is the subfunctor of $\textup{Ind}_H^G$ made of functions compactly supported modulo $H$, that is the subspace of functions $f \in \textup{Ind}_H^G(W)$ such that the image of $\textup{supp}(f)$ in $H \backslash G$ is a compact set. A representation $V \in \textup{Rep}_A(G)$ is said to be admissible if, for all compact open subgroups $K$ in $G$, the set of $K$-invariants $V^K = \{ v \in V \ | \ K \cdot v = v \}$ is finitely generated as an $A$-module.

\paragraph{Haar measures and pro-order.} We continue with the $G$ and $A$ from the previous paragraph and recall notations from \cite[I.1 \& I.2]{vig_book}. The pro-order $|G|$ of $G$ is the least common multiple, in the sense of supernatural integers, of the orders of its open compact subgroups. To be more explicit, $|G|$ is a function $\mathcal{P} \to \mathbb{N} \cup {\infty}$ on the set of prime numbers $\mathcal{P}$. This pro-order $|G|$ decomposes into two parts with disjoint supports, namely the finite part $|G|_f$ and the infinite one $|G|_\infty$. The only situation occuring in the present work is $|G| = |G|_f \times |G|_\infty$ with $|G|_\infty = p^\infty$ and $|G|_f$ prime-to-$p$. As long as $p$ is invertible in $A$, there exists a Haar measure on $G$ with values in $A$ \textit{i.e.} a non-zero left $G$-equivariant morphism $C_c^\infty(G,A) \to 1_G^A$ where $C_c^\infty(G,A)$ is the space of locally constant and compactly supported functions in $G$ with values in $A$ and $(1_G^A,A)$ is the trivial representation.

\paragraph{The space $W$.} Let $(W,\langle \ , \ \rangle)$ be a symplectic vector space of finite dimension over $F$. Its isometry group is composed of the $F$-linear invertible endomorphisms preserving the form $\langle \ , \ \rangle$ and is classically denoted $\textup{Sp}(W)$. A vector subspace $X$ of $W$ is totally isotropic if $\forall x, x' \in X, \langle x,x' \rangle=0$. A Lagrangian in $W$ is a maximal totally isotropic subspace; lagrangians are exactly totally isotropic subspaces of dimension $m$ where $n=2m$ is the dimension of $W$. Any finite dimensional vector space over $F$ inherits the locally profinite topology coming from $F$. This is the case for $W$ and $\textup{End}_F(W)$. In particular, the symplectic group $\textup{Sp}(W)$ is locally profinite via the subspace topology coming from $\textup{End}_F(W)$. The stabiliser group $P(X)$ of a totally isotropic subspace $X$ in $W$ is a maximal parabolic subgroup of $\textup{Sp}(W)$.

\subsection{Largest $\pi$-isotypic quotient or $\pi$-coinvariants} \label{largest-isotypic-quotient-sec}

Let $G$ be a locally profinite group and $R$ be a field of characteristic $\ell$. We assume $G$ contains a compact open subgroup of invertible pro-order in $R$. We let $\pi \in \textup{Irr}_R(G)$ be admissible. We define the largest $\pi$-isotypic quotient -- also known as $\pi$-coinvariants -- of $V \in \textup{Rep}_R(G)$ as the quotient: 
$$V_\pi = V / V[\pi] \textup{ where } V[\pi] = \bigcap_{f \in \textup{Hom}_G(V,\pi)} \textup{ker}(f).$$
As the name suggests, $V_\pi$ is $\pi$-isotypic. Furthermore, this construction is compatible with other commuting actions in the following sense. Let $A$ be an $R$-algebra, not necessarily commutative, and assume $V$ has a structure of $A$-module commuting with the $G$-action \textit{i.e.} the $A$-module structure $A \to \textup{End}_k(V)$ on $V$ actually lands in $\textup{End}_G(V)$. Then $V_\pi$ is naturally an $A$-module.

Defining the contradregient $\pi^\vee$ of $\pi$, we alternatively have $V_\pi \simeq \pi \otimes_R (V \otimes \pi^\vee)_{1_G}$ and we can call $(V \otimes \pi^\vee)_{1_G}$ the multiplicty space of $V_\pi$. Note that $(V \otimes \pi^\vee)_{1_G}$ naturally inherits the structure of $A$-module if $V$ had one.

We are especially interested in the following situation. Assume $R$ is algebraically closed. Let $G$ and $H$ be two locally profinite groups admitting compact open subgroups of invertible pro-order in $R$. Let $V \in \textup{Rep}_R(G \times H)$ be a smooth representation and $\pi \in \textup{Irr}_R(G)$. Then there exists a unique (up to isomorphism) $\Theta(\pi) \in \textup{Rep}_R(H)$ such that $V_\pi \simeq \pi \otimes_R \Theta(\pi)$. Furthermore $\Theta(\pi) \simeq (V \otimes \pi^\vee)_{1_G}$.

\paragraph{Exactness.} The $\pi$-coinvariants gives an endofunctor $V \mapsto V_\pi$ of $\textup{Rep}_R(G)$, or $\textup{Rep}_A(G)$ taking into account an extra structure, that is easily seen to be right exact.

\begin{prop} The functor $V \mapsto V_\pi$ is exact if $\pi$ is an injective object in $\textup{Rep}_R(G)$. \end{prop}

\begin{proof} Since this functor is right exact, we only have to check that it preserves injections. For $i : U \hookrightarrow V$ in $\textup{Rep}_R(G)$, we have $f \in \textup{Hom}_{R[G]}(V,\pi) \twoheadrightarrow f \circ i \in \textup{Hom}_{R[G]}(U,\pi)$ because $\pi$ is injective. Therefore:
$$U[\pi] = \bigcap_{f \in \textup{Hom}_G(U,\pi)} \textup{ker}(f) = \bigcap_{f \in \textup{Hom}_G(V,\pi)} \textup{ker}(f \circ i) = i^{-1}(V[\pi]).$$
In particular $U \to V_{\pi}$ has kernel $i^{-1}(i(U) \cap V[\pi]) = U[\pi]$, so $i_\pi : U_\pi \hookrightarrow V_\pi$. \end{proof}

\subsection{Projective vs. injective objects} \label{projective_vs_injective_objects_sec}

Let $R$ be an algebraically closed field and $G$ be a locally profinite group having a compact open subgroup of invertible pro-order. We want to understand how the notions of projective and injective representations compare to one another. In the classical theta correspondence, some proofs heavily rely on the exactness of the largest $\pi$-isotypic quotient, which is only valid when $\pi$ is an injective object.

\paragraph{Locally profinite groups.} We start by recalling the following result, which is a consequence of \cite[I.4.13 2) \& I.7.8 a) \& I.7.9]{vig_book}.

\begin{prop} \label{projective_V_V_vee_prop} Let $V \in \textup{Irr}_R(G)$ be admissible. The following are equivalent:
\begin{itemize}
\item $V$ is projective;
\item $V^\vee$ is projective.
\end{itemize}
Moreover in this case, both $V$ and $V^\vee$ are injective. \end{prop}

Therefore projectivity implies injectivity and ensures the exactness of the largest isotypic quotient. Going the other way round, we need to introduce the concept of compact representations. A representation $V$ is compact if its coefficients are compact \textit{i.e.} $c_{v \otimes_R v^\vee} : g \mapsto v^\vee(g v)$ has compact support for all $v \otimes_R v^\vee \in V \otimes_R V^\vee$. Equivalently, the coefficients $v \otimes_R v^\vee \mapsto c_{v \otimes_R v^\vee}$ induce a non-zero $(G \times G)$-morphism $V \otimes_R V^\vee \to C_c^\infty(G)$. Note that this morphism is injective when $V \otimes_R V^\vee \in \textup{Rep}_R(G)$ is irreducible. Recall a projective irreducible admissible representation is necessarily compact by \cite[I.7.9]{vig_book}.

\begin{prop} \label{injective_compact_implies_projective_prop} Let $V \in \textup{Irr}_R(G)$ be compact and injective. Then $V$ is projective. \end{prop}

Before proving the proposition, we want to highlight a lemma we use in the proof as it is interesting in its own right. We state it in the context of smooth representations, but the proof works for any category of non-degenerate modules $\mathcal{M}(\mathcal{A})$ over a ring with idempotents $\mathcal{A}$ to show $\textup{Hom}_{\mathcal{M}(\mathcal{A})}(A,-)$ is exact. We consider the regular representation as a $G$-module for either the left or right $G$-action.

\begin{lem} \label{regular_rep_is_projective_lem} Let $A$ be a commutative ring such that $G$ has an open compact subgroup of invertible pro-order. Then $C_c^\infty(G) \in \textup{Rep}_A(G)$ is projective. \end{lem}

\begin{proof} Blanc gave a proof (see \cite[Th A.4]{casselman_shintani}) of this result when $A=\mathbb{C}$. His proof works verbatim over a coefficient field. Roughly speaking, it consists in taking a vector space section of a surjective $G$-morphism and averaging it over $G$ to make it $G$-equivariant. Since this proof breaks down over coefficient rings, because surjective maps do not necessarily have sections, we found it interesting to give a new argument.

We study the functor $\mathcal{F} = \textup{Hom}_G(C_c^\infty(G),-)$. As $C_c^\infty(G) \simeq \textup{ind}_1^G(1) = \underset{\to}{\textup{lim}} \ \textup{ind}_K^G(1_K)$:
$$\mathcal{F} \simeq \underset{\leftarrow}{\textup{lim}} \ \textup{Hom}_G(\textup{ind}_K^G(1_K),-) \simeq \underset{\leftarrow}{\textup{lim}} \ \textup{Hom}_K(1_K,-) = \underset{\leftarrow}{\textup{lim}} \ (-)^K.$$
The last functor here is also known as the completion functor $V \mapsto \bar{V}$, where the transition maps are $V^K \to V^{K'}$ if $K \subset K'$. Therefore the functor $F$ is isomorphic to the completion functor. It is enough to consider the system over compact open subgroups $K$ of invertible pro-orders, as they are cofinal in the set of compact open subgroups ordered for the reverse inclusion. Take an exact sequence:
$$0 \to U \to V \to W \to 0$$
then for all compact open subgroups $K$ of invertible pro-order we have:
$$0 \to U^K \to V^K \to W^K \to 0.$$
Since for $K \subset K'$ of invertible pro-orders the transition map $(-)^K \twoheadrightarrow (-)^{K'}$ is surjective, all these systems are Mittag-Leffler and we deduce:
$$0 \to \underset{\leftarrow}{\textup{lim}} \ U^K \to \underset{\leftarrow}{\textup{lim}} \ V^K \to \underset{\leftarrow}{\textup{lim}} \ W^K \to 0.$$
In particular we have proved that the completion functor $V \mapsto \bar{V}$ is exact. \end{proof}

\begin{rem} Even though $C_c^\infty(G)$ is projective for the left or right $G$-action, it is not necessarily projective for both these actions \textit{i.e.} $C_c^\infty(G) \in \textup{Rep}_A(G \times G)$ is not projective in general. \end{rem}

\begin{proof}[Proof of Proposition \ref{injective_compact_implies_projective_prop}.] Because $V$ is compact, it is admissible \cite[I.7.5]{vig_book} and $V^\vee$ is irreducible \cite[I.4.18]{vig_book}. We obtain $V \otimes_R V^\vee \hookrightarrow C_c^\infty(G)$. To show $V$ is projective, we show $V$ is a direct factor of the regular representation $C_c^\infty(G)$ and use Lemma \ref{regular_rep_is_projective_lem} to conclude. So we want to find a retract of $V \otimes_R V^\vee \hookrightarrow C_c^\infty(G)$. By injectivity of $V$, we can complete the diagram:
$$\xymatrix{
		V \otimes_R V^\vee \ar@{^{(}->}[r] \ar@{->}[d]^{\textup{id}} & \ar@{-->}[dl] C_c^\infty(G) \\
		V \otimes_R V^\vee & 
		}.$$
The dotted arrow is \textit{a priori} not $(G \times G)$-equivariant, but simply $G$-equivariant. This is sufficient to conclude that $V$ is a direct factor of $C_c^\infty(G)$. We simply point out that the dotted arrow map is actually realising the largest $V$-isotypic quotient, so it has to be $(G \times G)$-equivariant. In particular $V \otimes V^\vee$ is a direct factor of $C_c^\infty(G)$, which is a stronger statement as it happens in $\textup{Rep}_R(G \times G)$. \end{proof}

We finally point out a complementary lemma based on $\textup{Hom}_G(V,W^\vee) \simeq \textup{Hom}_G(W,V^\vee)$ and the fact that admissible representations are reflexive \textit{i.e.} $(V^\vee)^\vee=V$.

\begin{prop} Let $\pi \in \textup{Irr}_R(G)$ be injective. Then $\textup{Hom}_G(\pi,-)$ is exact on the full subcategory of admissible representations. \end{prop}

In order to lift projectivity from admissible representations to the whole category, we will need more precise results involving finiteness properties of Hecke algebras in the context of reductive groups.

\paragraph{Reductive groups.} Let $G$ be a reductive group over a non-archimedean local field $F$ of residual characteristic $p$. We consider $\textup{Rep}_R(G)$ the category of smooth representations with coefficients in $R$. We have the following nice lemma, inspired by \cite[Sec 8]{chan_savin} and obtained from the depth decomposition and the finitess results on relative Hecke algebras \cite{dhkm-finiteness} (they are finite over their centres and their centres are noetherian $R$-algebras).

\begin{lem} \label{projective_injective_chan_savin_lem} Let $P \in \textup{Rep}_R(G)$ be finitely generated. Then $P$ projective $\Leftrightarrow$ $P$ injective. \end{lem}

\begin{cor} Let $\pi \in \textup{Irr}_R(G)$ be irreducible. The following assertions are equivalent:
\begin{enumerate}[label=\textup{\alph*)}]
\item $\pi$ is projective;
\item $\pi$ is injective;
\item $\pi^\vee$ is projective;
\item $\pi^\vee$ is injective.
\end{enumerate} \end{cor}

\begin{proof} We have a) $\Leftrightarrow$ b) and c) $\Leftrightarrow$ d) by Lemma \ref{projective_injective_chan_savin_lem}. Since irreducible $\Rightarrow$ admissible by \cite[II.2.8]{vig_book}, we have a) $\Leftrightarrow$ c) by Propostion \ref{projective_V_V_vee_prop}. This proves all equivalences. \end{proof}

\paragraph{Projective-injective representations.} Let $G$ be a locally profinite group.

\begin{defi} We say $\pi \in \textup{Rep}_R(G)$ admissible is $p\&i$ if it is projective and injective. We denote by $\textup{Irr}_R^{p\&i}(G)$ the set of irreducible $p\&i$ representations. \end{defi}

\begin{rem} Let $G$ be a reductive group with compact centre and $\pi \in \textup{Irr}_R^{p \& i}(G)$. Then $\pi$ is compact \cite[I.7.9]{vig_book} and cuspidal \cite[II.2.7]{vig_book}. In particular $\pi$ must be supercuspidal and we have $\textup{Irr}_R^{p\&i}(G) \subset \textup{Irr}_R^{\textup{scusp}}(G)$. \end{rem}

\paragraph{Sum-up graph.} We sum up below the relations between the different concepts when $\pi \in \textup{Irr}_R(G)$ is admissible. The arrows with no label are valid for $G$ locally profinite.{\parfillskip0pt\par} 
\begin{wrapfigure}{rd}{0.4\textwidth}
    \scalebox{0.5}{\begin{tikzpicture}[every node/.style={align=center,anchor=base,text depth=.5ex,text height=2ex,text width=4em}]
	\node[state] (proj) {proj};
	\node[below = of proj] (belowproj) {};
	\node[right = of proj] (rightproj) {};
	\node[state, right = of rightproj] (form) {form deg};
	\node[state, below = of form] (comp-form) {+};
	\node[state, below = of comp-form] (comp) {comp};
	\node[state, left = of comp] (inj-comp) {+};
	\node[state, left = of inj-comp] (inj)  {inj};
	\node[right = of form] (rightform) {};
	\node[state, right = of comp-form] (cusp) {cusp};
	
	\draw[>=triangle 45, <->] (proj) -- (inj-comp);
	\draw[>=triangle 45, <->] (proj) -- (comp-form);
	\draw[-] (comp) -- (inj-comp);
	\draw[-] (inj) -- (inj-comp);
	\draw[-] (comp) -- (comp-form);
	\draw[-] (form) -- (comp-form);
	\draw[>=triangle 45, <->] (cusp) -- (comp) node[midway,sloped,above] {parab};	
	\draw[>=triangle 45, ->] (cusp) -- (form) node[midway,sloped,above] {ban};
	\draw[>=triangle 45, ->] (inj) -- (proj) node[midway,sloped,above] {fin};
\end{tikzpicture}} \end{wrapfigure}

\noindent The label ``parab'' means $G$ has parabolic theory \textit{i.e.} $G$ admits a set of parabolic subgroups and cuspidal representations can be defined. In this case all irreducible representations are admissible. The other two labels only occur when $G$ has parabolic theory: the label ``ban'' for banal \textit{i.e.} when the characteristic $\ell$ of $R$ does not divide the pro-order of any open subgroups in $G$; the label ``fin'' for finite \textit{i.e.} finiteness properties on Hecke algebras, which is only known over $R$ when $G$ is a reductive group. 

\subsection{Covering groups} \label{covering-groups-sec}

Let $G$ be a locally profinite group and let $A$ be a discrete abelian group, perhaps finite. A topological central extension of $G$ by $A$ is often called an $A$-cover, or a covering group more generally. Any covering group $\tilde{G}$ fits into an exact sequence of topological groups:
$$1 \to A \overset{i}{\to} \tilde{G} \overset{p}{\to} G \to 1$$
where $p$ has local trivialisations and $i$ has central image. In particular $\tilde{G}$ is a locally profinite group.

\begin{rem} Because of the totally disconnectedness of the topology, there exist global trivialisations for $p$. However, there is a rather confusing fact for covering groups: there exist sections $\sigma : G \to \tilde{G}$ of $p$ that are not continuous, such that the $A$-valued $2$-cocyle $c(g,g') = \sigma(g) \sigma(g') \sigma(gg')^{-1}$ is not continuous, and nevertheless the group $G \times_c A$ is still a topological group, where the topology on $G \times_c A$ is locally the product topology, but not globally. See \cite[Rem 2.1]{kaplan_szpruch} for Borel measurable cocycles.  \end{rem}

Let $R$ be an algebraically closed field of characteristic $\ell$ and assume $G$ contains a compact open subgroup of invertible pro-order in $R$. We suppose given a group morphism $\iota : A \to R^\times$ is given. We call a smooth representation $\pi \in \textup{Rep}_R(\tilde{G})$ genuine if $\pi(i(a)) = \iota(a) \textup{id}$ for all $a \in A$. We denote by $\textup{Rep}_R^{\textup{gen}}(\tilde{G})$. We very often drop the ``gen'' superscript and always assume the representations of covering groups to be genuine, unless otherwise stated. Therefore we will simply write $\textup{Rep}_R(\tilde{G})$ or $\textup{Irr}_R(\tilde{G})$ for all genuine smooth representations of $\tilde{G}$. The category of genuine smooth representations is abelian. In the main body of the paper, $A$ will be $\{ \pm 1 \}$ or $R^\times$ and the map $\iota$ will be the natural map to $R^\times$. We call a genuine representation projective (resp. injective) if it is a projective (resp. injective) object in the category of genuine smooth representations.

If $G$ is a connected reductive group over $F$, a parabolic subgroup $P$ in $G$ together with a Levi decomposition $P=MU$ gives rise to a parabolic $\tilde{P}$ in $\tilde{G}$ by taking the inverse image of $P$ and a Levi decomposition $\tilde{P} = \tilde{M} U$ where $\tilde{M}$ is the inverse image of $M$ and $U$ is uniquely embedded into $\tilde{G}$ -- indeed, there exists a unique embedding $\iota : U \hookrightarrow \tilde{G}$ that is $M$-equivariant \cite[Sec 2]{kaplan_szpruch} \textit{i.e.} $\tilde{m} \iota(n) \tilde{m}^{-1} = \iota(m n m^{-1})$ for $\tilde{m} \in \tilde{M}$ and $m=p(\tilde{m})$ and $n \in U$. 

In this case, cuspidal representations are the representations that are killed by all proper Jacquet functors. A representation that is projective must be cuspidal. When $A$ is finite and $p$ does not divide the order of $A$, Moy-Prasad theory makes sense as long as we can find a splitting of the Iwahori by \cite[Sec 3.6]{pan_preservation_depth}. Then the depth decomposition holds for $\tilde{G}$ and, similarly to \cite[II.5]{vig_book}, the parabolic restriction functor preserves finite length because finitely generated and admissible representations have finite length.

Alternatively, if $\ell$ does not divide the pro-order of $\tilde{G}$, parabolic restriction preserves finite length regardless of whether $p$ divides the order of $A$ or not.  The proof is the same as the complex setting \cite[Th 3.5]{kaplan_szpruch}.

We sum up our conclusions for the so-called metaplectic group:

\begin{prop} Let $\textup{Mp}(W)$ be the non-trivial $\{ \pm 1 \}$-cover of $\textup{Sp}(W)$. Assume $p \neq 2$ or $\ell$ does not divide the pro-order of $\tilde{G}$. Then parabolic restriction preserves finite length. \end{prop}

\subsection{Generalised (genuine) regular representations} \label{generalised-regular-reps-sec}

Let $G$ be a locally profinite group. Let $R$ be an algebraically closed field of characteristic \nolinebreak $\ell$. Assume there exists a compact open subgroup of $G$ of invertible pro-order in $R$.

\paragraph{Regular representation.} The usual regular representation $C_c^\infty(G)$ is the space of locally constant compactly supported functions on $G$ together with the $(G \times G)$-action $((g,g') \cdot f )( x ) = f((g')^{-1} x g)$. We will denote by $\rho_l$ the action on the left (second copy of $G$) and $\rho_r$ the action on the right (first copy). We embed $G$ diagonally into $G \times G$ via $g \mapsto (g,g)$ and call this subgroup $G^\Delta$.

\begin{lem} Let $\textup{Reg}(G)=\textup{ind}_{G^\Delta}^{G \times G} (1_G)$. The map $f \mapsto f|_{G \times 1}$ gives an isomorphism of $(G \times G)$-representations $\textup{Reg}(G) \simeq C_c^\infty(G)$. \end{lem}

\paragraph{Generalised regular representations.} Assume we are given an homeomorphism of topological groups $\iota : G \to H$ and a smooth character $\chi$ of $G$. We embed $G$ into $G \times H$ via $g \mapsto (g, \iota(g))$ and we denote $G^\iota$ this subgroup. We also see $\chi$ as a character of $G^\iota$ via $\chi(g,\iota(g))=\chi(g)$. Define the generalised regular representation:
$$\textup{Reg}^{\iota,\chi}(G,H) = \textup{ind}_{G^\iota}^{G \times H}(\chi) \in \textup{Rep}_R(G \times H).$$
Of course $\textup{Reg}^{\textup{id}_G,1_G}(G,G) = \textup{Reg}(G)$. We have some nice properties for these generalised regular representations. They are pullbacks of the regular representation of $G$ \textit{i.e.} $\textup{Reg}^{\iota,1_G}(G,H) \simeq {}^{\iota^{-1}} C_c^\infty(G) = \rho_r \otimes (\rho_l \circ \iota^{-1} )$ as $(G \times H)$-representations and it is easy to see their isomorphism classes only depend on $\iota$ modulo inner automorphisms of $G$ (and $H$). They can also absorb characters of the form $\mu^{-1} \otimes (\mu \circ \iota^{-1})$, because the regular representation itself absorbs $\mu^{-1} \otimes \mu$. Moreover $\textup{Reg}^{\iota,\chi}(G,H) \simeq \chi \otimes \textup{Reg}^{\iota,1_G}(G,H) \simeq \textup{Reg}^{\iota,1_G}(G,H) \otimes (\chi \circ \iota^{-1})$

If $\pi \in \textup{Irr}_R(G)$, then $C_c^\infty(G)_{\pi} \simeq \pi \otimes_R \pi^\vee$. We say it pairs together $\pi$ and $\pi^\vee$. Similarly, the generalised regular representation $\textup{Reg}^{\iota,\chi}(G,H)$ pairs together $\pi$ and $\pi^{[\chi,\iota]} \in \textup{Irr}_G(H)$ where $\pi^{[\chi,\iota]} \simeq (\chi \circ \iota^{-1})(\pi^\vee \circ \iota^{-1}) \in \textup{Irr}_R(H)$. If $\iota = \textup{id}_G$ we shorten it to $\pi^{[\chi]}$. Alternatively $\pi^{[\chi]}$ is the only irreducible representation $\pi' \in \textup{Irr}_R(G)$ such that $\textup{Hom}_{G^\Delta}(\pi \otimes \pi',\chi ) \neq 0$. Note that $\pi^{[1_G]} \simeq \pi^\vee$ and $\pi \mapsto \pi^{[\chi]}$ is an involution.

\paragraph{Genuine tensor products.} Let $\textbf{G} = G \times_c R^\times$ and $\textbf{H} = H \times_{c'} R^\times$ be central extensions such that $c$ and $c'$ are $\{ \pm 1 \}$-valued. We define the group $\textbf{G} \times_{\triangledown} \textbf{H}$ as the quotient of $\textbf{G} \times \textbf{H}$ by the subgroup $\{ ((1,\lambda),(1,\lambda^{-1})) \ | \ \lambda \in R^\times \}$. Therefore $\textbf{G} \times_{\triangledown} \textbf{H}$ is a central extension of $G \times H$ by $R^\times$ whose cocyle is $\{ \pm 1 \}$-valued. Furthermore the restriction to $\textbf{G}$ and $\textbf{H}$ are embeddings of central extensions into $\textbf{G} \times_{\triangledown} \textbf{H}$. Given representations $\pi \in \textup{Rep}_R^{\textup{gen}}(\textbf{G})$ and $\sigma \in \textup{Rep}_R^{\textup{gen}}(\textbf{H})$, the representation $\pi \otimes_R \sigma$ can be considered in an obvious way as a representation of $\textbf{G} \times_{\triangledown} \textbf{H}$. Conversely, a representation of $\textbf{G} \times_{\triangledown} \textbf{H}$ can be pulled back to $\textbf{G} \times \textbf{H}$ and each of the factors acts genuinely. Very often we will use this fact and switch from one point a view to the other. If we are in the situation where $G$ and $H$ are subgroups of $\textup{Sp}(W)$ that commute, then it is known their inverse images $\textbf{G}$ and $\textbf{H}$ in the metaplectic group commute. Moreover the group morphism $((g,\lambda),(h,\mu)) \mapsto (g,\lambda) \cdot (h,\mu)$ factors through $\textbf{G} \times_{\triangledown} \textbf{H}$ and has image $\textbf{GH}$ the inverse image of $G H$ in the metaplectic group. So we can alternatively view a genuine representation of $\textbf{GH}$ as a genuine representation on each factor of $\textbf{G} \times \textbf{H}$ via pullback.

\paragraph{Generalised genuine regular representations.} If $\phi : \textbf{G} \to \textbf{H}$ is an isomorphism of central extensions, then $\phi(g,\lambda) = (\iota(g),\gamma_g \lambda)$ where $\iota : G \overset{\sim}{\to} H$ and $\partial \gamma \cdot c = c'$. Write $\textbf{G}^{\phi}$ for the inverse image of $G^\iota$ under $\textbf{G} \times \textbf{H} \twoheadrightarrow G \times H$, so its elements are of the form $((g,\lambda),(\iota(g),\mu))$. It contains $\{(g,\lambda),(\iota(g),\gamma_g \lambda^{-1}))\}$ which becomes isomorphic to $G$ in $\textbf{G} \times_{\triangledown} \textbf{H}$. Let $\chi$ be a character of $G$ and extend it to a character of $\textbf{G}^\phi$ by setting $\chi((g,\lambda),(\iota(g),\mu)) = \lambda \mu \gamma_g^{-1} \chi(g)$. As earlier, we can define the generalised genuine regular representations as:
$$\textup{Reg}^{\phi,\chi}(\textbf{G},\textbf{H}) = \textup{ind}_{\textbf{G}^\phi}^{\textbf{G} \times \textbf{H}}(\chi).$$
The regular representation of $\textbf{G}$ will be $\textup{Reg}(\textbf{G}) =  \textup{Reg}^{\textup{id}_{\textbf{G}},1_G}(\textbf{G},\textbf{G})$. Note that this is different (and better for our purpose!) from the usual space of functions $f \in C_c^{\infty,\textup{gen}}(\textbf{G})$ that are locally constant compactly supported and transform via $f(g,\lambda) = \lambda f(g,1)$. Indeed, the natural $\textbf{G}$-action $\rho_r$ is genuine but $\rho_l$ in general fails to be genuine. We can define the (genuine) contragredient of $\pi \in \textup{Irr}_R^{\textup{gen}}(\textbf{G})$ as the unique $\pi^\vee \in \textup{Irr}_R^{\textup{gen}}(\textbf{G})$ such that $\textup{Reg}(\textbf{G})_\pi = \pi \otimes_R \pi^\vee$. More generally define $\pi^{[\phi,\chi]} \in \textup{Irr}_R^{\textup{gen}}(\textbf{H})$, or simply $\pi^{[\chi]}$ if $\phi = \textup{id}_{\textbf{G}}$, as the unique $\pi'$ such that $\textup{Reg}^{\phi,\chi}(\textbf{G},\textbf{H})_\pi = \pi \otimes_R \pi'$.

\begin{rem} We may use a variant of $\textup{Reg}^{\phi,\chi}(\textbf{G},\textbf{H})$ when $\iota : G \to H$ is an isomorphism and $\xi$ is a genuine character of $\textbf{G}^\iota = \{ (g,\lambda),(\iota(g),\mu) \}$. Because $\xi$ is genuine, its kernel is of the form $\{((g,\lambda),(\iota(g),\gamma_g \lambda^{-1}))\}$ and induces the isomorphism of central extensions $\phi : (g,\lambda) \mapsto (\iota(g),\gamma_g \lambda)$. We set $\textup{Reg}^{\iota,\xi}(\textbf{G},\textbf{H})=\textup{Reg}^{\phi,1}(\textbf{G},\textbf{H})$. \end{rem}

\section{On invariant distributions} \label{on_invariant_distributions_sec}

Let $X$ be a locally profinite space. Let $G$ be a locally profnite group acting continuously on $X$ \textit{i.e.} the action $G \times X \to X$ is a continuous map, or equivalently, it defines a group morphism $\gamma$ from $G$ to the homeomorphisms of $X$. All our actions are assumed to be continuous. Denote by $R_X^\gamma$ the graph of $\gamma$ \textit{i.e.} the subset $\{ (x, \gamma(g)(x)) \ | \ x \in X, g \in G \}$ of $X \times X$.

\subsection{Regular strata vs. constructive actions} \label{regular_vs_constructive_sec}

A partition of $X$ is a decomposition $X = \bigsqcup_{i \geq 0} X_i$ into countably many disjoint locally closed subsets $X_i$ of $X$. We say it is a $G$-stratification if all $X_i$ are $G$-stable. It is closed/open if all $X_{ \leq j} = \bigsqcup_{i \leq j} X_i$ are open/closed in $X$. In this case we call $(X_i)_{i \geq 0}$ the set of strata associated to the $G$-stratification. A $G$-stratification is finite if there are finitely many (non-empty) strata. Note that a finite closed $G$-stratification can always be turned into an open $G$-stratification in an obvious way.

\begin{defi} We say the action of $G$ on $X$ has regular strata if there exists a $G$-stratification $X=\bigsqcup_{i \geq 0} X_i$, either closed or open, such that $X_i / G$ is Hausdorff for all $i$. Similarly the action has local regular strata if $X_i / G$ is locally Hausdorff for all $i$. \end{defi}

We may add the word ``finite'' to mean the $G$-stratification can be chosen to be finite. Recall from \cite[6.6]{bz1} that the action is said to be constructive if $R_X^\gamma$ is constructive \textit{i.e.} $R_X^\gamma$ is a union of finitely many locally closed subsets of $X \times X$. We will show that constructive is equivalent to having finite local regular strata.

For $M$ a subset of a topological space $Y$, we set: 
$$U(M) = \{ y \in M \ | \ M \textup{ is closed in a neighbourhood of } y \} \textup{ and } M^1 = M \backslash U(M).$$
From \cite[6.7]{bz1}, the set $M^1$ is closed in $M$, in other words, $U(M)$ is open in $M$. Moreover, the set $U(M)$ is a locally closed subset of $Y$. Set $M^0 = M$ and define:
$$M^2 = (M^1)^1, M^3 = (M^2)^1, \dots, M^{k+1} = (M^k)^1.$$
According to \cite[Lem. 6.7]{bz1}, there exists, when $M$ is a non-empty constructive set, $k \in \mathbb{N}^*$ such that $M^k = \emptyset$. Moreover $U(M)$ is dense in $M$ so it is non-empty. Let $k$ be minimal such that $M^{k-1} \neq \emptyset$. We therefore get:
$$M = U(M) \sqcup U(M^1) \sqcup U(M^2) \sqcup \dots \sqcup U(M^{k-1}).$$
Applying it with $M=R_X^\gamma$ when the action is constructive, we easily get that the previous $(G \times G)$-filtration induces an open $G$-filtration of $X$ by projection on the first coordinate via $p_1 : X \times X \to X$ \textit{i.e.} $X_i = p_1(U(M^i))$ and $X = X_0 \sqcup X_1 \sqcup \dots \sqcup X_{k-1}$. We know by \cite[Prop 6.8 b)]{bz1} that $X_i / G$ is locally Hausdorff, therefore constructive implies finite local regular strata. Conversely, an action with finite local regular strata $X = \bigsqcup X_i$ will be constructive as $R_{X_i}^\gamma$ will be locally closed in $X \times X$. Therefore:

\begin{prop} \label{constructive_equiv_to_finite_local_regular_strata_prop} Constructive is equivalent to having finite local regular strata. Furthermore there is a canonical finite open $G$-stratification $X = \bigsqcup X_i$ induced by $U((R_X^\gamma)^i)$ and such that $X_i /G$ is locally Hausdorff. \end{prop}

Actions with finite regular strata can be obtained by considering a regular pair $(\textbf{G},\textbf{X})$ in the sense of \cite{gk}. For $F$ a non-archimedean local field, these pairs consist of an algebraic group $\textbf{G}$ defined over $F$ and an algebraic $F$-manifold $\textbf{X}$ -- \textit{i.e.} a smooth $F$-algebraic variety -- such that the action of $\textbf{G}$ on $\textbf{X}$ is regular. Examples of regular pairs \cite[Ex p.103]{gk} include the conjugation action for classical groups in characteristic $0$. In positive characteristic however, it is more restrictive \textit{e.g.} the conjugation action of $\textup{\textbf{SL}}_n$ is regular if $\textup{char}(F) \nmid n$ and that of $\textup{\textbf{Sp}}_{2n}$ is regular when $\textup{char}(F) \neq 2$.

\subsection{Invariant distributions: the vs. battle continues} \label{invariant_distrib_modular_sec}

\noindent Let $R$ be a field of characteristic $\ell$ and $C_c^\infty(X)$ be the space of $R$-valued functions that are locally constant and compactly supported. A distribution is a linear form $C_c^\infty(X) \to R$. When the context requires it, we may add a reference to the field by denoting $C_c^\infty(X,R)$ and using the word $R$-distributions. When $F$ is a closed subset of $X$ we can restrict the support of functions $\textup{res}_F : C_c^\infty(X) \to C_c^\infty(F)$ to $F$. Let $C_c^\infty(X)[G]$ be the $R$-subspace of $C_c^\infty(X)$ generated by functions of the form $g \cdot f -f$ where $g \in G$ and $f \in C_c^\infty(X)$. The quotient space $C_c^\infty(X) / C_c^\infty(X)[G]$ is called the $G$-coinvariants and is the largest quotient of $C_c^\infty(X)$ on which $G$ acts trivially.

\paragraph{Finite regular strata \`a la Gelfand-Kazhdan.} We can generalise \cite[Th 1]{gk} to the modular setting. We recall the details of the proofs for the reader's convenience.

\begin{lem} \label{stratification_G_conivariants_lemma} Let $X_0 \subset X$ be a closed $G$-stable subset. Then:
\begin{itemize}
\item $\displaystyle \textup{res}_{X_0}^{-1} (C_c^\infty(X_0)[G]) = C_c^\infty(X)[G] + C_c^\infty( X \backslash X_0)$.
\end{itemize} 

\noindent Furthemore, if $X_0 / G$ is Hausdorff (in particular all $G$-orbits $\mathcal{O} \subset X_0$ are closed in $X_0$ and the restriction map $\textup{res}_\mathcal{O} : C_c^\infty(X_0) \to C_c^\infty(\mathcal{O})$ is well-defined) we have:
\begin{itemize} \item $\displaystyle C_c^\infty(X_0)[G] = \bigcap_{\mathcal{O} \in X_0 / G} \textup{res}_{\mathcal{O}}^{-1} (C_c^\infty(\mathcal{O})[G])$. \end{itemize}  \end{lem}

\begin{proof} The inclusion $\supseteq$ holds because $\textup{res}_{X_0} (\sum g \cdot f - f) = \sum g \cdot \textup{res}_{X_0}(f) - \textup{res}_{X_0}(f)$ for $\sum g \cdot f -f \in C_c^\infty(X)[G]$ and $\textup{res}_{X_0} (h) = 0$  for $h \in C_c^\infty(X \backslash X_0)$. To check the opposite inclusion, note that for $\textup{res}_{X_0}(f) = \sum g \cdot h_0 - h_0 \in C_c^\infty(X_0)[G]$, if we lift each $h_0 \in C_c^\infty(X_0)$ in the sum to some arbitrary $h \in C_c^\infty(X)$, we obtain $\textup{res}_{X_0}\big(f - (\sum g \cdot h -h) \big) = 0$. As a result $f \in C_c^\infty(X)[G] + C_c^\infty(X \backslash X_0)$.

Regarding the second claim, as $\textup{res}_{\mathcal{O}}(\sum g \cdot f_0 - f_0) = \sum g \cdot \textup{res}_{\mathcal{O}}(f_0) - \textup{res}_{\mathcal{O}}(f_0)$ for any $G$-orbit $\mathcal{O}$ in $X_0$, the inclusion $\subseteq$ holds. The remainder of the proof focuses on proving the reverse inclusion following ideas of \cite{gk}. 

Let $f \in C_c^\infty(X_0)$ such that $\textup{res}_\mathcal{O}(f) \in C_c^\infty( \mathcal{O})[G]$ for all $G$-orbits $\mathcal{O}$ in $X_0$. For each orbit $\mathcal{O}$, choose a decomposition $\textup{res}_{\mathcal{O}}(f) = \sum g \cdot h_\mathcal{O} - h_\mathcal{O}$. Then choose an arbitrary lift for each $h_\mathcal{O}$, say $h_\mathcal{O}'$, and define:
$$V'(\mathcal{O}) = \{ x \in X_0 \ | \ f(x) \neq \sum g \cdot h_\mathcal{O}'(x) - h_\mathcal{O}'(x) \}.$$
The set $V'(\mathcal{O})$ is closed open compact in $X_0$ and has empty intersection with $\mathcal{O}$. Take the saturated set associated to $V'(\mathcal{O})$ and call it $V(\mathcal{O})$. It is the smallest $G$-stable set containing $V'(\mathcal{O})$, or equivalently, the set $p^{-1} ( p (V'(\mathcal{O})))$ if $p : X_0 \to X_0 / G$ denotes the quotient map. The set $V(\mathcal{O})$ is open because the quotient map $p$ is open. It is closed because $p (V'(\mathcal{O}))$ itself is closed, indeed, it is the image of a compact by a continuous map (here we have made use of the fact that $X_0 / G$ is compact). Moreover it is a proper subset of $X_0$ as it does not contain $\mathcal{O}$. Name the complement $U(\mathcal{O}) = X_0 \backslash V(\mathcal{O})$, which is a proper open closed $G$-stable subset of $X_0$ that contains $\mathcal{O}$.

Considering for all $G$-orbits the sets $U(\mathcal{O})$, one has an open covering of $X_0$. As $f$ has compact support, we choose finitely many such sets that cover the support of $f$. Say this finite cover is associated to a finite family of $G$-orbits $(\mathcal{O}_i)_{i \in I}$ where $I$ is a finite set. We can build a finer covering $(U_k)_{k \in K}$ from $(U(\mathcal{O}_i))_{i \in I}$ satisfying the following: $K$ is a finite set; each $U_k$ is saturated open closed; the $U_k$ are mutually disjoint; there exists a map $\alpha : K \to I$ such that $U_k \subset U(\mathcal{O}_{\alpha(k)})$ for all $k \in K$. Now define $f_k \in C_c^\infty(X_0)$ by:
$$f_k(x) = \left\{ \begin{array}{cc}
\sum g \cdot h_{\mathcal{O}_{\alpha(k)}}'(x) - h_{\mathcal{O}_{\alpha(k)}}'(x) & \textup{ if } x \in U_k; \\
 & \\
0 & \textup{ otherwise.}
\end{array} \right.$$
By definition, one has $f_k = f$ on $U_k$, and each $f_k \in C_c^\infty(X_0)[G]$. Therefore the opposite inclusion holds because $f = \sum f_k \in C_c^\infty(X_0)[G]$.
\end{proof}

\begin{prop} \label{invariance_a_la_GK_prop} Let $\gamma$ be an action of $G$ on $X$ with finite regular strata. Suppose that $\sigma : X \to X$ is a homeomorphism satisfying:
\begin{itemize}
\item if there exists a non-trivial $G$-invariant distribution $T$ on an orbit $S$, then $\sigma (S) = S$ and $\sigma \cdot T = T$.
\end{itemize}
Then any $G$-invariant distribution on $C_c^\infty(X)$ is $\sigma$-invariant.
\end{prop}

\begin{proof} We would like to prove that $C_c^\infty(X)[G] \supseteq C_c^\infty(X)[\sigma]$. To do so, we will consider a finite regular closed $G$-stratification $X = \sqcup_{i=0}^r X_i$ and proceed step by step going up in the stratification. Let $f = \sum \sigma \cdot h - h \in C_c^\infty(X)[\sigma]$ and let $\mathcal{O} \subset X_0$ be a (closed) $G$-orbit. We claim that $\textup{res}_\mathcal{O}(f) \in C_c^\infty(\mathcal{
O})[G]$. Indeed, $\textup{res}_\mathcal{O}(f) = \sum \textup{res}_\mathcal{O}(\sigma \cdot h) - \textup{res}_\mathcal{O}(h)$. There are two cases now. Either there is no local non-trivial $G$-invariant distribituion, \textit{i.e.} $C_c^\infty(\mathcal{O}) = C_c^\infty(\mathcal{O})[G]$, and therefore $\textup{res}_\mathcal{O}(f) \in C_c^\infty(\mathcal{O}) [G]$. Or there does exist a non-trivial $G$-invariant distribution, but the assumption on local distributions implies that $\textup{res}_\mathcal{O}(\sigma \cdot h) = \sigma \cdot\textup{res}_\mathcal{O}(h)$ as $\sigma(\mathcal{O}) = \mathcal{O}$ and $\textup{res}_\mathcal{O}(f) \in C_c^\infty(\mathcal{O})[\sigma] \subset C_c^\infty(\mathcal{O})[G]$ as local $G$-invariant distributions are $\sigma$-invariant. Using Lemma \ref{stratification_G_conivariants_lemma}, this proves that $f \in C_c^\infty(X)[G] + C_c^\infty(X \backslash X_0)$. Now choose $h \in C_c^\infty(X)[G]$ such that $f-h \in C_c^\infty(X \backslash X_0)$ and set $X' = X \backslash X_0$. We can apply the same argument to $X_1$ in $X'$ to get that $f -h \in C_c^\infty(X')[G] + C_c^\infty(X' \backslash X_1)$. In particular $f$ belongs to $C_c^\infty(X)[G] + C_c^\infty(X \backslash X_{\leq 1})$. A finite induction proves that $f \in C_c^\infty(X)[G]$, the last step involving $X \backslash X_{\leq r} = \emptyset$. \end{proof}

\paragraph{Finite local regular strata \`a la Bernstein-Zelevinski.} From Proposition \ref{constructive_equiv_to_finite_local_regular_strata_prop}, having finite local regular strata and being constructive are equivalent. When $G$ is countable at infinity, a homogeneous space $X$ is homeomorphic to a quotient $G/H$  \textit{e.g.} choosing a base point $x \in X$ we can take $H=\textup{Stab}_G(x)$. That way a $G$-invariant distribution on $X$ is simply the quotient Haar measure $\mu_{G\cdot x}$ when it exists \textit{i.e.} when $C_c^\infty(X)[G] \varsubsetneq C_c^\infty(X)$. The existence is also equivalent to the equality of modulus characters ${\delta_G|}_H = \delta_H$. We recall that our coefficient field $R$ has characteristic $\ell$. 

\begin{prop} \label{invariance_a_la_BZ_prop} Suppose that $G$ is countable at infinity. Let $\gamma$ be a constructible action of $G$ on $X$. Suppose that $\sigma : X \to X$ is a homeomorphism satisfying:
\begin{itemize}
\item for all $g \in G$, there exists $g_\sigma \in G$ such that $\gamma(g) \circ \sigma = \sigma \circ \gamma(g_\sigma)$;
\item there exists $n \in \mathbb{N}^*$ and $g_0 \in G$ such that $\sigma^n = \gamma(g_0)$;
\item if there exists a non-zero $G$-invariant distribution $T$ on a $G$-orbit $S$, then $\sigma(S) = S$ and $\sigma \cdot T = T$.
\end{itemize}
If $\ell \nmid  n$, then all $G$-invariant distributions on $X$ are $\sigma$-invariant.
\end{prop}

\begin{proof} If $R$ is algebraically closed, the assumption $\ell \nmid n$ allows us to use the same proof as in \cite[Thm 6.10]{bz1}. If $R$ is not algebraically closed, the compatibility $C_c^\infty(X,R) \otimes_R \bar{R} = C_c^\infty(X,\bar{R})$ for an algebraic closure $\bar{R}$ of $R$ yields the result by scalar extension as a $G$-invariant $R$-distribution on $X$ extends to a $G$-invariant $\bar{R}$-distribution on $X$. \end{proof}

\begin{rem} Strictly speaking, one of the intermediate claims in \cite[Th 6.10]{bz1} is not right as such. Indeed, if $\gamma$ is a constructive action of $G$ on $X$ and $\sigma$ is such that for all $g \in G$, there exists $g_\sigma \in G$ such that $\gamma(g) \circ \sigma = \sigma \circ \gamma(g_\sigma)$, then it is not necessarily true that $\sigma$ rearranges orbits. Take $\mathbb{Z}$ with the discrete topology and $\gamma$ the natural action by translation on itself. Define:
$$\begin{array}{cccl} s : & \mathbb{Z} & \to & \mathbb{Z} \sqcup \mathbb{Z} \\
 & n & \mapsto & \left\{ \begin{array}{cl}
n/2 & \textup{ in the first copy if } n \textup{ is even} \\
(n-1)/2 & \textup{ in the second copy if } n \textup{ is odd}
\end{array} \right. \end{array}$$
Then $\gamma(1) \circ s = s \circ \gamma(2)$. Let $X = \bigsqcup_{\mathbb{N}} \mathbb{Z}$ with the action of $\mathbb{Z}$ on each separate factor and define $\sigma : X \to X$ by sending the copy $i \in \mathbb{N}$ on copies $2i$ and $2i+1$ via $s$. Then $\sigma$ is not orbit-preserving. However, if we were assuming that $g \mapsto g_\sigma$ is a bijection, or $\sigma^{-1}$ satisfies a similar statement -- which holds if $\sigma$ has finite order or some power of $\sigma$ acts through $\gamma$, then $\sigma$ would induce a permutation on orbits.  \end{rem}
,

\paragraph{Comparison of the two methods} We consider and compare Propositions \ref{invariance_a_la_GK_prop} and \ref{invariance_a_la_BZ_prop}. In the first case, we can deduce the $\sigma$-invariance under quite limited assumptions on $\sigma$ and no assumptions at all on $R$ precisely because the action of $G$ on $X$ has finite regular strata. In the second case, the action only admits finite local regular strata, which is more general, but we then have to add more properties on $\sigma$ and $R$ for the result to work. 

As discussed after Proposition \ref{constructive_equiv_to_finite_local_regular_strata_prop}, a regular pair $(\textbf{G},\textbf{X})$ provides an action of $G=\textbf{G}(F)$ on $X=\textbf{X}(F)$ with finite regular strata. The following cases are not covered by these so-called regular pairs and are traditionally considered as constructive actions:
\begin{itemize}
\item $F^\times$ acting on $F^\times$ via $\lambda \cdot x = \lambda^2 x$ when $F$ has characteristic $2$;
\item $SL_2(F)$ acting on itself by conjugation when $F$ has characteristic $2$.
\end{itemize}
Nevertheless, the first action has finite regular strata even though it is coming from an algebraic action $(\textbf{G}_m,\textbf{G}_m)$ that does not form a regular pair. Indeed, it is an elementary exercise to check that the quotient space $X/G$ is Hausdorff. Regarding the second action, it would be interesting to know whether it admits finite regular strata as well -- we would expect it does. Therefore many usual constructive actions (or, equivalently, actions with finite local regular strata) considered in the literature may actually admit finite regular strata. This remark is not so important when the characteristic of $R$ is zero, but it is more meaningful for positive characteristic coefficient fields.

For instance, if one uses finite regular strata actions, it is possible to prove that the transpose-inverse for $\textup{GL}_n(F)$ realises the contragredient on irreducible representations when the coefficient field $R$ has arbitrary characteristic $\ell$. By considering constructive actions, or equivalently, actions with finite local regular strata, one would have to exclude the case $R$ of characteristic $2$.

\subsection{A new approach: orbit-packaging and Matryoshka regularity} \label{Matryoshka_sec}

We can sometimes package orbits in a Hausdorff fashion, even though we do not know whether the action is finite regular or even constructive. This approach will prove to be useful when we deal with the metaplectic and the MVW-involution in Section \ref{MVW_involution_sec}. These techniques may also apply to other covering groups, so we have aimed to present them within the most general framework possible. Before we state the precise definitions, we start by illustrating orbit-packaging techniques for conjugacy classes of covering groups.

\paragraph{Covering groups.} Let $1 \to A \to \tilde{G} \to G \to 1$ be an $A$-cover of a (connected) reductive group $G$. Then the $G$-conjugation action is always constructive. In the previous section, we listed a few groups forming a regular pair and having finite regular strata for the conjugation action.

The following question seems difficult to address: since we know the $G$-conjugation action is constructive, is the $\tilde{G}$-conjugation action constructive as well? Likewise, the same question with finite regular strata seems equally hard. Instead of trying to answer these questions, we propose to construct ``packets'' of orbits in $\tilde{G}$ that are easier to deal with. For instance, a $G$-stratification of $X=G=\sqcup_i X_i$ lifts to a $\tilde{G}$-stratification $\tilde{X} = \tilde{G} = \sqcup_i \tilde{X}_i$ where $\tilde{X}_i$ is the pre-image of $X_i$. However, it seems hard to determine whether $\tilde{X}_i / \tilde{G}$ is (locally) Hausdorff when $X_i/G$ is (locally) Hausdorff, as more orbits are now nested in the pre-image of a $G$-conjugacy class. Nevertheless, if $A$ is finite, we can apply \cite[Prop 1.5]{bz1} to deduce that any $\tilde{G}$-conjugacy class $C_{\tilde{g}}$ is closed in the pre-image $\tilde{C}_g$ of the $G$-conjugacy class $C_g$. As we now explain, this will be enough to deal with invariant distribution.

\paragraph{Matryoshka regular strata.} Given an equivalence relation $\sim$ on $X$, we call $\mathcal{P} \subset X$ a $\sim$-packet if it consists of a single $\sim$-equivalence class and we denote by $X/\sim$ the set of $\sim$-packets in $X$. If $X$ is a topological space, then $X / \sim$ is naturally the quotient space of $X$ and can be endowed with the quotient topology.

\begin{defi} We say the action of $G$ on $X$ is Matryoshka\footnote{the term \textit{Matryoshka} comes from the russian dolls that are nested in one another.} regular if there exists a finite sequence of refined regular equivalence relations $(\sim_i)_{1,\dots,n}$, that is:
\begin{itemize}[label=$\bullet$]
\item $\sim_j$ is a refinement of $\sim_i$ for all $i \leq j$;
\item $\sim_n$-packets are precisely $G$-orbits;
\item $\mathcal{P}_i/\sim_{i+1}$ is Hausdorff for all $\sim_i$-packets $\mathcal{P}_i$.
\end{itemize} \end{defi}

\noindent It is easy to see that $X/G$ is Hausdorff if and only if the action of $G$ on $X$ is Matryoshka regular for $n=1$. These finite sequences of refined regular equivalence relations are just focusing on the behaviour of orbit-packaging on a single stratum. This definition also implies that any $\sim_i$-packet is closed; in particular $G$-orbits must be closed as well. 

\begin{prop} Assume there exists an equivalence relation $\sim$ on $X$ such that the quotient space $X/\sim$ is Hausdorff. Then:
$$C_c^\infty(X)[G] = \bigcap_{\mathcal{O} \in X/\sim} \textup{res}_{\mathcal{O}}^{-1} (C_c^\infty(\mathcal{O})[G]).$$ \end{prop}

The proof is the same as in Lemma \ref{stratification_G_conivariants_lemma} since we only need $X/\sim$ is Hausdorff to build closed open $\sim$-saturated sets. Applying the proposition in stages gives:

\begin{cor} Assume the action of $G$ on $X$ is Matryoshka regular. Then:
$$C_c^\infty(X)[G] = \bigcap_{\mathcal{O} \in X/G} \textup{res}_{\mathcal{O}}^{-1} (C_c^\infty(\mathcal{O})[G]).$$ \end{cor}

Since all proofs are the same as in the previous section, we simply omit them. Finally, we can state a more general version of Proposition \ref{invariance_a_la_GK_prop} in this context:

\begin{prop} \label{invariance_Matryoshka_stratif_a_la_GK_prop} Let $\gamma$ be an action of $G$ on $X$ with finite Matryoshka regular strata. Suppose that $\sigma : X \to X$ is a homeomorphism satisfying:
\begin{itemize}
\item if there exists a non-trivial $G$-invariant distribution $T$ on a packet $S$, then $\sigma (S) = S$ and $\sigma \cdot T = T$.
\end{itemize}
Then any $G$-invariant distribution on $C_c^\infty(X)$ is $\sigma$-invariant.
\end{prop}

\section{Preliminaries for the theta correspondence}

\subsection{The metaplectic group and the Weil representation}

\paragraph{Metaplectic group} The cohomology group $H^2(\textup{Sp}(W),\{\pm 1 \})$ has order $2$, it parametrises the cohomology classses of topological $2$-cocyles $\textup{Sp}(W) \times \textup{Sp}(W) \to \{ \pm 1 \}$. Therefore, the isomorphism classes of topological central extensions of $\textup{Sp}(W)$ by $\{ \pm 1 \}$ solely consists in two kinds: the trivial and the non-trivial.

Let $X$ be a Lagrangian in $W$. There is an explicit $\{ \pm 1 \}$-valued $2$-cocycle $c_X$ whose cohomology class is non-trivial \cite{rao,trias_theta1}. Composing with $\{ \pm 1 \} \to R^\times$, we get a $2$-cocyle still denoted $c_X$. Note that $\{ \pm 1 \} \to R^\times$ may fail to be injective! Indeed, when the characteristic of $R$ is $2$, the $2$-cocycle $c_X$ actually becomes the trivial $2$-cocycle.

Define $\textup{Mp}^{c_X}(W) = \textup{Sp}(W) \times_{c_X} R^\times$ the associated group. It is a locally profinite group. However, its topology is not the product topology. Nevertheless, the topology near the identity $(\textup{id}_W,1)$ is the product topology. Indeed $c_X$ is not continuous on $\textup{Sp}(W) \times \textup{Sp}(W)$, but it is continuous in a neighbourhood of $(\textup{id}_W,\textup{id}_W)$. 

\paragraph{Weil representation} Let $H= W \times F$ be the Heisenberg group where the group operation is $(w,t) \cdot (w',t') = (w+w',t+t'+\frac{1}{2}\langle w ,w' \rangle)$. In \cite{trias_theta1}, we defined the Heisenberg representation of $H$ with coefficients in $R$ in the following way: choose a Lagrangian $X$ in $W$ and set $S_{\psi,X} = \textup{ind}_{X \times F}^{W \times F}(\psi_X)$ where $\psi_X$ extends $\psi$ trivially to $X$ \textit{i.e.} $\psi_X(x,t) = \psi(t)$. Because the natural action of $\textup{Sp}(W)$ on $H$ preserves the centre, this action preserves the central character $\psi$ of $S_{\psi,X}$ and the Stone-von Neumann Theorem provides a projective representation $\sigma_{\psi,X} : \textup{Sp}(W) \to \textup{PGL}(S_{\psi,X})$ that can be lifted to a representation of a central extension of $\textup{Sp}(W)$ by $R^\times$. This central extension is canonically isomorphic to the $R$-metaplectic group $\textup{Mp}^{c_X}(W) = \textup{Sp}(W) \times_{c_X} R^\times$. In \cite{trias_theta1}, we defined the Schr\"odinger model of the Weil representation: 
$$\omega_{\psi,X} : (g,\lambda) \in \textup{Mp}^{c_X}(W) \to \textup{GL}(S_X).$$
Note that the construction in \cite{trias_theta1} actually does not depend on a choice of a symplectic basis $\mathcal{B}_X = \{ e_1, \dots, e_m, f_m, \dots, f_1 \}$ where $\{e_1, \dots e_m\}$ generates $X$, so the construction of the Weil representation only depends on $X$ and $\psi$ as the notation suggests. Classically, which means over the complex numbers, the Weil representation is constructed as a unitary representation. In the modular setting, these considerations do not make sense, so our definition of $\omega_{\psi,X}$ is slightly different: it is more direct but relies heavily on new quantities introduced in \cite{trias_theta1}, such as the non-normalised Weil factor.

\paragraph{Making $\omega_{\psi,X}$ explicit.} We make the definition of the Weil representation more explicit. The action of $g \in \textup{Sp}(W)$ on $H$ induces a morphism $S_{\psi,X} \to S_{\psi,g \cdot X}$. Now $S_{\psi,g \cdot X}$ and $S_{\psi,X}$ have the same central character, so the Stone-von Neumann Theorem ensures there is an intertwining operator $I_{X,gX,\mu_g} : S_{\psi,g \cdot X} \to S_{\psi,X}$ which is unique up to scalar -- this scalar dependence is contained in the choice of Haar measure $\mu_g$ on $(gX \cap X) \backslash X$. Then:
$$\omega_{\psi,X}(g,1) = I_{X,gX,\mu_g} \circ I_g$$
for a very specific choice of $\mu_g$ for each $g \in \textup{Sp}(W)$. The main technical difficulty in defining the Weil representation lies in making this choice right or explicit. We won't recall the explicit construction of $\mu_g$ but rather give $\omega_{\psi,X}$ on some specific elements generating $\textup{Sp}(W)$.

First of all, for $p \in P(X)$, the stabiliser of $X$ in $\textup{Sp}(W)$, we have:
$$\omega_{\psi,X}(p,1) \cdot f =  \Omega_{1,\textup{det}_X(p)} \times (I_p \cdot f)$$
where $\Omega_{1,\textup{det}_X(p)} \in R^\times$ is the non-normalised Weil factor and $(I_p \cdot f)(w,1) = f(p^{-1}w,1)$.
Next, as $\omega_{\psi,X}$ is a group morphism, for all $p_1, p_2 \in P(X)$ and $g \in \textup{Sp}(W)$ we have:
$$\omega_{\psi,X}(p_1,1) \omega_{\psi,X}(g,1) \omega_{\psi,X}(p_2,1) = c_X(p_1,g) c_X(p_1 g,p_2) \times  \omega_{\psi,X}(p_1 g p_2,1).$$
It only remains to describe $\omega_{\psi,X}$ on representatives of double cosets $P(X) \backslash \textup{Sp}(W) / P(X)$. Pick a symplectic basis $\mathcal{B}_X$. Then for a subset $S \subset \{1, \dots, m\}$ we have an element $w_S \in \textup{Sp}(W)$ and a choice of measure $\mu_{w_S}$ described in \cite{trias_theta1}. It turns out this choice of measure does not depend on the choice of the basis $\mathcal{B}_X$ \textit{i.e.} the measure $\mu_{w_S}'$ for another choice of basis $\mathcal{B}_X'$ is equal to $\mu_{w_S}$. We then set:
$$\omega_{\psi,X}(w_S,1) = I_{X,w_S X,\mu_{w_S}} \circ I_{w_S}.$$

\paragraph{Changing the Lagrangian.} If we choose two Lagrangians $X$ and $X'$ in $W$, the groups $\textup{Mp}^{c_X}(W)$ and $\textup{Mp}^{c_{X'}}(W)$ are canonically isomorphic as central extensions, so there exists a unique map $\gamma = \gamma_{X',X} : \textup{Mp}^{c_X}(W) \to \textup{Mp}^{c_{X'}}(W)$ written $(g,\lambda) \mapsto (g,\lambda \gamma_g)$, where:
$$\partial \gamma \cdot c_{X'} = c_X \textit{ i.e. } \gamma_g \gamma_{g'} c_{X'}(g,g') = \gamma_{g g'} c_X(g,g').$$
Moreover such a $g \mapsto \gamma_g$ is $\{\pm 1\}$-valued and this canonical isomorphism gives rise to a commutative diagram:
$$\xymatrix@R+1pc@C+1pc{
\textup{Mp}^{c_X}(W) \ar@{->}[r]^{\gamma_{X',X}} \ar@{->}[d]^{\omega_{\psi,X}} &  \textup{Mp}^{c_{X'}}(W) \ar@{->}[d]^{\omega_{\psi,X'}} \\
\textup{GL}(S_{\psi,X}) \ar@{->}[r]^{\mathfrak{conj}_{X',X}} & \textup{GL}(S_{\psi,X'})
}$$
where $\mathfrak{conj}_{X',X}$ is the conjugation $M \mapsto I_{X',X} \circ M \circ (I_{X',X})^{-1}$ by any intertwining operator $I_{X',X} : S_{\psi,X} \overset{\sim}{\to} S_{\psi,X'}$. Note that $\mathfrak{conj}_{X',X}$ is well-defined since these operators are unique up to scalars. Therefore, for $g \in \textup{Sp}(W)$ the commutativity of the diagram reads:
$$I_{X',X} \circ \omega_{\psi,X'}(g,1) \circ (I_{X',X})^{-1} = \gamma_g \times \omega_{\psi,X}(g,1).$$

\subsection{$\epsilon$-hermitian spaces and dual pairs} \label{hermitian_spaces_sec}

\paragraph{$\epsilon$-hermitian spaces.} Let $D$ be a division algebra of finite dimension over its centre and endowed with an involution $\tau$ \textit{i.e.} an anti-automorphism of $D$. Let $V$ be a right $D$-vector space of finite dimension. Let $\epsilon$ in the centre of $D$ be such that $\epsilon \tau(\epsilon) = 1$. An $\epsilon$-hermitian product $\langle \ , \ \rangle$ is a non-degenerate sesquilinear form $V \times V \to D$, \textit{i.e.} $\langle v d , v' d' \rangle = \tau(d) \langle v,v' \rangle d$, satisfiying $\langle v',v \rangle = \epsilon \ \tau( \langle v,v' \rangle)$. Then $(V,\langle \ , \ \rangle)$, or simply $V$ for short, is an $\epsilon$-hermitian space. Two $\epsilon$-hermitian spaces $V$ and $V'$ are isometric if there exists an isomorphism of right $D$-vector spaces $\varphi : V \to V'$ preserving the hermitian products \textit{i.e.} $\langle \varphi(v) , \varphi(v') \rangle_{V'} =  \langle v,v' \rangle_V$. We write $V \simeq V'$. An element $g \in \textup{GL}_D(V)$ is called an isometry if it makes $V$ isometric to itself. The group of isometries of $V$, also called the isometry group, is denoted by $U(V)$. When $D$ is commutative and $\tau = \textup{id}_D$, the space $V$ is orthogonal when $\epsilon=1$ and symplectic when $\epsilon=-1$. These two isometry groups are usually denoted by $O(V)$ and $\textup{Sp}(V)$.

We say a vector subspace $X$ in $V$ is totally isotropic if the restriction of the $\epsilon$-hermitian product to $X \times X$ is zero. A vector $v \in V$ is isotropic if $\langle v , v \rangle = 0$. Equivalently, the $D$-subspace generated by $v$ is totally isotropic. We say $V$ is anisotropic if it does not contain any non-trivial isotropic vector. The $\epsilon$-hermitian space $V$ is said to be split if there exist two totally isotropic subspaces $X$ and $Y$ in $V$ such that $V = X \oplus Y$. In this case, the totally isotropic subspaces $X$ and $Y$ have the same dimension, that is half the dimension of $V$. We call a subspace such as $X$ or $Y$ a Lagrangian and the decomposition $V=X \oplus Y$ a complete polarisation of $V$. Therefore, we allow ourselves to talk about Lagrangians only when $V$ is split. Remark that for split spaces, Lagrangians are exactly maximal totally isotropic subspaces. For arbitrary $V$, we can still consider maximal totally isotropic subspaces and Witt's Theorem ensures they all have the same dimension, which is called the Witt index of $V$.

For $W$ a $D$-vector subspace of $V$, one can define the orthogonal of $W$ in $V$ by $W^\perp = \{ v \in V \ | \ \langle v , W \rangle = 0\}$. An $\epsilon$-hermitian subspace $W$ of $V$ is a $D$-vector subspace such that the restriction of $\langle \ , \ \rangle$ to $W$ is non-degenerate. In this case, we have an orthogonal decomposition $V = W \oplus W^\perp$.

The $\epsilon$-hermitian plane $\mathbb{H} = D \oplus D$ is the split $\epsilon$-hermitian space of dimension $2$ associated to the $\epsilon$-hermitian product: 
$$\langle (d_1,d_2) , (d_1',d_2') \rangle = \tau(d_1) d_2' + \epsilon \tau(d_2) d_1'.$$
It is easy to see that $\mathbb{H} = (D,0) \oplus (0,D)$ is a complete polarisation of $\mathbb{H}$. We recall the following classical result:

\begin{lem} Assume $V$ has a non-zero isotropic vector $v$. Then there exists $v' \in V$ such that $W = v D \oplus v' D \simeq \mathbb{H}$ and $V = W \oplus W^\perp$. \end{lem}

\paragraph{Type I dual pairs.} Let $D$ be a division algebra of finite dimension over its centre $E$ and assume $F \subset E$. Assume $d=\textup{dim}_F(D)$ divides $4$ and we are in one of the three cases:
\begin{itemize}[label=$\bullet$]
\item $d=1$, $D=E=F$ and $\tau = \textup{id}_D$;
\item $d=2$, $D=E$ is quadratic over $F$ and $\tau$ generates $\textup{Gal}(E/F)$;
\item $d=4$, $D$ is quaternionic over $E=F$ and $\tau$ is the canonical\footnote{$\tau$ is unique up to inner automorphisms of $D$.} involution.
\end{itemize}
Of course, when $D$ is commutative, left vector spaces over $D$ are also right vector spaces, and conversely. Let $V_1$ be an $\epsilon_1$-hermitian space of dimension $n_1$ over $D$ and let $V_2$ be an $\epsilon_2$-hermitian of dimension $n_2$. We define the tensor product $W = V_1 \otimes_D V_2$ where we consider $V_2$ as a left $D$-vector space via $d \cdot v_2 = v_2 \tau(d)$, so it is actually changing $V_2$ not only in the quaternionic case but also in the quadratic one. Assume $\epsilon_1 \epsilon_2=-1$. Then $W$ is an $F$-vector space of dimension $d n_1 n_2$ and we endow it with the symplectic product induced by:
$$\langle v_1 \otimes v_2 , v_1' \otimes v_2' \rangle = \textup{tr}_{D/F} \big( \langle v_1 ,v_1' \rangle_1 \times \tau(\langle v_2 , v_2' \rangle_2) \big).$$
We can naturally embed $U(V_1)$ and $U(V_2)$ as subgroups of $\textup{Sp}(W)$ via:
$$u_1 \mapsto u_1 \otimes_D \textup{id}_{V_2} \textup{ and } u_2 \mapsto \textup{id}_{V_1} \otimes _D u_2.$$
These are mutual centralisers in $\textup{Sp}(W)$, meaning:
$$C_{\textup{Sp}(W)}(U(V_1)) = U(V_2) \textup{ and } C_{\textup{Sp}(W)}(U(V_2)) = U(V_1),$$
and we say $(U(V_1),U(V_2))$ is a type I dual pair in $\textup{Sp}(W)$.

\begin{rem} There is a more general definition for type I reductive dual pairs. The ones we have just given are exactly the irreducible ones, and all type I reductive dual pairs are obtained as products of irreducible ones. In the theta correspondence, having Howe duality for irreducible pairs should be sufficient to build up the general case. \end{rem}

\paragraph{Modulus character.} We write $\textup{det}_F(x)$ for the determinant of $x \in D$ over $F$.

If $d \in \{1 , 2 \}$, then $D=E$ and we have $\textup{det}_F(x) = \textup{Nrd}_{D/F}(x)$ where $\textup{Nrd}_{D/F}$ is the (reduced) norm associated to the extension $E/F$. For $a \in \textup{GL}_n(E)$, we write $\textup{det}_F(a)$ (resp. $\textup{Nrd}_{D/F}(a)$) to mean the composition of the determinant $\textup{GL}_k(E) \to E^\times$ with $\textup{det}_F$ (resp. $\textup{Nrd}_{D/F}$). Therefore $\textup{det}_F = \textup{Nrd}_{D/F}$ in this situation

In the quaternionic case, we have $\textup{det}_F(x) = \textup{Nrd}_{D/F}(x)^2$ where $\textup{Nrd}_{D/F}$ is the reduced norm of the central simple algebra $D/F$. Since $F$ has not characteristic $2$, we have $[D^\times,D^\times] = D^1$ where $D^1$ is the kernel of $\textup{Nrd}_D$ by \cite[Chap I, Prop 3.5]{vig_quat}. Therefore, if $a \in \textup{GL}_k(D)$, we write $\textup{det}_F(a)$ (resp. $\textup{Nrd}_{D/F}(a)$) to mean the composition of the determinant $\textup{GL}_k(D) \to D^\times/[D^\times,D^\times] = D^\times / D^1$ with $\textup{det}_F$ (resp. $\textup{Nrd}_{D/F}$). Unlike the previous situation, we obtain $\textup{det}_F = \textup{Nrd}_{D/F}^2$ in the quaternionic case.

If $X$ is a vector space of finite dimension over $D$, we can easily make sense of $\textup{det}_F(a)$ for $a \in \textup{GL}_D(X)$. Let $V$ is an $\epsilon$-hermitian space of dimension $n$ over $D$ and $X$ a totally isotropic subspace of $V$ of dimension $k$. We give the modulus character of the stabiliser $P(X)$ of $X$ in $U(V)$. For $p \in P(X)$ and $a = p|_X \in \textup{GL}_D(X)$, we set $\textup{det}_X(p) = \textup{det}_F(a)$.

\begin{prop} \label{modulus_character_parabolic_prop} The modulus of $P(X)$ is:
\begin{itemize}
\item[] \textup{[$D=F$]} \ $\delta_N(p) = \textup{Nrd}_{D/F}(a)^{n - k - \epsilon}$;
\item[] \textup{[$D=E$]} \ $\delta_N(p) = \textup{Nrd}_{D/F}(a)^{n - k}$;
\item[] \textup{[$D$ quat]} \ $\delta_N(p) = \textup{Nrd}_{D/F}(a)^{2n - 2k + \epsilon}$.
\end{itemize}
We can sum up all cases setting $\eta = \epsilon, 0, - \frac{\epsilon}{2}$ to obtain $\delta_N(p) = \textup{det}_X(p)^{n-k-\eta}$. \end{prop}

\subsection{Lifts of dual pairs} \label{lifts_of_dual_paris_sec}

Let $\textup{Mp}(W) = \textup{Sp}(W) \times_c R^\times$ be the $R$-metaplectic group, where we recall the cocycle $c$ takes values in $\{ \pm 1 \}$. For $G$ a subgroup in $\textup{Sp}(W)$, consider $\textbf{G}$ the inverse image of $G$ in $\textup{Mp}(W)$. We can see $\textbf{G} = G \times_c R^\times$ by restricting the cocyle $c$ in an obvious way. If $H$ is a subgroup of $\textup{Sp}(W)$ commuting with $G$, then if $\textbf{GH}$ is the inverse image of $GH$ in $\textup{Mp}(W)$ we have a natural group morphism into the metaplectic group $\textbf{G} \times \textbf{H} \to \textbf{GH}$.

\begin{rem} The group morphism $\textbf{G} \times \textbf{H} \to \textbf{GH}$ can be written in several forms since it is $((g,\lambda),(h,\mu)) \mapsto (g,\lambda) \cdot (h,\mu) = (gh,\lambda \mu \hspace{1mm} c(g,h))=(hg,\lambda \mu \hspace{1mm} c(h,g))$. \end{rem}

\paragraph{Parabolics.} If $\Phi$ is a totally isotropic flag in $W$, its stabiliser $P(\Phi)$ in $\textup{Sp}(W)$ lifts to the covering group $P_\Phi = P(\Phi) \times_c R^\times$ in $\textup{Mp}(W)$ and there exists a unique splitting $n \in N(\Phi) \to (n,\gamma_n) \in \textup{Mp}(W)$ of the unipotent radical $N(\Phi)$ of $P(\Phi)$ normalised by $P(\Phi)$ \textit{i.e.} for all $p \in P(\Phi)$ and all $n \in N(\Phi)$, we have $(p,1)^{-1} (n,\gamma_n) (p,1) = (p^{-1} n p , \gamma_{p^{-1} n p})$. The mixed Schr\"odinger model shows such a splitting always exist. It is unique because two splittings $\gamma$ and $\gamma'$ differ by a character $\chi$ where $\chi(p^{-1} n p) = \chi(n)$ for all $n \in N(\Phi)$ and all $p \in P(\Phi)$, which implies $\chi$ is the trivial character \textit{i.e.} the uniqueness of the splitting. By choosing a dual flag $\Phi^\vee$ of $\Phi$, we have a Levi decomposition $P(\Phi) = M(\Phi) \ltimes N(\Phi)$ and an isomorphism $((m,\lambda),n) \in M_\Phi \ltimes N(\Phi) \overset{\sim}{\to} (m,\lambda) \cdot (n,\gamma_n) \in P_\Phi$. If $c$ and $\Phi$ are chosen nicely, in the sense that there exists a Lagrangian $X$ in $W$ such that $c=c_X$ and $P(\Phi) \subset P(X)$, then the splitting of $N(\Phi)$ is $n \mapsto (n,1)$. If we decompose $M(\Phi)$ further as $G(\Phi) \times U(\Phi)$ according to its GL-part and its isometric part, we can define the lifts $G_\Phi$ and $M_\Phi$ as well as the surjective group morphism $G_\Phi \times U_\Phi \twoheadrightarrow M_\Phi$ via $((g,\lambda),(u,\mu)) \mapsto (g,\lambda) \cdot (u,\mu)$. Therefore, the full Levi decomposition reads as the surjection $(G_\Phi \times U_\Phi) \ltimes N(\Phi) \twoheadrightarrow P_\Phi$.

\paragraph{Lifts of dual pairs.} When $(U(V_1),U(V_2))$ is a dual pair, the two groups are naturally subgroups of $\textup{Sp}(W)$ and they lift to a dual pair in $\textup{Mp}(W)$. We set $H_1 = U(V_1) \times_{c_1} R^\times$ and $H_2 = U(V_2) \times_{c_2} R^\times$ where $c_1$ and $c_2$ are obtained from $c$ by restriction. We have a group morphism:
$$\begin{array}{ccc}
 H_1 \times H_2 & \to & \textup{Mp}(W) \\
 ((u_1,\lambda_1),(u_2,\lambda_2)) & \mapsto & (u_1 \otimes_D \textup{id}_{V_2},\lambda_1) \cdot (\textup{id}_{V_1} \otimes_D u_2,\lambda_2) \end{array}$$
For $\Phi_1$ a totally isotropic flag in $V_1$, consider the parabolic subgroup $P(\Phi_1)$ in $U(V_1)$ stabilising it. We can create in an obvious way a totally isotropic flag $\Phi_1 \otimes_D V_2$ in $W$ out of $\Phi_1$. Then the splitting $n \in N(\Phi_1 \otimes_D V_2) \mapsto (n,\gamma_n) \in \textup{Mp}(W)$ induces, by setting $\gamma_{n_1} = \gamma_{n_1 \otimes_D \textup{id}_{V_2}}$, a splitting $n_1 \mapsto (n_1,\gamma_{n_1})$ of $N(\Phi_1)$ in $P_{\phi_1} = P(\Phi_1) \times_{c_1} R^\times$. Similarly, if $\Phi_1^\vee$ is a dual flag, then $P_{\Phi_1} = M_{\Phi_1} \ltimes N(\Phi_1)$ is compatible to the Levi decomposition of $P_{\Phi_1 \otimes_D V_2}$ previously obtained \textit{i.e.} $(m,\lambda) \in M_{\Phi_1} \mapsto (m \otimes_D \textup{id}_{V_2},\lambda) \in M_{\Phi_1 \otimes_D V_2}$ and we can push it even further to $(G_{\Phi_1} \times U_{\Phi_1} )\ltimes N(\Phi_1) \twoheadrightarrow P_{\Phi_1}$. The same applies for $V_2$.

\paragraph{Splitting.} The covering groups $H_1$ and $H_2$ are always split over $U(V_1)$ and $U(V_2)$ \textit{i.e.} there exist group sections $U(V_1) \to H_1$ and $U(V_2) \to H_2$ of the natural projections, except if one of $V_1$ or $V_2$ is odd orthogonal. In the latter case, for instance if $V_1$ is odd orthogonal, then $H_1$ is split over $U(V_1)=O(V_1)$ but $H_2 = \textup{Mp}(V_2)$ is the metaplectic group, which is not split over $U(V_2) = \textup{Sp}(V_2)$.

\subsection{The MVW-involution} \label{MVW_involution_sec}

Let $G$ be a locally profinite group and assume there exists a compact open subgroup of $G$ of invertible pro-order in $R$. We suppose $G$ is unimodular \textit{i.e.} any left Haar measure on $G$ is also right invariant. Let $(\pi,V) \in \textup{Irr}_R(G)$ be irreducible. We define the trace-character associated to $\pi$ which is a distribution $\textup{tr}_\pi : C_c^\infty(G) \to R$ defined in the following way. For $f \in C_c^\infty(G)$, there exists $K$ open compact subgroup such that $f$ is bi-$K$-invariant. Then the endomorphism $\pi(f) = \int_G f(g) \pi(g) d \mu(g)$ of $V$ factors through an endomorphism $V^K \to V^K$, still denoted $\pi(f)$, and by admissibility of $\pi$, we can take its trace $\textup{tr}_\pi(f) \in R$, which does not depend on the choice of $K$. Moreover, the trace-character is invariant by conjugation \textit{i.e.} if we denote by $\rho_l$ and $\rho_r$ the natural $G$-actions, we have $\textup{tr}_\pi(\rho_r(g) \rho_l(g) f) = \textup{tr}_\pi(f)$ for all $f \in C_c^\infty(G)$ and all $g \in G$. We now denote by $\gamma$ the conjugation action.

Let $(V,\langle \ , \ \rangle)$ be an $\epsilon$-hermitian space over $D=E$ \textit{i.e.} $U(V)$ is symplectic, orthogonal or unitary, but not unitary quaternionic. We fix $\delta \in \textup{GL}_F(V)$ such that $\langle \delta v , \delta v' \rangle = \langle v' , v \rangle$ for all $v, v'$ in $V$. Such an element $\delta$ exists by virtue of \cite[Chap 4, I.2]{mvw} and $\delta$ is necessarily $\tau$-linear. Moreover $g \mapsto \delta g \delta^{-1}$ defines an automorphism of $U(V)$ as well as a covariant endofunctor $\pi \mapsto \pi^\delta$ of $\textup{Rep}_R(U(V))$. Recall \cite[Chap 4, I.2 Prop]{mvw}:


\begin{lem} \label{conjugacy_g_g_inverse_delta_lem} For all $g \in U(V)$, the elements $\delta g \delta^{-1}$ and $g^{-1}$ are conjugate in $U(V)$. \end{lem}

In the context of regular actions \`a la Gelfand-Kazhdan appearing in Section \ref{on_invariant_distributions_sec}, we have a very clean result. Recall that the conjugation action for isometry groups always has finite regular strata when the characteristic of $F$ is zero. When $F$ has positive (odd) characteristic, the action has finite regular strata in the symplectic case \cite[Ex p.103]{gk}. We do not try to determine whether other isometry groups have also this property.

\begin{theo} \label{MVW_involution_regular_action_thm} Let $R$ be a field of characteristic $\ell \neq p$. Assume the conjugation action of $U(W)$ on itself has finite regular strata. Then for all $\pi \in \textup{Irr}_R(U(W))$, one has $\pi^\delta \simeq \pi^\vee$. \end{theo}

\begin{proof} Recall all irreducible representations of $U(V)$ are admissible, so considering the trace-character makes sense. Two irreducible representations are isomorphic if and only if they have the same trace-character \cite[Th 1.1.3.7]{trias_thesis}. We want to prove $\textup{tr}_{\pi^\vee} = \textup{tr}_{\pi^\delta}$. First, because $U(V)$ is unimodular, we have $\textup{tr}_{\pi^\vee}(f) = \textup{tr}_\pi(f^\vee)$ where $f^\vee(g) = f(g^{-1})$ and $\textup{tr}_{\pi^\delta}(f) = \textup{tr}_\pi(f^\delta)$. Therefore we should prove that $\textup{tr}_\pi((f^\vee)^\delta) = \textup{tr}_\pi(f)$, or in other words, the distribution $\textup{tr}_\pi$ is invariant under $\sigma : g \mapsto \delta g^{-1} \delta^{-1}$. To do so, we will use Proposition \ref{invariance_a_la_GK_prop}, so we now check this proposition applies. The map $\sigma$ preserves conjugacy classes thanks to Lemma \ref{conjugacy_g_g_inverse_delta_lem} and it preserves local $U(V)$-invariant distributions: such an invariant distribution on an orbit $S = U(V) \cdot g$ is simply a quotient Haar measure of $U(V)$ by the stabiliser -- or centraliser, here -- of $g$, and $\sigma^2$ is the conjugation by $\delta^2 \in U(V)$, so the modulus of $\sigma$ is $1$. Therefore the proposition applies and $\textup{tr}_\pi$ is $\sigma$-invariant \textit{i.e.} $(\pi^\vee)^\delta \simeq \pi$ or equivalently $\pi^\vee \simeq \pi^\delta$. \end{proof}

The $U(V)$-conjugation is always a constructive action \`a la Bersntein-Zelevinski. Again, these actions appeared in Section \ref{on_invariant_distributions_sec}. In the proof of Theorem \ref{MVW_involution_regular_action_thm}, we can replace the use of Proposition \ref{invariance_a_la_GK_prop} by Proposition \ref{invariance_a_la_BZ_prop} to obtain a similar result for all isometry groups, at the cost of a small restriction on the characteristic of the coefficient field:

\begin{theo} \label{MVW_involution_constructive_action_thm} Let $R$ be a field of characteristic $\ell \nmid 2  p$. Then for all $\pi \in \textup{Irr}_R(U(W))$, one has $\pi^\delta \simeq \pi^\vee$. \end{theo}

Nevertheless, the situtation for the metaplectic group will be similar to the regular action picture. We first have to explain what the conjugation by $\delta$ means in this context. If $\widehat{\textup{Sp}}(W) = \textup{Sp}(W) \times_c \{ \pm 1 \}$ is the metaplectic group, then there exists a unique lift of the conjugation by $\delta$ on $\textup{Sp}(W)$ \textit{i.e.} there exists a unique $\gamma_\delta : \textup{Sp}(W) \to \{ \pm 1 \}$ such that:
$$\delta : (g, \epsilon) \mapsto (\delta g \delta^{-1},\epsilon \gamma_\delta(g))$$
is an automorphism of central extensions of $\widehat{\textup{Sp}}(W)$. We still call it the conjugation by $\delta$ and alternatively denote $(\delta g \delta^{-1},\epsilon \gamma_\delta(g))$ by $\delta (g,\epsilon) \delta^{-1}$.

We want to be able to use Jordan decomposition for elements of $\textup{Sp}(W)$, therefore we restrict ourselves to the separable case when $F$ is a function field.

\begin{theo} \label{MVW_involution_metaplectic_thm} Let $R$ be a field of characteristic $\ell \neq p$. If $\textup{char}(F) = p > 0$, we assume that $p > \textup{dim}(W)/2$. Then for all $\pi \in \textup{Irr}_R(\widehat{\textup{Sp}}(W))$, one has $\pi^\delta \simeq \pi^\vee$. \end{theo}

\begin{proof} We recall briefly why all irreducible representations are admissibile in the context of connected reductive groups. This is a consequence of the existence of cuspidal support \cite[II.2.4]{vig_book} and the fact that parabolic induction preserves admissibility \cite[II.2.1]{vig_book}, so it is enough to justify it for cuspidals. This fact is known for cuspidals as they are compact representations \cite[II.2.7]{vig_book}. Since we obtained in  Section \ref{lifts_of_dual_paris_sec} a theory of parabolic subgroups for the metaplectic group, all these facts apply to the metaplectic group equally.

Unfortunately, we do not know whether the conjugation action of the metaplectic group on itself has finite regular strata. However, the best we can prove in this direction is obtained by lifting a finite regular $\textup{Sp}(W)$-stratification $\textup{Sp}(W) = \bigsqcup X_i$ to a finite $\widehat{\textup{Sp}}(W)$-stratification $\widehat{\textup{Sp}}(W) = \bigsqcup \widehat{X}_i$ and noticing that:
\begin{itemize}[$\bullet$]
\item the inverse image $\widehat{C}_g$ of a conjugacy class $C_g$ consists of two orbits $C_{(g,1)}$ and $C_{(g,-1)}$ because $(u,1)$ commutes with $(g,1)$ when $u$ commutes with $g$;
\item the points $C_{(g,1)}$ and $C_{(g,-1)}$ are closed in $\widehat{X}_i/\widehat{\textup{Sp}}(W)$ because $\widehat{C}_g$ is locally profinite and it has two homeomorphic orbits \cite[Prop 1.5]{bz1}.
\end{itemize}
The fact that points are closed is necessary for $\widehat{X}_i/\widehat{\textup{Sp}}(W)$ to be Hausdorff, but not sufficient. For instance, the line with a double point is locally Hausdorff not Hausdorff, so all its points are closed.  Nevertheless, if we closely examine the proof of Lemma \ref{stratification_G_conivariants_lemma}, we had a quotient map $p$ and we can replace it here by $\widehat{X}_0 \to X_0/\textup{Sp}(W)$ to obtain the same conclusion as in the lemma (we still need to know that each orbit in $\widehat{X}_0$ is closed, but this is fine as we have just discussed). 

For the metaplectic group, each stratum $\widehat{X}_i$ is Matryoshka regular, in the sense of Section \ref{Matryoshka_sec}. Indeed, let $\sim_2$ be the equivalence relation given by the $\widehat{\textup{Sp}}(W)$-orbit relation and $\sim_1$ be the projection to the $\textup{Sp}(W)$-orbit relation \textit{i.e.} $\widehat{X}_i/\sim_2 = \bigsqcup_{(g,\epsilon)} C_{(g,\epsilon)}$ and $\widehat{X}_i/\sim_1 = \bigsqcup_g \widehat{C}_g$. As $\widehat{X}_i / \sim_1 = X_i / \textup{Sp}(W)$, it is Hausdorff and, for all $g \in \textup{Sp}(W)$:
$$\widehat{C}_g / \sim_2 = \widehat{C}_g / \widehat{\textup{Sp}}(W) = C_{(g,1)} \sqcup C_{(g,-1)}$$
is Hausdorff as well (it is even discrete). So $\widehat{X}_i$ is Matryoshka regular.

Let $\sigma : \widehat{\textup{Sp}}(W) \to \widehat{\textup{Sp}}(W)$ be the homeomorphism defined by $\sigma(\widehat{g}) = \delta \widehat{g}^{-1} \delta^{-1}$. Because the metaplectic group is unimodular, the proof will go the same way as in Theorem \ref{MVW_involution_regular_action_thm}. In other words, we will prove the $\sigma$-invariance of $\textup{tr}_\pi$ using Proposition \ref{invariance_Matryoshka_stratif_a_la_GK_prop}. When $\widehat{g} = (g,\epsilon)$, it is known that $\sigma$ preserves semisimple conjugacy classes \textit{i.e.} $\sigma(C_{(g,\epsilon)})=C_{(g,\epsilon)}$ when $g$ is semisimple  \cite[Chap 4, I.8 Prop]{mvw}. We first extend this result:

\begin{lem} For $\widehat{g} \in \widehat{\textup{Sp}}(W)$, the elements $\delta \widehat{g} \delta^{-1}$ and $\hat{g}^{-1}$ are conjugate in $\widehat{\textup{Sp}}(W)$. \end{lem}

\begin{proof} We use an induction argument to generalise \cite[Chap 4, I.8 Prop]{mvw}.

First of all, the explicit formulas for the metaplectic cocycle \cite[Chap 4, I.11]{mvw}, which we recall below, easily prove the lemma when $W$ has dimension $n=2$. Indeed, as $\textup{Sp}(W) \simeq \textup{SL}_2(F)$, we have:
$$\delta(g,1)\delta^{-1} = (\delta g \delta^{-1},\gamma_\delta(g)) \textup{ for } \delta = \left[ \begin{array}{cc}
1 & 0   \\
  0 & -1
\end{array} \right] \textup{ and } g = \left[ \begin{array}{cc}
 a & b \\
 c & d 
\end{array} \right]$$
with $\gamma_\delta(g) = 1$ if $c \neq 0$ and $\gamma_\delta (g) = (d,-1)$ if $c=0$. When $g$ is semsimple, the result is already known. When $g$ is not semisimple, we can assume, up to conjugating, that $c=0$ and $a=d=\pm 1$ and $b \neq 0$. Since $(g,1)^{-1} = (g^{-1},(a,a))$ and $\delta(g,1)\delta^{-1}=(g^{-1},(d,-1))$, the result follows in dimension $2$.

Let $(g,\lambda) \in \textup{Mp}(W)$. There exists \cite[Chap 4, I.3]{mvw} an orthogonal decomposition into non-zero symplectic subspaces $W = \oplus_{i \in I} W_i$ stabilised by $g$. In particular:
$$g = \oplus_{i \in I} g_i = (g_i)_{i \in I} \in \prod_{i \in I} \textup{Sp}(W_i).$$
We want to apply the induction hypothesis thanks to the so-called diagonal ``embedding'' $\prod_{i \in I} \textup{Mp}(W_i) \to \textup{Mp}(W)$ of metaplectic groups \cite[Chap 4, I.9]{mvw}. We can therefore use the induction hypothesis as long as $I$ has cardinality at least $2$.

In the remaining cases \cite[Chap 4, I.3]{mvw}, the minimal polynomial of $g$ is either of the form $P^d$ with $P \in F[Z^{\pm 1}]$ irreducible, which we refer to as (I), or $P^d Q^d$ with $P$ and $Q$ irreducible and coprime, which we refer to as (II). In both cases, we have an involution $\tau$ on $F[Z^{\pm 1}]$ such that $\tau(P) \in P F[Z^{\pm 1}]$ for (I) or $\tau(P) \in Q F[Z^{\pm 1}]$ for (II). All the claims below find their justification in \cite[Chap 4, I.3--I.7]{mvw}.

We first consider the second case, which is easier, because there exists a complete polarisation $W=X \oplus Y$ such that the action of $F[Z^{\pm 1}]$ factors through $A_P = F[Z^{\pm 1}]/(P^d)$ on $X$ and $A_Q = F[Z^{\pm 1}]/(Q^d)$ on $Y$, and the action of $A_P$ and $A_Q$ respectively are faithful. The metaplectic cocycle is then very similar to the dimension $2$ because:
$$\delta(g,1)\delta^{-1} = (g,(a,-1)^m) \textup{ for } \delta = \left[ \begin{array}{cc}
\textup{id}_X & 0   \\
  0 & -\textup{id}_Y
\end{array} \right] \textup{ and } g = \left[ \begin{array}{cc}
g_X & 0   \\
  0 & g_Y
\end{array} \right].$$
Choosing a symplectic basis, we have:
$$w (g,1) w^{-1}=(g^{-1},1) = (g,(a,a)^m)^{-1} \textup{ where } w=\left[ \begin{array}{cc}
0 & -J   \\
  J & 0
\end{array} \right].$$
Therefore the result follows in this case.

Now assume we are in case (I) and set $A = F[Z^{\pm 1}]/(P^d)$ with $d \geq 2$. Note that $P$ is a separable polynomial. Indeed, inseparability problems can only occur if $\textup{char}(F) > 0$, but $\textup{char}(F) = p > n / 2$, and as the degree of $P$ is at most $n / 2$, the polynomial $P$ must be separable. Therefore we can use Jordan decomposition for $g$. We interpret it in the following: there exists $Q \in A$ such that the map $F[Y^{\pm 1}] \to F[Q] \subseteq A$ sending $Y$ to $Q$ factors through an isomorphism $F[Y^{\pm 1}]/(P) \overset{\sim}{\to} F[Q] \subseteq A$. If we set $K = F[Y^{\pm 1}]/(P)$, then $A$ is a $K$-algebra. Note that $Q$ will be corresponding to the semisimple part in Jordan decomposition. The involution $\tau$ acts on $K$, and we denote its fixed points by $K_0$. Then $W$ is an anti-hermitian space over $K$, therefore it is a symplectic space over $K_0$ and the element $g \in \textup{Sp}(W_{K_0})$ by  \cite[Chap 4, I.12]{mvw}. Moreover, the element $(g,1) \in \textup{Mp}(W_{K_0})$ and if $\textup{dim}_F(K_0) > 1$, we can apply the induction hypothesis as $\textup{dim}_{K_0}(W) < \textup{dim}_F(W)$. Therefore the lemma holds.

There only remains to check the last case $K_0 = F$ \textit{i.e.} $P$ has degree $1$ and $K=F$ or $P$ has degree $2$ and $K/F$ is quadratic. If $P$ has degree $1$, it must be of the form $Z \pm 1$. We treat this case first. We start by $P=Z-1$. In particular $g$ is a unipotent element and there exists a complete polarisation $W=X \oplus Y$ and $g \in P(X)$ is unipotent. Write:
$$g=\left[ \begin{array}{cc}
u_X & u_{XY}   \\
  0 & u_Y
\end{array} \right] \textup{ and } \delta = \left[ \begin{array}{cc}
\textup{id}_X & 0   \\
  0 & \textup{id}_Y
\end{array} \right].$$
Then: 
$$\delta(g,1)\delta^{-1} = (\delta g \delta^{-1},1) = \left( \left[ \begin{array}{cc}
u_X & - u_{XY}   \\
  0 & u_Y
\end{array} \right],1 \right).$$
Moreover the conjugation by an element $p \in P(X)$ preserves unipotent elements \textit{i.e.} $p(u,1)p^{-1}=(pup^{-1},1)$ for all $u \in P(X)$ unipotent. Therefore $\delta(g,1)\delta^{-1} = (\delta g \delta^{-1},1)$ and $(g,1)^{-1} = (g^{-1},1)$ are conjugate. When $P=Z+1$, we show the lemma is a consequence of the previous case. We have: 
$$(g,1)(-\textup{id}_W,1) = (-g,(-1,-1)^m).$$
Therefore $\delta(-g,1)\delta^{-1}$ and $(-g,1)^{-1}$ are conjugate since $-a=1$. From: 
$$\delta(-\textup{id}_W,1)\delta^{-1}  = (-\textup{id}_W,(-1,-1)^m),$$
we obtain:
$$\delta(g,1)\delta^{-1} = (-\textup{id}_W,(-1,-1)^m) \cdot \delta(-g,1)\delta^{-1}.$$
Let $h \in \textup{Sp}(W)$ such that $h(\delta(-g,1)\delta^{-1})h^{-1} = (-g,1)^{-1} = (-g^{-1},1)$. Because $-\textup{id}_W$ is central, it commutes with $h$ and we have $h(-\textup{id}_W,\epsilon)h^{-1} =  (-\textup{id}_W,\epsilon)$. This implies: 
$$h(\delta(g,1)\delta^{-1})h^{-1} = (-\textup{id}_W,(-1,-1)^m)(-g^{-1},1) = (g^{-1},(-1,-1)^m) =(g,1)^{-1}.$$
This finishes the case $P$ has degree $1$.

If $P = Z^2 + bZ +c$, then $\tau(P) = Z^{-2} + b Z^{-1} + c \in P F[Z^{\pm1}]$. Therefore writing $P = (Z-\lambda)(Z-\mu)$ for $\lambda, \mu \in K$, the roots of $\tau(P)$ must be $\lambda^{-1}$ and $\mu^{-1}$, so for $P$ to be irreducible we must have $\lambda = \mu^{-1}$ \textit{i.e.} $c=1$ and $P = Z^2 +bZ +1$. Let $E = F[Z^{\pm 1}] /(P)$. Because the characteristic of $F$ is not $2$, we can write the Jordan decomposition and construct an anti-hermitian form \cite[Chap 4, I.5 \& I.12]{mvw} over $E$ that restricts to our symplectic form over $F$ by restriction of scalars. We denote by $V = W$ the space thus obtained and $U(V)$ its isometry group. Then $g \in U(V)$. 

Moreover, the inverse image $\widetilde{U(V)}$ of $U(V)$ in $\textup{Mp}^8(W) = \textup{Sp}(W) \times_c \mu_8$ is split \textit{i.e.} there exists an homeomorphism of central extensions:
$$\begin{array}{ccc} U(V) \times \mu_8 & \overset{\sim}{\rightarrow} & \widetilde{U(V)} \\ 
(u,\lambda) & \mapsto & \lambda \cdot \iota(u) \end{array}.$$
Let $\delta \in \textup{GL}_F(V)$ be $E$-sesquilinear as in the unitary case. The conjugation by $\delta$ on $\widetilde{U(V)}$ induces an automorphism of central extensions $(g,\lambda) \mapsto (\delta g \delta^{-1}, \lambda \lambda_\delta(g))$ of $U(V) \times \mu_8$. The $U(V)$-conjugacy class of $\lambda \iota(g)$ corresponds to the $U(V)$-conjugacy class of $(g,\lambda)$. We want to show $\delta \lambda \iota(g) \delta^{-1}$ and $(\lambda \iota(g))^{-1}$ belong to the same $U(V)$-conjugacy class if $(g,\lambda) \in \textup{Mp}(W)$ \textit{i.e.} $(\delta g \delta^{-1},\lambda \lambda_\delta(g))$ and $(g,\lambda)^{-1} = (g^{-1},\lambda^{-1})$ are $U(V)$-conjugate. Note that $g^{-1}$ is $U(V)$-conjugacy to $\delta g \delta^{-1}$ by Lemma \ref{conjugacy_g_g_inverse_delta_lem}. Let $g_{\textup{ss}} \in U(V)$ be an element in the unique semisimple conjugacy class in the $U(V)$-closure of the conjugacy class of $g$. 

Because the lemma is true in the semisimple case and $\delta$ defines, by restriction of scalars, an element of similitude $-1$, the elements $\delta \lambda \iota(g_{\textup{ss}}) \delta^{-1}$ and $(\lambda \iota(g_{\textup{ss}}))^{-1}$ of $\textup{Mp}(W)$ belong to the same $\textup{Sp}(W)$-conjugacy class, and so do $(\delta g_{\textup{ss}} \delta^{-1},\lambda \lambda_\delta(g))$ and $(g_{\textup{ss}}^{-1},\lambda^{-1})$, but the first is $U(V)$-conjugate to $(g_{\textup{ss}}^{-1},\lambda \lambda_\delta(g))$ by Lemma \ref{conjugacy_g_g_inverse_delta_lem}. Therefore $\lambda^2 \lambda_\delta(g) = 1$.

Consider the $\widetilde{U(V)}$-closure $\overline{C_{(g,\lambda)}}$ of the $U(V)$-conjugacy class $C_{(g,\lambda)}$ of $(g,\lambda)$. In particular $(g_{\textup{ss}},\lambda) \in \overline{C_{(g,\lambda)}}$. By continuity, we also have: 
$$(g_{\textup{ss}}^{-1},\lambda^{-1}) \in \overline{C_{(g^{-1},\lambda^{-1})}} \textup{ and } (\delta g_{\textup{ss}} \delta^{-1},\lambda \lambda_\delta(g_{\textup{ss}})) \in \overline{C_{(\delta g \delta^{-1},\lambda \lambda_\delta(g))}},$$
therefore their intersection is non-empty. But the right-hand side is also $\overline{C_{(g^{-1},\lambda \lambda_\delta(g))}}$, where $(g^{-1},\lambda \lambda_\delta(g))$ differs from $(g^{-1},\lambda^{-1})$ by a scalar. Because of the splitting, the closures $\overline{C_{(u,\lambda)}}$ and $\overline{C_{(u,\lambda')}}$ are disjoint if $\lambda \neq \lambda'$, so $(g^{-1},\lambda^{-1}) \in C_{(\delta g \delta^{-1},\lambda \lambda_\delta(g))}$ \textit{i.e.} $\delta \lambda \iota(g) \delta^{-1}$ and $(\lambda \iota(g))^{-1}$ are $U(V)$-conjugate, so $\textup{Sp}(W)$-conjugate by restriction of scalars. \end{proof}

Therefore, the homeomorphism $\sigma$ preserves conjugacy classes \textit{i.e.} $\widehat{\textup{Sp}}(W)$-orbits when $\gamma$ is the conjugation action. As a result, the condition of Proposition \ref{invariance_Matryoshka_stratif_a_la_GK_prop} holds as $\sigma$ has modulus $|\sigma|=1$ and distributions on an orbit $S$ are quotient Haar measures. It implies that $\textup{tr}_\pi$ is $\sigma$-invariant, or equivalently $(\pi^\vee)^\delta \simeq \pi$ \textit{i.e.} $\pi^\vee \simeq \pi^\delta$. \end{proof}

\subsection{Witt towers and dual pairs} \label{witt_towers_dual_pairs_type_I_section}

\paragraph{Witt towers.} Let $V^0$ be an anisotropic $\varepsilon$-hermitian space. Recall that $\mathbb{H}$ denotes the $\varepsilon$-hermitian plane. We say an $\varepsilon$-hermitian space $V$ is in the Witt series of $V^0$ if $V$ is isometric to $V^0 \oplus m \mathbb{H}$ for some non-negative integer $m$. The Witt tower of $V^0$ will be the isometry class of $\{ V^0 \oplus m \mathbb{H} \ | \ m \in \mathbb{N} \}$. The Witt index of $V^0 \oplus m \mathbb{H}$ is exactly $m$ and its dimension is $\textup{dim}_D(V^0) + 2 m$. Given an $\varepsilon$-hermitian space $V$, there exists an anisotrpic $V^0$ such that $V$ is isometric to $V^0 \oplus m \mathbb{H}$ where $m$ is the Witt index of $V$. Moreover $V^0$ is unique up to isometry. Note that the anisotropic part is uniquely determined as a subspace of $V$ if we fix a maximal split space $m \mathbb{H} \hookrightarrow V$ because $V^0 = (m \mathbb{H})^\perp$ in $V$.

\paragraph{Type I dual pairs.} Let $V_1$ be an $\varepsilon_1$-hermitian space over $D$ and let $V_2$ be $\varepsilon_2$-hermitian. Assume $\varepsilon_1 = - \varepsilon_2$. The tensor product $W = V_1 \otimes_D V_2$ from Section \ref{hermitian_spaces_sec} is a symplectic space over $F$ and we have natural embeddings of $U(V_1)$ and $U(V_2)$ into $\textup{Sp}(W)$ via $u_1 \mapsto u_1 \otimes_D \textup{id}_{V_2}$ and $u_2 \mapsto \textup{id}_{V_1} \otimes_D u_2$. As these groups are mutual centralisers, we get a group morphism, sometimes called an embedding even though it may fail to be injective:
$$\begin{array}{ccc}
 U(V_1) \times U(V_2) & \to & \textup{Sp}(W) \\
 (u_1,u_2) & \mapsto & u_1 \otimes_D u_2 \end{array}.$$

\paragraph{Dual pairs for Witt towers.} Consider the Witt tower $\{ V_2^0 \oplus m \mathbb{H} \ | \ m \in \mathbb{N}\}$ and endow it with a totally ordered set structure by fixing succesive inclusions:
$$V_2^0 \subset V_2^0 \oplus \mathbb{H} \subset V_2^0 \oplus 2 \mathbb{H} \subset \cdots.$$
This gives rise in an obvious way to the totally ordered set $\{ U(V_2^0 \oplus m \mathbb{H}) \ | \ m \in \mathbb{N} \}$. We call it a Witt tower of groups. In addition, the pair $(U(V_1),U(V_2^0 \oplus m \mathbb{H}))$ naturally is a dual pair in $\textup{Sp}(V_1 \otimes_D (V_2^0 \oplus m \mathbb{H}))$. We introduce below the notion of seesaw dual pairs which keeps track of dual pair properties when the index $m$ is varying, that is when going up and down in the tower.

\paragraph{Seesaw dual pairs.} Let $\dot{V}_2$ and $\ddot{V}_2$ be two $\varepsilon_2$-hermitian spaces. Tensoring with $V_1$, they produce two spaces $\dot{W} = V_1 \otimes_D \dot{V}_2$ and $\ddot{W} = V_1 \otimes_D \ddot{V}_2$ as well as their sum $W=\dot{W} \oplus \ddot{W}$. Setting $V_2 = \dot{V}_2 \oplus \ddot{V}_2$ we also have $W = V_1 \otimes_D V_2$. We can see the groups $U(V_1) \times  U(V_1)$, $U(\dot{V}_2) \times U(\ddot{V}_2')$ and $U(V_2)$ as natural subgroups of the symplectic group $\textup{Sp}(W)$. In this situation, one has a diagram:
$$\xymatrix{
		U(V_1) \times U(V_1) \ar@{-}[rd] \ar@{-}[d] &  U(V_2) \ar@{-}[d] \\
		U(V_1) \ar@{-}[ru] & U(\dot{V}_2) \times U(\ddot{V}_2)
		}.$$
where vertical edges stand for a natural inclusion relation -- here the left vertical arrow is the diagonal embedding of $U(V_1)$ into the product -- and diagonal ones relate to groups forming a dual pair in $\textup{Sp}(W)$. This diagram is quite convenient to sum up the situation for seesaw dual pairs. To be more explicit, the dual pair embeddings are given by the identification $W = \dot{W} \oplus \ddot{W} = (V_1 \otimes \dot{V}_2) \oplus (V_1 \otimes \ddot{V}_2)$:
$$((u_1,u_1'),(\dot{u}_2,\ddot{u}_2)) \in (U(V_1) \times U(V_1)) \times (U(\dot{V}_2) \times U(\ddot{V}_2)) \mapsto (u_1 \otimes_D \dot{u}_2) \oplus (u_1' \otimes_D \ddot{u}_2) \in \textup{Sp}(W)$$
and the identification $W = V_1 \otimes V_2$:
$$(u_1,u_2) \in U(V_1) \times U(V_2) \mapsto u_1 \otimes_D u_2 \in \textup{Sp}(W).$$

\paragraph{Weil representation.} Let $W$ be a symplectic space. The Weil representation associated to $\psi$ and $W$ is an isomorphism class $\omega_\psi$ of smooth representations of the metaplectic group $\textup{Mp}(W)$ and, by extension, we also use the same denomination for any of its representatives. There are several famous models of the Weil representation: the so-called Schrödinger model $\omega_{\psi,X}$ is associated to a Lagrangian $X$ in $W$; whereas the lattice model $\omega_{\psi,L}$ is associated to a self-dual lattice $L$ in $W$. Encompassing the previous two, any self-dual subgroup $A$ in $W$ gives rise to a model $\omega_{\psi,A}$.

\paragraph{Diagonal metaplectic embedding.} Let $\dot{X}$ be a Lagrangian in $\dot{W}$. We have a realisation of the metaplectic group $\textup{Mp}^{c_{\dot{X}}}(\dot{W})$ associated to this Lagrangian together with its Schrödinger model $(\omega_{\psi,\dot{X}},S_{\psi,\dot{X}})$. Similarly for another $\ddot{W}$ and $\ddot{X}$, consider the associated group $\textup{Mp}^{c_{\ddot{X}}}(\ddot{W})$ and representation $(\omega_{\psi,\ddot{X}},S_{\psi,\ddot{X}})$. Set $W=\dot{W} \oplus \ddot{W}$ and $X=\dot{X} \oplus \ddot{X}$. We have a morphism of central extensions:
$$\begin{array}{ccc}
 \textup{Mp}^{c_{\dot{X}}}(\dot{W}) \times  \textup{Mp}^{c_{\ddot{X}}}(\ddot{W})  & \to & \textup{Mp}^{c_X}(W) \\
 ((\dot{g},\dot{\lambda}),(\ddot{g},\ddot{\lambda})) & \mapsto & (\dot{g} \oplus \textup{id}_{\ddot{W}},\dot{\lambda}) \cdot (\textup{id}_{\dot{W}} \oplus \ddot{g},\ddot{\lambda}) \end{array}$$
compatible to the Weil representations \textit{i.e.} $\omega_{\psi,\dot{X}} \otimes \omega_{\psi,\ddot{X}} \simeq \omega_{\psi,X}$ via pullback. We refer to the previous morphism as the diagonal embedding of metaplectic groups, even though it is not an actual embedding as it is not injective.

\paragraph{The mixed model.} Consider the totally isotropic flag $\Phi = \{ \dot{X} \}$ with the notations of the previous paragraph. The mixed model Schr\"odinger model is a situation in-between the diagonal embedding and the Schr\"odinger model. It describes the action of the parabolic $P_\Phi \subset \textup{Mp}^{c_X}(W)$ ``through'' the diagonal embedding. Let $P_\Phi = (G_\Phi \times U_\Phi ) \ltimes N(\Phi)$ be the Levi decomposition where $U_\Phi = \textup{Mp}^{c_{\ddot{X}}}(W)$ and $G_\Phi \subset \textup{Mp}^{c_{\dot{X}}}$. For $n \in N(\Phi)$, there exists $h_n \in \textup{Hom}_F(\ddot{W},\dot{X})$ such that $n(\ddot{w}) = \ddot{w} + h_n(\ddot{w})$ for all $\ddot{w} \in \ddot{W}$. We can define:
$$N(\Phi)_1 = \{ n \in N(\Phi) \ | \ h_n=0 \}.$$
This is a normal subgroup of $N(\Phi)$ and $G_\Phi \ltimes N(\Phi)_1$ is the Siegel parabolic in $\textup{Mp}^{c_{\dot{X}}}(\dot{W})$. The group $(G_\Phi \ltimes N(\Phi)_1) \times U_\Phi$ acts via the diagonal embedding on $\omega_{\psi,X} \simeq \omega_{\psi,\dot{X}} \otimes \omega_{\psi,\ddot{X}}$. The whole group $N(\Phi) \subset \textup{Mp}^{c_X}(W)$ acts on $\omega_{\psi,X}$, but only $N(\Phi)_1$ is acting through the diagonal embedding \textit{i.e.} separately on each piece of the tensor product. We can describe the action of $N(\Phi)$ on this tensor product, to obtain the formulas for the so-called mixed Schr\"odinger model. The group morphism $n \mapsto h_n$ factors through a group isomorphism $N(\Phi) / N(\Phi)_1 \simeq \textup{Hom}_F(\ddot{W},\dot{X})$. For $h \in \textup{Hom}_F(\ddot{W},\dot{X})$, let $n_h \in N(\Phi)$ be the element $n_h(\ddot{w}) = \ddot{w} + h(\ddot{w})$ and $n_h(\dot{y}) = -\frac{1}{2} h h^*(\dot{y}) - h^*(\dot{y}) + \dot{y}$, or equivalently with respect to the decomposition $W = \dot{X} \oplus \ddot{W} \oplus \dot{Y}$, the matrix representation of this element is:
$$n_h = \left[ \begin{array}{ccc}
 1 & h &  -\frac{h h^*}{2}   \\
   & 1 & -h^* \\
   &   & 1
\end{array} \right] \in \textup{Sp}(W).$$
In general the map $h \mapsto n_h$ is a section of $n \mapsto h$ between sets and not between groups. However we have $n_h^{-1} = n_{-h}$. Therefore $n_h^{-1} (\dot{w} + \ddot{w}) = \dot{w} - \frac{1}{2} h h^*(\dot{y})+ \ddot{w} + h(\ddot{w}) - h^*(\dot{y})$, so in the Heisenberg group: 
$$\bigg(n_h^{-1}(\dot{w} + \ddot{w}),1\bigg) =\bigg(h(\ddot{w}) - \frac{1}{2} h h^*(\dot{y}),\frac{1}{2}\langle \dot{y} , h(\ddot{w}) \rangle \bigg) \bigg(\dot{w} + \ddot{w} - h^*(\dot{y}),1 \bigg),$$
with the first factor in $\dot{X} \times F$.

If $f=\dot{f} \otimes \ddot{f}$ is an elementary tensor via the isomorphism $S_{\psi,X} \simeq S_{\psi,\dot{X}} \otimes S_{\psi,\ddot{X}}$, then $f (n_h^{-1}(\dot{w} + \ddot{w}),1) = \psi(\frac{1}{2}\langle \dot{y} , h(\ddot{w}) \rangle) \times \dot{f}(\dot{w},1) \times \ddot{f}(\ddot{w} - h^*(\dot{y}),1))$ and we obtain:
$$(\omega_{\psi,X}(n_h,1) \cdot f)(\dot{w} + \ddot{w}) = \dot{f}(\dot{w},1) \times \rho_{\psi,\ddot{X}}(-h^*(\dot{y}),1) \ddot{f}(\ddot{w},1).$$

\paragraph{Weil representation in Witt towers} We start with $W = V_1 \otimes_D V_2$. We assume given decompositions $V_1 = X_1 \oplus V_1^0 \oplus Y_1$ and $V_2 = X_2 \oplus V_2^0 \oplus Y_2$ where $V_1^0$ and $V_2^0$ are anisotropic spaces. Set $V_1^{\mathbb{H}} = X_1 \oplus Y_1$ and $V_2^{\mathbb{H}} = X_2 \oplus Y_2$ for the split spaces obtained.

We have a first orthogonal direct sum decomposition:
$$\begin{array}{ccccc}
W & = & (V_1 \otimes V_2^0) & \oplus & (V_1 \otimes V_2^{\mathbb{H}}) \\
& = & W^0 & \oplus & W^{\mathbb{H}}
\end{array}$$
Let $X^0$ be a Lagrangian in $W^0$. If $V_2^0 \neq 0$, let $H_1 = U(V_1) \times_{c_1} R^\times$ be the inverse image of $U(V_1)$ in $\textup{Mp}^{c_{X^0}}(W^0)$. Otherwise we set $H_1 = U(V_1) \times R^\times$ and in this case $c_1$ is the trivial cocycle. Define $X^{\mathbb{H}} = V_1 \otimes X_2$. Then $X^W = X^0 \oplus X^{\mathbb{H}}$ is a lagragian in $W$ and we can consider the composition of the group morphism:
$$\begin{array}{ccc}
H_1 & \to & \textup{Mp}^{c_{X^0}}(W^0) \times \textup{Mp}^{c_{X^{\mathbb{H}}}}(W^{\mathbb{H}}) \\
(u_1,\lambda_1) & \mapsto & ((u_1 \otimes_D \textup{id}_{V_2^0},\lambda_1),(u_1\otimes \textup{id}_{V_2^\mathbb{H}},\Omega_{1,\textup{det}_F(u_1)^{m_2}}))
\end{array}$$
with the diagonal embedding to obtain $\iota_{\psi,X^W} : H_1 \to \textup{Mp}^{c_{X^W}}(W)$. The Weil representation associated to $\psi$ and $X^W$ is $(\omega_{\psi,{X^W}} \circ \iota_{\psi,X^W}, S_{\psi,X^W}) \in \textup{Rep}_R(H_1)$. If $X$ is another Lagrangian in $W$, we can realise it over $S_{\psi,X}$ via $(\omega_{\psi,X} \circ \gamma_{X^W,X} \circ \iota_{\psi,X^W},S_{\psi,X}) \in \textup{Rep}_R(H_1)$.

We can apply the same procedure the other way round considering the decomposition:
$$\begin{array}{ccccc}
W & = & (V_1^0 \otimes V_2) & \oplus & (V_1^\mathbb{H} \otimes V_2) \\
& = & {}^0 W & \oplus & {}^\mathbb{H} W
\end{array}$$
to define $H_2$ and obtain $(\omega_{\psi,{}^W X} \circ \iota_{\psi,{}^W X},S_{\psi,{}^W X}) \in \textup{Rep}_R(H_2)$ with ${}^W X = {}^0 X \oplus {}^\mathbb{H} X$. Likewise we have a well-defined isomorphism class $S_{\psi,X} \in \textup{Rep}_R(H_2)$ for $X$ a Lagrangian in $W$ and the action of $H_1$ commutes with this $H_2$-action.

\begin{defi} The Weil representation $\omega_{\psi,H_1,H_2} \in \textup{Rep}_R(H_1 \times H_2)$ is the isomorphism class $(S_{\psi,X})_{X \in \Omega(W)}$ coming from the (commuting) actions of $H_1$ and $H_2$ defined above. By a usual abuse of language, an explicit model $S_{\psi,X}$ of this isomorphism class is still called the Weil representation. \end{defi}

\begin{rem} \ \begin{enumerate}[label=\textup{\alph*)}]
\item The Lagrangian ${}^W X$ above is different from $X^W$, so the natural group or model in which $H_2$ lives is in general different from that associated to $H_1$. However, the canonical isomorphisms of central extensions $\gamma$ allow to embed them in the same group and realise them over the same space $S_{\psi,X}$.
\item The inverse image $U(V_1) \times_{c_1^W} R^\times$ of $U(V_1)$ in $\textup{Mp}^{c_X}(W)$ satisfies $c_1^W=c_1$ if $U(V_1)$ is not orthogonal. So this group is $H_1$. If $U(V_1)$ is orthogonal the cocycle only depends on the parity of the Witt index $m_2 = \textup{dim}(X_2)$. It is trivial for even $m_2$, so the previous group is $H_1$ itself; it is the composition of the determinant and the Hilbert symbol otherwise, so the group is isomorphic to $H_1$ via $\iota_{\psi,X}$. Therefore $H_1$ has the advantage of being the same for all $m_2$, unlike the inverse image of $U(V_1)$.
\item If $V_2 =X_2 \oplus Y_2$ is split, then $H_1 = U(V_1) \times R^\times$ acts on $S_{\psi,X^\mathbb{H}} \simeq C_c^\infty(V_1 \otimes Y_2)$ via the geometric action on $V_1$ \textit{i.e.} $(u_1,\lambda_1) \cdot f = \lambda_1 \times \left( f \circ (u_1^{-1} \otimes \textup{id}_{Y_2}) \right)$. \end{enumerate}
\end{rem}

\paragraph{Theta lifts.} According to Sections \ref{largest-isotypic-quotient-sec} and \ref{covering-groups-sec}, we can define, for $\pi_1 \in \textup{Irr}_R(H_1)$ a genuine representation, the largest $\pi_1$-isotypic of the Weil representation $\omega_{\psi,H_1,H_2}$, and there exists a genuine representation $\Theta(\pi_1,V_2) \in \textup{Rep}_R(H_2)$ such that: 
$$(\omega_{\psi,H_1,H_2})_{\pi_1} \simeq \pi_1 \otimes \Theta(\pi_1,V_2).$$
Unless otherwise stated, all our representations of covering groups are genuine. If we go the other way round, we will write $\Theta(\pi_2,V_1) \in \textup{Rep}_R(H_1)$ where $\pi_2 \in \textup{Irr}_R(H_2)$.

\section{Type I dual pairs}

\subsection{Filtrations on degenerate principal series} \label{Rallis_filtration_sec}

Let $V_1$ be a right $\varepsilon_1$-hermitian space and $V_2$ be a left $\varepsilon_2$-hermitian space, with $\varepsilon_1 = - \varepsilon_2$ and respective dimensions $n_1$ and $n_2$. Assume furthermore that $V_2$ is a split space \textit{i.e.} it is isometric to $m_2 \mathbb{H}$. In particular the Witt index of $V_2$ is $m_2 = \frac{1}{2} n_2$. We fix a complete polarisation $V_2 = X_2 \oplus Y_2$. It induces a complete polarisation on $W = V_1 \otimes V_2$ by setting $X=V_1 \otimes X_2$ and $Y= V_1 \otimes Y_2$. Let $i_1$ and $i_2$ be the embeddings induced by $i : U(V_1) \times U(V_2) \to \textup{Sp}(W)$ on each of the two factors.

Let $H_2 = U(V_2) \times_{c_2} R^\times$ where $c_2$ is the $\{ \pm 1 \}$-valued cocycle obtained from ${}^0 X$ as in the previous section. We can assume ${}^0 X = V_1^0 \otimes X_2$. Then $P_2 = P(X_2) \times_{c_2} R^\times$ is a subgroup of $H_2$ and for $p$ and $p'$ in $P(X_2)$, the $2$-cocycle is: 
$$c_2(p,p') = (\textup{det}_{{}^0 X}(i_2(p)),\textup{det}_{{}^0 X}(i_2(p')))_F = (\textup{det}_{X_2}(p)^{n_1^0} , \textup{det}_{X_2}(p')^{n_1^0})_F$$
where $\textup{det}_{X_2}(p_2) = \textup{det}_F(p_2|_{X_2})$ and $p_2|_{X_2} \in \textup{GL}_D(X_2) \simeq \textup{GL}_{m_2}(D)$. As $n_1$ and $n_1^0$ have same parity:
 $$c_2(p,p') = (\textup{det}_{X_2}(p) , \textup{det}_{X_2}(p'))_F^{n_1^2} =  (\textup{det}_{X_2}(p) , \textup{det}_{X_2}(p'))_F^{n_1^{ \ }}.$$
Define the smooth character $\nu_{X_2,n_1} : P_2 \to R^\times$ by: 
$$\nu_{X_2,n_1}(p,\lambda) = \lambda \times \Omega_{1,\textup{det}_X(i_2(p))} = \lambda \times \Omega_{1,\textup{det}_{X_2}(p)^{n_1}} $$ According to the properties of the non-normalised Weil factor \cite[Prop. 1.5 g)]{trias_theta1}:
$$\nu_{X_2,n_1}(p,1) = (-1,\textup{det}_{X_2}(p))_F^{\frac{n_1(n_1-1)}{2}} \times (\Omega_{1,\textup{det}_{X_2}(p)})^{n_1};$$
alternatively, if we write $n_1=2 k_1 + j_1$ with $k_1$ and $j_1$ integers:
$$\nu_{X_2,n_1}(p,1) = \nu_{X_2,2k_1+j_1}(p,1) = | \textup{det}_{X_2}(p)|_F^{k_1} \times \Omega_{1,\textup{det}_{X_2}(p)}^{j_1}$$
so $\nu_{X_2,2 k_1 + j_1} = |\textup{det}_{X_2}|_F^{k_1} \times \nu_{X_2,j_1}$. Also $(\nu_{X_2,n_1}(p,1))^2 = (-1,\textup{det}_{X_2}(p))^{n_1} |\textup{det}_{X_2}(p)|_F^{n_1}$.

We consider the Weil representation $S_{\psi,X} \in \textup{Rep}_R(H_1 \times H_2)$ where $H_1 = U(V_1) \times R^\times$. The group $U(V_1)$ acts geometrically, \textit{i.e.} via $I_p \cdot f = f \circ p^{-1}$ for $p \in P(X)$ and $f \in S_{\psi,X}$, on the Weil representation via $u_1 \mapsto I_{u_1 \otimes \textup{id}_{V_2}}$. The coinvariant representation $(S_{\psi,X})_{U(V_1)}$ can be considered as a subrepresentation of some explicit parabolically induced representation. This is originally attributed to Rallis.

\begin{theo} \label{Rallis_argument_thm} There exists an injective $H_2$-morphism
$$(S_{\psi,X})_{U(V_1)} \hookrightarrow \textup{Ind}_{P_2}^{H_2}(\nu_{X_2,n_1}).$$
In other words $(\omega_{\psi,X})_{U(V_1)}$ is isomorphic to a subrepresentation of $\textup{Ind}_{P_2}^{H_2}(\nu_{X_2,n_1})$.
\end{theo}

\begin{proof} Define the map: 
$$\begin{array}{cccl}
\Phi :& C_c^\infty(Y) & \to & C^\infty(H_2) \\
 & f & \mapsto & \Phi(f) : h_2 \mapsto \omega_{\psi,X}(\textup{Id}_{V_1},h_2) f(0)  \end{array}.$$
First of all, the image of $\Phi$ lies in $\textup{Ind}_{P_2}^{H_2}(\nu_{X_2,n_1})$ because, thanks to the formula for the Schrödinger model, we have $\Phi(f)(p_2 h_2) = \nu_{X_2,n_1}(p_2) \times \Phi(f)(h_2)$ for all $p_2 \in P_2$, $h_2 \in H_2$ and $f \in C_c^\infty(Y)$. Now as $H_1$ acts geometrically on $Y$, the map $\Phi$ is $H_1$-invariant \textit{i.e.} $\Phi(\omega_{\psi,X}(u_1,(\textup{Id}_{V_2},1))f) = \Phi(f)$ for all $f \in C_c^\infty(Y)$ and $u_1 \in H_1$. In particular $\Phi$ factorises through the biggest trivial isotypic quotient \textit{i.e.} factorises through some $(C_c^\infty(Y))_{H_1} \to C^\infty(H_2)$.

We will prove that $\textup{Ker}(\Phi) = C_c^\infty(Y)[H_1]$, which amounts to the injectivity of the previous map. The inclusion $\supseteq$ holds as the map $\Phi$ is $H_1$-invariant, so we will show the opposite inclusion in the remainder of the proof.

Recall that $Y = W_1 \otimes Y_2 \simeq \textup{Hom}(Y_2^*,W_1)$ and decompose $Y = \sqcup_{i=0}^r Z_i$ according to the rank in $\textup{Hom}(Y_2^*,W_1)$, with $r = \textup{min}(m_2,n_1)$. This is a closed $H_1$-stratification of $Y$ with finite regular strata by \cite[Chap 3, V.4]{mvw}. We now explain the finite induction argument we use to show the remaining inclusion $\textup{Ker}(\Phi) \subset C_c^\infty(Y)[H_1]$. Let $f \in \textup{Ker}(\Phi)$. First of all we claim that $f \in C_c^\infty(Y)[H_1] + C_c^\infty(Y \backslash Z_0)$. Indeed as $f \in \textup{Ker}(\Phi)$ we have:
$$\Phi(f)((\textup{Id}_{V_2},\lambda)) = \lambda \times f(0) = 0.$$
In this particular situation, note that $Z_0 = \{ 0 \}$ itself is a single $H_1$-orbit. The local $H_1$-invariant measure $\mu_\mathcal{O}$ on $C_c^\infty(\mathcal{O})$ is, up to a scalar, the counting measure \textit{i.e.} the evaluation at $0$. So Lemma \ref{stratification_G_conivariants_lemma} applies and $f \in C_c^\infty(Y)[H_1] + C_c^\infty(Y \backslash Z_0)$.

As a result, we can choose $h \in C_c^\infty(Y)[H_1]$ such that $f-h \in C_c^\infty(Y \backslash Z_0)$. The latter element $f - h$ still belongs to $\textup{Ker}(\Phi)$. Assume the following fact for the moment: the orbital integrals of $f-h$ vanish on $Z_1$ \textit{i.e.} for all $H_1$-orbits $\mathcal{O}$ in $Z_1$ and for all $H_1$-invariant measures $\mu_\mathcal{O}$ on $C_c^\infty(\mathcal{O})$ we have $\mu_\mathcal{O}(f-h) = 0$. Again, by Lemma \ref{stratification_G_conivariants_lemma}, there exists $h' \in C_c^\infty(Y)[H_1]$ such that $f-h-h' \in \textup{Ker}(\Phi) \cap C_c^\infty(Y \backslash (\sqcup_{i=0}^1 Z_i))$. The finite induction argument will run over the rank, eliminating terms in the stratification as long as orbital integrals vanish. Therefore we simply have to prove this vanishing result of orbital integrals for the finite induction to work:

\begin{lem} Let $f \in \textup{Ker}(\Phi) \cap C_c^\infty(Y \backslash (\sqcup_{i=0}^{k-1} Z_i))$ where $k$ is an integer. Then for all $H_1$-orbits $\mathcal{O}$ in $Z_k$, we have  $\mu_\mathcal{O}( \textup{res}_\mathcal{O}(f)) = 0$ \textit{i.e.} $\textup{res}_\mathcal{O}(f) \in C_c^\infty(\mathcal{O})[H_1]$. \end{lem}

\begin{proof} First of all, it is enough to test it for the Haar measure $\mu_\mathcal{O}$ on $C_c^\infty(\mathcal{O})$ because all other $H_1$-invariant measures are scalar multiple of $\mu_\mathcal{O}$. Take $y \in Z_k$. Through the canonical isomorphism $Y \simeq \textup{Hom}(Y_2^*,W_1)$, the element $y$ can be seen as a rank $k$ morphism in $\textup{Hom}(Y_2^*,W_1)$. Write $\textup{Ker}(y)$ to mean the kernel of the latter rank $k$ morphism and consider the set: 
$$Z_{\textup{Ker}(y)} = \{ y' \in Z_k \ | \ \textup{Ker}(y') = \textup{Ker}(y) \}.$$
Then $Z_{\textup{Ker}(y)}$ is an $H_1$-stable closed subset of $Z_k$.

When $g =\left[ \begin{array}{cc}
a & b \\
c & d
\end{array} \right] \in \textup{Sp}(W)$ with $c \in \textup{Hom}(X,Y)$, the formula for the Schrödinger model reads:
$$\big(\omega_{\psi,X}((g,\lambda)) f\big) (0) = \lambda \times \int_{X / \textup{ker}(c^*)} \psi \bigg( \frac{\langle c^* x , d^* x \rangle}{2} \bigg) f(c^*x) d\mu_g(x)$$
for some Haar measure $\mu_g$ of $X / \textup{Ker}(c^*)$. We choose $c$ of co-rank $k$ in the following way. Let $(e_i)$ be a basis of $X_2$ and $(f_i)$ be its dual basis in $Y_2$ so that $((e_i),(f_j))$ is a symplectic basis of $V_2$. Set ${}_k X_2 = \langle e_1 , \cdots, e_k \rangle$ and ${}_k Y_2 = \langle f_1 , \cdots, f_k \rangle$ and let ${}_k V_2$ be their symplectic complement \textit{i.e.} $V_2 = ({}_k X_2 \oplus {}_k Y_2) \oplus {}_k V_2$. Define $h_k(d) \in U(V_2)$ in the decomposition $V_2 = {}_k X_2 \oplus {}_k V_2 \oplus {}_k Y_2$, where $d \in \textup{End}_D({}_k Y_2)$ is $(-\epsilon_2)$-hermitian, by the matrix:
$$\left[ \begin{array}{ccc}
0 & 0 & \epsilon_2 \\
0 & 1 & 0 \\
1 & 0 & d
\end{array} \right].$$
In the decomoposition $W = (V_1 \otimes {}_k X_2) \oplus (V_1 \otimes {}_k V_2) \oplus (V_1 \otimes {}_k Y_2)$, it induces an element $g = 1 \otimes h_k(d) \in \textup{Sp}(W)$ of the form:
$$g =  \left[ \begin{array}{cc}
a & b \\
c & d
\end{array} \right] = \left[ \begin{array}{cccc}
 &  &  & \epsilon_2 \\
 & 1 &  &  \\
 &  & 1 & \\
1 &  &  & d \otimes 1
\end{array} \right].$$
The morphism $c^*$ induces an isomorphism:
$$(c^*)' : v_1 \otimes e_i \in X / \textup{ker}(c^*) \mapsto \langle e_i , \bullet \rangle \cdot v_1 \in \textup{Hom}({}_k Y_2,V_1).$$
Rewriting the integral, we obtain:
$$\int_{\textup{Hom}({}_k Y_2,V_1)} \psi ([z,zd]) f(z) dz = 0$$
where $zd = z \circ d$ and $[z,z'] = \langle z, ((c^*)')^{-1} z' \rangle$ is non-degenerate. Furthermore $f$ is zero on $\sqcup_{i=0}^{k-1} Z_i$ so the previous integral can be taken over $\textup{Hom}({}_k Y_2,V_1) \cap Z_k$ as elements have at most rank $k$.

Now the Gram matrix gives a surjective map to $\epsilon_2$-hermitian matrices:
$$\begin{array}{cccl}
\textup{Gr} : & \textup{Hom}({}_k Y_2,V_1) \cap Z_k & \to & \mathcal{M}_{k,k}^{\epsilon_2}(D) \\
 & z & \mapsto & \textup{Gr}(z) = \bigg( {\langle z f_i, z f_j \rangle}_1 \bigg)_{i, j}
\end{array}$$
that is a submersion of analytic varieties in the sense of \cite[V.2]{hc} (in matrix notations $\textup{Gr}(z) = {}^{\tau t} z J_1 z$ where $J_1$ is the $\epsilon_1$-matrix defining the form on $V_1$). Moreover $\textup{Gr}$ is a quotient map for the $H_1$-action by \cite[Chap 1, I.9]{mvw} so $\mathcal{M}_{k,k}^{\epsilon_2}(D)$ is the quotient space. Note that the map $z \mapsto \psi([z,zd])$ is constant on $H_1$-orbits and $[z,zd] = t(\textup{Gr}(z)d)$ where $t = \textup{tr}_{D/F} \circ \textup{tr}_{\mathcal{M}_k(D)}$.

Write $\textup{Gr} : M \to N$ with $M = \textup{Hom}({}_k Y_2,V_1) \cap Z_k$ and $N=\mathcal{M}_{k,k}^{\epsilon_2}(D) = M / H_1$. Let $\mu_N$ be the Haar measure of $N$ and $\mu_M$ be the restriction of the Haar measure of $\textup{Hom}({}_k Y_2,V_1)$ to the open subset $M$. By \cite[V.2]{hc}, which is still valid in our modular setting, there exists a surjective morphism $\phi : f \in C_c^\infty(M) \mapsto \phi_f \in C_c^\infty(N)$ of $R$-vector spaces, compatible with support of functions in the sense that $\textup{supp}(\phi_f) \subset \textup{Gr}(\textup{supp}(f))$, such that $\mu_M = \mu_N \circ \phi$ and for all $d \in \mathcal{M}_{k,k}^{\epsilon_2}(D)$:
$$\int_M \psi([z,z d]) f(z) d \mu_M(z) = \int_N \psi(t(n d)) \phi_f(n) d \mu_N(n) = 0.$$
The latter integral is the Fourier transform of $\phi_f$ on $N$ and therefore it implies that $\phi_f$ is identically zero. The support preservation of $\phi$ and its surjectivity implies that $f \mapsto \phi_f(\mathcal{O})$ is a non-trivial scalar multiple of $\mu_{\mathcal{O}} \circ \textup{res}_{\mathcal{O}}$. This proves the lemma. \end{proof}

Now applying the previous lemma and Lemma \ref{stratification_G_conivariants_lemma}, there exists $h \in C_c^\infty(Y \backslash (\sqcup_{i=0}^{k-1} Z_i))$ such that $f -h \in \textup{Ker}(\Phi) \cap C_c^\infty(Y \backslash (\sqcup_{i=0}^k Z_i))$.  We carry on until we run into the empty set $Y \backslash (\sqcup_{i=0}^r Z_i)$, therefore $\textup{Ker} (\Phi) \subset C_c^\infty(Y)[H_1]$. \end{proof}

The induced representation $\textup{Ind}_{P_2}^{H_2}(\nu_{X_2,n_1})$ has a remarkable filtration, which appears in \cite{kudla_rallis_boundary}. We consider a slightly more general situation by setting $I(s)=\textup{Ind}_{P_2}^{H_2}(\nu_{X_2,s})$ for $s \in \mathbb{Z}$ if $D$ is commutative, and $s \in \frac{1}{2} \mathbb{Z}$ if $D$ is quaternionic. We allow half integers in the last case as $\textup{det}_{X_2}(p_2)$ is actually the square of a reduced norm. Moreover in this case $P_2 = P(X_2) \times R^\times$ as $c_2$ must be trivial on $P(X_2)$. It is implicit in other cases, but the group law of $P_2$ depends on the parity of $s$.

Let $\dot{V}_2$ and $\ddot{V}_2$ be $\epsilon_2$-hermitian spaces in opposite Witt series so that $V_2 = \dot{V}_2 \oplus \ddot{V}_2$ is split. We assume their Witt indices satisfy $\dot{m}_2 \leq \ddot{m}_2$ and set $\delta = \ddot{m}_2 - \dot{m}_2$. In this case $\dot{V}_2$ is anti-isometric to a subspace of $\ddot{V}_2$ and we choose such an anti-isometric embedding $\alpha^- : \dot{V}_2 \hookrightarrow \ddot{V}_2$. We fix two complete isotropic flags $(\dot{X}_2^t)_{0 \leq t \leq \dot{m}_2}$ in $\dot{V}_2$ and $(\ddot{X}_2^t)_{0 \leq t \leq \ddot{m}_2}$ in $\ddot{V}_2$ and suppose they are compatible with $\alpha$ in the sense that $\alpha^-(\dot{X}_2^t) = {}^\delta \ddot{X}_2^t$ for all $t$. We also fix the anisotropic part $\dot{V}_2^0$ of $\dot{V}_2$ and define $\ddot{V}_2^0 = \alpha^-(\dot{V}_2^0)$. Set $X_2 = \dot{X}_2 \oplus \ddot{X}_2 \oplus \Delta_{\alpha^-}$ where $\Delta_{\alpha^-} = \{ (\dot{v}_2,\alpha^-(\dot{v}_2)) \ | \ \dot{v}_2 \in \dot{V}_2^0 \}$. This is a Lagrangian of $V_2$ and we can define $P_2$ in $H_2$ as earlier. We recall the notations $X=V_1 \otimes X_2$ is a Lagrangian of $W = V_1 \otimes V_2$ and $H_2 = H_2^{c_2}$ where $c_2$ is defined thanks to ${}^0 X = V_1^0 \otimes X_2$. 

Let $\dot{H}_2 = \dot{H}_2^{\dot{c}_2}$ and $\ddot{H}_2 = \ddot{H}_2^{\ddot{c}_2}$ be the groups defined by the two Lagrangians:
\begin{itemize}[label=$\bullet$]
\item ${}^0 \dot{X} = (V_1^0 \otimes \dot{X}_2) \oplus X_{V_1^0 \otimes \dot{V}_2^0} = {}^0 \dot{X}^\mathbb{H} \oplus {}^0 \dot{X}^0$;
\item ${}^0 \ddot{X} = (V_1^0 \otimes \ddot{X}_2) \oplus \alpha^-(^0\dot{X}^0) = {}^0 \ddot{X}^\mathbb{H} \oplus {}^0 \ddot{X}^0$.
\end{itemize}
In particular, ${}^0 \ddot{X} = (V_1 \otimes \ddot{X}_2^\delta) \oplus (\textup{id} \otimes \alpha^-)({}^0 \dot{X})$. We now specify the lift $\dot{H}_2 \times \ddot{H}_2 \to H_2$ of $U(\dot{V}_2) \times U(\ddot{V}_2) \hookrightarrow U(V_2)$ we consider in order to pullback $I(s)$ as a representation of $\dot{H}_2 \times \ddot{H}_2$. Note that all three cocycles above are trivial if $V_1$ is split and, in this case, the map is simply:
$$((\dot{u}_2,\dot{\lambda}_2),(\ddot{u}_2,\ddot{\lambda}_2)) \mapsto (\dot{u}_2 \oplus \ddot{u}_2,\dot{\lambda}_2 \ddot{\lambda}_2).$$
When $V_1^0 \neq 0$, the map is obtained by completing the following diagram:
$$\xymatrix{ \dot{H}_2 \times \ddot{H}_2 \ar@{-->}[rr] \ar@{->}[d] & & H_2 \ar@{->}[d] \\
\textup{Mp}^{c_{{}^0 \dot{X}}} \times \textup{Mp}^{c_{{}^0 \ddot{X}}} \ar@{->}[r] & \textup{Mp}^{c_{{}^0 \dot{X} \oplus {}^0 \ddot{X}}} \ar@{->}[r] & \textup{Mp}^{c_{{}^0 X}}
}.$$
To be more explicit, this map is:
$$((\dot{u}_2,\dot{\lambda}_2),(\ddot{u}_2,\ddot{\lambda}_2)) \mapsto (\dot{u}_2 \oplus \textup{id}_{\ddot{V}_2},\dot{\lambda}_2 \gamma(\dot{u}_2 \oplus \textup{id}_{\ddot{V}_2})) \cdot  (\textup{id}_{\dot{V}_2} \oplus \ddot{u}_2,\ddot{\lambda}_2 \gamma(\textup{id}_{\dot{V}_2} \oplus \ddot{u}_2))$$
where $\gamma$ comes from the canonical isomorphism $\textup{Mp}^{c_{{}^0 \dot{X} \oplus {}^0 \ddot{X}}} \to \textup{Mp}^{c_{{}^0 X}}$ of central extensions.

Choosing $X_1$ maximal totally isotropic in $V_1$, let $\dot{X} = {}^\mathbb{H} \dot{X} \oplus {}^0 \dot{X}$ and $\ddot{X} = {}^\mathbb{H} \ddot{X} \oplus {}^0 \ddot{X}$. The following diagram commutes:
$$\xymatrix{ \dot{H}_2 \times \ddot{H}_2 \ar@{->}[rr] \ar@{->}[d]^{\iota_{\psi,\dot{X}} \times \iota_{\psi,\ddot{X}}} & & H_2 \ar@{->}[d]^{\iota_{\psi,X}} \\
\textup{Mp}^{c_{\dot{X}}} \times \textup{Mp}^{c_{\ddot{X}}} \ar@{->}[r] & \textup{Mp}^{c_{\dot{X} \oplus \ddot{X}}} \ar@{->}[r] & \textup{Mp}^{c_X}
}.$$
Therefore $\omega_{\psi,H_2} \simeq \omega_{\psi,\dot{H}_2} \otimes \omega_{\psi,\ddot{H}_2}$ via pullback. We will study $I(s)$ via this pullback. 

We denote by $\dot{P}_t$ the parabolic in $\dot{H}_2$ that is the inverse image of $P(\dot{X}_2^t)$. As the embedding $\dot{n}_t \in N({}_t \dot{X}_2) \mapsto (\dot{n}_t,1) \in \dot{P}_t$ defines a normal subgroup, its quotient Levi is:
$$\dot{G}_t \times \dot{U}_t \twoheadrightarrow \dot{P}_t / N({}_t \dot{X}_2).$$
We use similar notations for $\ddot{H}_2$.

\begin{theo} \label{filtration_degenerate_principal_series_thm} As an $(\dot{H}_2 \times \ddot{H}_2)$-representation, the induced representation $I(s)$ has a filtration:
$$0 \subset I_0(s) \subset I_1(s) \subset \cdots \subset I_{\dot{m}_2}(s) = I(s)$$
whose subquotients are:
$$J_t(s) = I_t(s) / I_{t-1}(s) \simeq \textup{Ind}_{\dot{P}_t \times \ddot{P}_{t+\delta}}^{\dot{H}_2 \times \ddot{H}_2} \bigg( \nu_{{}_t \dot{X}_2,s} \otimes \nu_{{}_{t+\delta} \ddot{X}_2,s} \otimes \textup{Reg}^{\alpha_t^-,\xi_t^-}(\dot{U}_t, \ddot{U}_{t+\delta}) \bigg)$$
where $\alpha_t^- : \dot{U}_t \overset{\sim}{\to} \ddot{U}_t$ is induced by the anti-isometry $\alpha^-$ and $\xi_t^-$ is a genuine character of $\dot{U}_t^{\alpha_t^-} = \{ ((\dot{u}_2,\dot{\lambda}),(\dot{u}_2^{\alpha_t^-}, \ddot{\lambda})) \}$ determined by $I(s)$. We do not need to determine $\xi_t^-$ explicitly, but we remark that:
$$\xi_t^- ((\dot{u}_2,\dot{\lambda}),(\dot{u}_2^{\alpha_t} , \ddot{\lambda})) = \pm \dot{\lambda} \ \ddot{\lambda} \ \Omega_{1,\textup{det}_F(\dot{u}_2)^s}.$$ \end{theo}

\begin{proof} The map $g \mapsto g^{-1}(X_2)$ induces a homeomorphism $P(X_2) \backslash U(V_2) \simeq \Omega(V_2)$ between cosets and Lagrangians. There is a well-known open stratification \cite[Cor 2.1.0.21]{trias_thesis} on $P(X_2) \backslash U(V_2)$ for the right-action of $U(\dot{V}_2) \times U(\ddot{V}_2)$ defined by:
$$P(X_2) \backslash U(V_2) = \bigsqcup_{t=0}^{\dot{m}_2} \mathcal{O}_t$$
where $\mathcal{O}_t = \{ Y \in \Omega(V_2) \ | \ \textup{dim}( Y \cap \dot{V}_2) = t \} = \{ Y \in \Omega(V_2) \ | \ \textup{dim}( Y \cap \ddot{V}_2) = t + \delta_m \}$.

Let $(\dot{X}_2^t)^\perp$ be the orthogonal of $\dot{X}_2^t$ in $\dot{V}_2$ and similarly for $\ddot{X}_2^{t+\delta}$ in $\ddot{V}_2$. The spaces $(\dot{X}_2^t)^\perp / \dot{X}_2^t$ and $(\ddot{X}_2^{t+\delta})^\perp / \ddot{X}_2^{t + \delta}$ are non-degenerate and the anti-isometry $\alpha^-$ induces an anti-isometry $\alpha_t^-$ between these two spaces. If we fix a dual flag $(\dot{Y}_1^t)$ of $(\dot{X}_2^t)$ then the orthogonal $\dot{V}_2^t$ of $\dot{X}_1^t \oplus \dot{Y}_1^t$ is well-defined and naturally isomorphic to $(\dot{X}_2^t)^\perp / \dot{X}_2^t$. Moreover $\alpha$ allows to define $\ddot{V}_2^{t+\delta}$ in $\ddot{V}_2$ as the image of $\dot{V}_2^t$.

Set ${\Delta}_{\alpha^-}^t = \{ (\dot{v}_2,\alpha^-(\ddot{v}_2) \ | \ v_2 \in \dot{V}_2^t \}$. Then $Y = \dot{X}_2^t \oplus \ddot{X}_2^{t+\delta} \oplus \Delta_{\alpha^-}^t$ is a Lagrangian of $V_2$. The stabiliser of $Y \in \Omega(V_2)$ in $U(\dot{V}_2) \times U(\ddot{V}_2)$ is \cite[Prop 2.1.0.20]{trias_thesis} the kernel of:
$$(\dot{p},\ddot{p}) \in P(\dot{X}_2^t) \times P(\ddot{X}_2^{t+\delta})\mapsto (\dot{u} \cdot (\alpha_t^-)^{-1} \ddot{u}^{-1} \alpha_t^-) \in U(\dot{V}_2^t)$$
where $\dot{p} \mapsto \dot{u}$ and $\ddot{p} \mapsto \ddot{u}$ are coming from projections on the isometric factor of each Levi. Denote by $\textup{St}_t$ the inverse image of this stabiliser in $\dot{H}_2 \times \ddot{H}_2$. 

Let $w_t \in U(V_2)$ be such that $Y = w_t^{-1}(X_2)$. The pre-image of $\mathcal{O}_t$ in $U(V_2)$ is the double coset $\mathcal{O}_t' = P(X_2) w_t (U(\dot{V}_2) \times U(\ddot{V}_2))$ and therefore the stabiliser of $Y$ can be rewritten as $(U(\dot{V}_2) \times U(\ddot{V}_2)) \cap (w_t^{-1} P(X_2) w_t)$. Recall $\dot{P}_t$ is the inverse image in $\dot{H}_2$ of $P(\dot{X}_2^t) \subset U(\dot{V}_2)$ and $\ddot{P}_{t+\delta}$ is the inverse image in $\ddot{H}_2$ of $P(\ddot{X}_2^{t + \delta}) \subset U(\ddot{V}_2)$. Then the stratification $U(V_2) = \sqcup \mathcal{O}_t'$ induces a filtration $(I_t(s))_t$ on $I(s)$ where $I_t(s)$ is made of functions supported on the inverse image of $\mathcal{O}_{\leq t}'$. We obtain:
$$J_t(s) = I_t(s)/I_{t-1}(s)  = \textup{ind}_{\textup{St}_t}^{\dot{H}_2 \times \ddot{H}_2}(\nu_{X_2,s} \circ \gamma_{w_t} \circ \gamma).$$
Here for $((\dot{p},\dot{\lambda}),(\ddot{p},\ddot{\lambda})) \in \textup{St}_t$ we have:
$$\gamma_{w_t} \circ  \gamma \bigg( ( \dot{p} \oplus \textup{id}_{\ddot{V}_2} , \dot{\lambda}) ( \textup{id}_{\dot{V}_2} \oplus \ddot{p} , \ddot{\lambda}) \bigg) \in \{ (w_t (\dot{p} \oplus \ddot{p})w_t^{-1},\pm \dot{\lambda} \ddot{\lambda}) \}$$
where $(\dot{p},\ddot{p})$ is in the stabiliser $(U(\dot{V}_2) \times U(\ddot{V}_2)) \cap (w_t^{-1} P(X_2) w_t)$ and is identified with $\dot{p} \oplus \ddot{p}$ via the embedding into $U(V_2)$. Therefore we can express $\nu_{X_2,s} \circ \gamma_{w_t} \circ  \gamma$ as a tensor product $\dot{\chi} \otimes \ddot{\chi} \otimes \overset{\dots}{\chi}$ of characters of $\dot{G}_t \times \ddot{G}_{t+\delta} \times \dot{U}_t^{\alpha_t^-}$. 

We can determine these characters. First of all, elements of $\dot{G}_t$ and $\ddot{G}_{t+\delta}$ commute with $(w_t,1)$. They also stabilise both $\dot{X} \oplus \ddot{X}$ and $X$ and induce norm $1$ morphisms on $X/X \cap (\dot{X} \oplus \ddot{X})$. Therefore $\gamma_{w_t} \circ  \gamma$ is the identity on these groups. We obtain $\dot{\chi} = \nu_{{}_t \dot{X}_2,s}$ and $\ddot{\chi} = \nu_{{}_{t+\delta} \ddot{X}_2,s}$ and we set $\xi_t^- = \overset{\dots}{\chi}$. Since $\textup{det}_{X_2}(w_t(\dot{u}_2 \oplus \dot{u}_2^{\alpha_t^-})w_t^{-1}) = \textup{det}_F(\dot{u}_2)$, there exists a quadratic character $\epsilon$ of $U(\dot{V}_2^t)$ such that:
$$\xi_t^- ((\dot{u}_2,\dot{\lambda}),(\dot{u}_2^{\alpha_t^-} , \ddot{\lambda})) = \epsilon(\dot{u}_2) \dot{\lambda} \ \ddot{\lambda} \ \Omega_{1,\textup{det}_F(\dot{u}_2)^s}.$$
We obtain the final form of $J_t(s)$ by transitivity of induction. \end{proof}

\begin{rem} We don't really need to make the character $\xi_t^-$ more explicit as knowing it up to a sign will be sufficient for our purpose. \end{rem}

\begin{rem} The case $\dot{m}_2 \geq \ddot{m}_2$ is similar as we can simply reverse the roles of the two groups in the induction. \end{rem}

\subsection{Filtrations on Jacquet functors (Kudla)}

Let $(U(V_1),U(V_2))$ be a dual pair of type I in $\textup{Sp}(W)$ with $W = V_1 \otimes_D V_2$. Let $(V_1^k)_{k \in \mathbb{N}}$ be the Witt tower constructed on $V_1^0$ and such that $V_1 = V_1^{m_1}$. Let $(X_1^k)_{0 \leq k \leq m_1}$ be a complete isotropic flag in $V_1$ and choose a dual complete isotropic flag $(Y_1^k)_{0 \leq k \leq m_1}$ \textit{i.e.} $X_1^k \oplus Y_1^k$ is a split $\epsilon_1$-hermitian space of dimension $k$. Its complement ${}^k V_1^{m_1}$ such that $V_1 = V_1^k \oplus {}^k V_1^{m_1} = (X_1^k \oplus Y_1^k) \oplus {}^k V_1^{m_1}$ is well-defined.

For $k \leq k'$, we set ${}^k X_1^{k'} = X_1^{k'} \cap {}^k V_1^{m_1}$. This provides a complement of $X_1^k$ in $X_1^{k'}$ \textit{i.e.} $X_1^{k'} = X_1^{k} \oplus {}^k X_1^{k'}$. In particular $X_1^k = {}^0 X_1^k$ as $X_1^0=0$. Note that the choice of a complete isotropic flag determines a set of standard parabolics; and adding the choice of a dual flag determines a set of standard Levi subgroups. In particular it gives a standard decomposition $P(X_1^k) = (\textup{GL}(X_1^k) \times U(V_1^k)) \ltimes N(X_1^k)$. Furthermore the unipotent radical sits into an exact sequence:
$$1 \to N_1(X_1^k) \to N(X_1^k) \to \textup{Hom}_D({}^k V_1^{m_1},X_1^k) \to 1$$
where $N_1(X_1^k) \subset N(X_1^k)$ is central and the quotient $\textup{Hom}_D({}^k V_1^{m_1},X_1^k)$ is abelian, so the unipotent radical $N(X_1^k)$ has a two-step filtration by abelian groups.

For $X_1' \subset X_1^k$, we denote by $Q_k(X_1')$ the parabolic subgroup of $\textup{GL}(X_1^k)$ stabilising $X_1'$. We are going to use the previous two-step filtration on $N(X_1^k)$ to compute the normalised Jacquet functor $\mathfrak{r}_{X_1^k}$ with respect to $H_1$ of the Weil representation:

\begin{theo} \label{kudla_filtration_thm} Let $0 \leq k \leq m_1$ and set $r = \textup{min}(k,m_2)$. Then $\mathfrak{r}_{X_1^k}(\omega_{m_1,m_2})$ has a filtration:
$$\mathfrak{r}_{X_1^k}(\omega_{m_1,m_2}) = R^0 \supset R^1 \supset \cdots \supset R^r \supset 0$$
whose subquotients are given by:
$$J_t = R^t / R^{t+1} \simeq \mathfrak{i}_{Q_{{}^t X_1^k} \times {}^k H_1^{m_1} \times P_t^{m_2}}^{G_{X_1^k} \phantom{k} \times {}^k H_1^{m_1} \times H_2}(\delta_{{}^t X_1^k,n_2} |\textup{det}_{{}^t X_1^k}|^{\frac{s+k-t}{2}}  \otimes \textup{Reg}^{\alpha_t,\xi_t}(G_1^t,G_2^t) \otimes \omega_{m_1-k,m_2-t}).$$
where $s=n_2-n_1+\eta_1$ with $\eta_1$ defined in Proposition \ref{modulus_character_parabolic_prop} and:
\begin{itemize}[label=$\bullet$]
\item $Q_{{}^t X_1^k}$ is the parabolic of $G_{X_1^k}$ that is the inverse image of $Q_k({}^t X_1^k)$ in $H_1$;
\item the genuine character $\delta_{{}^t X_1^k,n_2}$ of $G_{{}^t X_1^k}$ is $\nu_{{}^t X_1^k,n_2} |\textup{det}_{{}^t X_1^k}|^{-\frac{n_2}{2}}$, which only depends on the parity of $n_2$;
\item $\alpha_t : X_1^k = X_1^t \oplus {}^t X_1^k \to X_1^t \overset{\sim}{\to} X_2^t$ induces $\textup{GL}(X_1^t) \simeq \textup{GL}(X_2^t)$ and $\xi_t$ is the genuine character of $G_1^{\alpha_t} = \{ ((g_1,\lambda),(g_1^{\alpha_t},\mu)) \}$ given on $((g_1,1),(g_1^{\alpha_t},1))$ by:
\begin{itemize}[label=$\circ$]
\item $1$, if $n_1$ and $n_2$ are both even;
\item $(-1,\textup{det}_{X_1^t}(g_1))_F$, if $n_1$ and $n_2$ are both odd;
\item $\delta_{X_1^k,1}(g_1,1) =\omega_{1,\textup{det}_{X_1^t}(g_1)} \in \mu_4(R)$, otherwise. 
\end{itemize}
\end{itemize} \end{theo}

\begin{proof} Let $(\omega,S)$ be the mixed model of Section \ref{witt_towers_dual_pairs_type_I_section} associated to the decomposition:
$$W = V_1 \otimes_D V_2 =  X_1^k \otimes_D V_2  + {}^k V_1^{m_1} \otimes_D V_2 + Y_1^k \otimes_D V_2 =  X_k  +  W_{X_k}  +  Y_k.$$
Let $(\rho_\psi^k,S_k)$ be the metaplectic representation of $H(W_{X_k})$. We identify $C_c^\infty(Y_k) \otimes S_k$ with $C_c^\infty(Y_k,S_k)$, as well as $Y_k$ with $\textup{Hom}_D(X_1^k,V_2)$ via the symplectic form. As a result:
$$S \simeq C_c^\infty(\textup{Hom}_D(X_1^k,V_2),S_k).$$

\begin{lem} \label{kudla_filtration_intermediate_lem} We have an isomorphism of $(P_k^{m_1} \times H_2)$-modules:
$$S_{N_1(X_1^k)} \overset{\sim}{\longrightarrow} C_c^\infty(Z,S_k)$$
where $Z$ is the closed subset of $\textup{Hom}_D(X_1^k,V_2)$ made of morphisms whose image is a totally isotropic subspace of $V_2$.

Furthermore, the action of $P_k^{m_1} \times H_2$ on $Z$ has finite regular strata induced by the rank $Z = \sqcup_{t=0}^r Z_t$, where $Z_t$ are morphisms of rank $t$. It induces a filtration:
$$C_c^\infty(Z,S_k) = T^{(0)} \supset T^{(1)} \supset \cdots \supset T^{(r)} \supset 0$$
where the subquotients $T_t = T^{(t)} / T^{(t+1)}$ are isomorphic to:
$$\textup{ind}_{Q_t^k \times {}^k H_1^{m_1} \times P_t^{m_2}}^{G_k^{m_1} \times {}^k H_1^{m_1} \times H_2}(\nu_{{}^t X_1^k,n_2} \otimes \nu_{X_1^t,n_2} \otimes C_c^\infty(\textup{Iso}_D(X_1^t,X_2^t)) \otimes \omega_{m_1-k,m_2}).$$ \end{lem}

\begin{proof} We can apply \cite[Chap 3, V.1]{mvw} which relates coinvariants with the locus $Z$ where the action on the fibre is trivial. More prosaically, the group $N_1(X_1^k)$ acts trivially on $\textup{Hom}_D(X_1^k,V_2)$ and for each $x \in \textup{Hom}_D(X_1^k,V_2)$ the action of the group $N_1(X_1^k)$ on the fibre $S_k$ is given by $n(s) \mapsto \psi(\frac{\langle x \circ s , x \rangle}{2})$ where we have used the identification between $\mathfrak{n}_1 = \{s \in \textup{Hom}_D(Y_1^k,X_1^k) \ | \ s^* = -s \}$ and $N_1(X_1^k)= \{ \textup{Id} + s \ | \ s \in \mathfrak{n}_1 \}$. Therefore the previous locus will be $Z = \{ x \in \textup{Hom}_D(X_1^k,V_2) \ | \ \psi(\frac{\langle x \circ s , x \rangle}{2}) = 1 \textup{ for all } s \in \mathfrak{n}_1\}$. The latter condition implies that $x \in Z$ if and only if its image is totally isotropic. This finishes the first part of the statement.

Regarding the filtration, the stratification on $Z$ clearly has finite regular strata because the action is transitive on each stratum. The action of $P_t^{m_1} \times H_2$ factors through $\textup{GL}(X_1^k) \times U(V_2)$. Choose $x_t \in \textup{Hom}_D(X_1^k,V_2)$ with kernel ${}^t X_1^k$ and image $X_2^t$ and assume $x_t$ is adapted to the flags \textit{i.e.} $x_t(X_1^{t'}) \subset X_2^{t'}$ for all $t'$. As already mentioned, the orbit of $x_t$ is $Z_t$ and its stabiliser is contained in $Q_k({}^t X_1^k) \times P(X_2^t)$. The action on $x_t$ factors through $\textup{GL}(X_1^t) \times \textup{GL}(X_2^t)$ via $(\alpha,\beta) \cdot x_t = \beta \circ x_t \circ \alpha^{-1}$. The stabiliser $\textup{St}_{x_t} = \{ (\alpha,\beta) \ | \ \beta x_t \alpha^{-1} = x_t\}$ is isomorphic to both $\textup{GL}(X_1^t)$ and $\textup{GL}(X_2^t)$ via the first and second projections. Therefore: 
$$C_c^\infty(Z_t,S_k) \simeq \textup{ind}_{\textup{St}_{x_t}}^{M_k^{m_1} \times H_2}(S_k|_{x_t})$$
where $(S_k|_{x_t},\rho) \in \textup{Rep}_R(\textup{St}_{x_t})$ is the representation on the fibre $S_k$ at $x_t$. In more concrete terms, it means $(s \cdot f)(x_t) = \rho(s) f(x_t)$ for $s \in \textup{St}_{x_t}$ and $f \in C_c^\infty(Z_t,S_k)$.

Therefore $T_t \simeq \textup{ind}_{\textup{St}_t}^{M_k^{m_1} \times H_2}(S_k|_{x_t})$. In the mixed Schr\"odinger model $(\omega,S)$, the group $M_k^{m_1} \times H_2$ acts directly on $S_k$ via $\nu_{X_1^k,n_2} \otimes \omega_{m_1-k,m_2}$. Moreover $h \in \textup{Hom}_D({}^k V_1^{m_1},X_1^k)$ acts on $S_k|_{x_t}$ at $x_t \in \textup{Hom}_D(X_1^k,V_2) \simeq Y_k$, so the action is:
$$h \mapsto \rho_\psi^k(- (h \otimes_D \textup{id}_{V_2})^*(x_t),0)) = \rho_\psi^k((-x_t \circ h,0)).$$
Because $S_k|_{x_t}$ extends as the Weil representation $\omega_{m_1-k,m_2}$ on the group $P_k^{m_1} \times H_2$, it extends in particular to $Q_t^k \times {}^k H_1^{m_1} \times P_t^{m_2}$ where $Q_t^k = Q_k({}^t X_1^k) \times R^\times \subset G_k^{m_1}$. We obtain by transitivity of induction:
$$T_t \simeq \textup{ind}_{Q_t^k \times H_1^{m_1-k} \times P_t^{m_2}}^{G_k^{m_1} \times {}^k H_1^{m_1} \times H_2} ( \textup{ind}_{\textup{St}_{x_t}}^{Q_t^k \times G_t^{m_2}}(1) \otimes S_k).$$
Eventually we get the following induced representation:
$$T_t \simeq  \textup{ind}_{Q_t^k \times H_1^{m_1-k} \times P_t^{m_2}}^{G_k^{m_1} \times H_1^{m_1-k} \times H_2}( \nu_{{}^t X_1^k,n_2} \otimes \nu_{X_1^t,n_2} \otimes C_c^\infty(\textup{Iso}_D(X_1^t,X_2^t)) \otimes \omega_{m_1-k,m_2})$$
where the induction with respect to $H_2$ is not necessarily the parabolic induction because the unipotent radical of $P_t^{m_2}$ acts in general non-trivially on the Weil representation. \end{proof}

If $N_1(X_1^k)=N(X_1^k)$, then the filtration of the lemma yields the result. Otherwise we can realise $T_t$ thanks to a new mixed Schr\"odinger model to determine $(\omega_{m_1,m_2})_{N(X_1^k)}$. We realise $S_k$ as the mixed Schr\"odinger model induced by $V_2 = X_2^t + {}^t V_2^{m_2} + Y_2^t$ and associated to the decomposition:
$$W_{X_k} = {}^k V_1^{m_1} \otimes_D X_2^t + {}^k V_1^{m_1} \otimes_D {}^t V_2^{m_2} + {}^k V_1^{m_1} \otimes_D Y_2^t = X_{k,t} + W_{X_{k,t}} + Y_{k,t}.$$  
As earlier, we can realise $(\omega_{m_1-k,m_2},S_k)$ as the mixed model $C_c^\infty(\textup{Hom}_D(X_2^t,{}^k V_1^{m_1}),S_{k,t})$ where $(\rho_\psi^{k,t},S_{k,t})$ is the metaplectic representation of $H(W_{X_{k,t}})$. On this mixed model, the action of $h \in \textup{Hom}_D({}^k V_1^{m_1},X_1^k) = N(X_1^k) / N_1(X_1^k)$ at $y \in \textup{Hom}_D(X_2^t,{}^k V_1^{m_1}) = Y_{k,t}$ is given by $h \mapsto \psi(\langle y , x_t \circ h \rangle)$ where our previous $x_t$ is seen as an element of $\textup{Hom}_D(X_1^k,X_2^t)$ and where ${}^k V_1^{m_1} \otimes X_2^t = X_{k,t}$ is identified to $\textup{Hom}_D({}^k V_1^{m_1},X_2^t)$ via the symplectic product. The locus of $y$ where this character of $\textup{Hom}_D({}^k V_1^{m_1},X_1^k))$ is trivial is $\{ 0 \}$. Therefore:
$$(\omega_{m_1-k,m_2})_{\textup{Hom}_D({}^k V_1^{m_1},X_1^k)} = C_c^\infty(\{0\},S_{k,t}) = S_{k,t} = \nu_{X_2^t,n_1-2k} \otimes \omega_{m_1-k,m_2-t}$$
and together with Lemma \ref{kudla_filtration_intermediate_lem}, this implies:
$$(\omega_{m_1,m_2})_{N(X_1^k)} \simeq \textup{ind}_{Q_t^k \times {}^k H_1^{m_1} \times P_t^{m_2}}^{G_k^{m_1} \times {}^k H_1^{m_1} \times H_2}( C_c^\infty(\textup{Iso}_D(X_1^t,X_2^t)) \otimes \nu' \otimes \omega_{m_1-k,m_2-t})$$
where $\nu' = \nu_{{}^t X_1^k,n_2} \otimes \nu_{X_1^t,n_2} \otimes \nu_{X_2^t,n_1-2k}$. Normalising the $N(X_1^k)$-coinvariants and the parabolic induction, thanks to Proposition \ref{modulus_character_parabolic_prop}, replaces $\nu'$ by $\nu = {}^t \nu_{1,s}^k \otimes \nu_1^t \otimes \nu_2^t$, where:
\begin{itemize}[label=$\bullet$]
\item ${}^t \nu_{1,s}^k = \nu_{{}^t X_1^k,n_2}|\textup{det}_{{}^t X_1^k}|^{\frac{-n_1+\eta_1}{2}} |\textup{det}_{{}^t X_1^k}|^{\frac{k-t}{2}}$;
\item $\nu_1^t = \nu_{X_1^t,n_2} |\textup{det}_{X_1^t}|^{\frac{-n_1+2k-t+\eta_1}{2}}$;
\item $\nu_2^t = \nu_{X_2^t,n_1-2k} |\textup{det}_{X_2^t}|^{\frac{-n_2+t-\eta_1}{2}}$.
\end{itemize}
In particular $\nu_1^t C_c^\infty(\textup{Iso}_D(X_1^t,X_2^t)) \nu_2^t = \textup{Reg}^{\alpha_t,\xi_t}(G_1^t,G_2^t)$ where $\xi_t= \nu_1^t \otimes \nu_2^t$ only depends on the parities of $n_1$ and $n_2$ because $\nu_{X,a+2b} = \nu_{X,a} |\textup{det}_X|^b$ for any $X$. \end{proof}

\subsection{Generalised doubling method and seesaw identity} \label{generalised-doubling-method-sec}

\paragraph{Generalised doubling method.} Let $\dot{V}_2^0$ and $\ddot{V}_2^0$ be two anisotropic $\epsilon_2$-hermitian spaces in opposite Witt series \textit{i.e.} $\dot{V}_2^0 \oplus \ddot{V}_2^0$ is split. Let $\dot{V}_2$ and $\ddot{V}_2$ be in the same Witt series as $\dot{V}_2^0$ and $\ddot{V}_2^0$ respectively. We set $V_2=\dot{V}_2 \oplus \ddot{V}_2$ that is split by assumption. Tensoring with $V_1$ we can form three symplectic spaces $\dot{W}, \ddot{W}, W$ and their associated dual pairs:
\begin{itemize}[label=$\bullet$]
\item $(\dot{H}_1,\dot{H}_2)$ the lift of $(U(V_1),U(\dot{V}_2))$ to $\textup{Mp}(\dot{W})$;
\item $(\ddot{H}_1,\ddot{H}_2)$ the lift of $(U(V_1),U(\ddot{V}_2))$ to $\textup{Mp}(\ddot{W})$;
\item $(H_1^\triangle,H_2)$ the lift of $(U(V_1),U(V_2))$ to $\textup{Mp}(W)$.
\end{itemize}
Because of the so-called diagonal embedding $\textup{Mp}(\dot{W}) \times \textup{Mp}(\ddot{W}) \to \textup{Mp}(W)$ there are embeddings between these groups such as:
\begin{itemize}
\item $H_1^\triangle$ is diagonally contained in the image of $\dot{H}_1 \times \ddot{H}_1$ in $\textup{Mp}(W)$ namely:
\begin{itemize}[label=$\circ$]
\item $\dot{H}_1 \times \ddot{H}_1 \to \textup{Mp}(W)$ factors through the quotient by $\{((\textup{id}_{V_1},\lambda),(\textup{id}_{V_1},\lambda^{-1}))\}$;
\item the image of $\{((u_1,\lambda),(u_1,\lambda'))\}$ in $\textup{Mp}(W)$ is naturally isomorphic to $H_1^\triangle$.
\end{itemize}
\item the image of $\dot{H}_2 \times \ddot{H}_2 \to \textup{Mp}(W)$ lies in $H_2$. \end{itemize}
In the literature, the situation is usually summed up in a see-saw diagram where vertical arrows accounts for embeddings and diagonal ones for being a dual pair:
$$\xymatrix{ \dot{H}_1 \times \ddot{H}_1 \ar@{-}[dr] \ar@{-}[d] & H_2 \ar@{-}[d] \\
H_1^\triangle \ar@{-}[ur] & \dot{H}_2 \times \ddot{H}_2
}$$

Because $H_1^\triangle$ is known to be split, this proves in particular that the cocycles defining $\dot{H}_1$ and $\ddot{H}_1$ are cohomologous. As a result, these two groups are isomorphic to a common group we denote $H_1$. Let $\chi_1$ be the character of $H_1^\triangle$ associated to the coinvariants for the geometric action of $U(V_1)$ \textit{i.e.} $\chi_1(u_1,\lambda) = \lambda \lambda_{u_1}^{-1}$ where $\iota : u_1 \in U(V_1) \mapsto (u_1,\lambda_{u_1}) \in H_1^\triangle$ is the geometric embedding. For $\pi_1 \in \textup{Irr}_R(H_1)$, we define the $\chi_1$-contragredient of $\pi_1$ as the unique $\pi_1^{[\chi_1]} \in \textup{Irr}_R(H_1)$ such that $\textup{Hom}_{H_1^\triangle}(\pi_1 \otimes \pi_1^{[\chi_1]},\chi_1) \neq 0$.

\begin{lem} \label{generalised_doubling_theta_lem} We have an $(\dot{H}_2 \times \ddot{H}_2)$-surjection:
$$\Theta(\chi_1,\dot{V}_2 \oplus \ddot{V}_2) \twoheadrightarrow \Theta(\pi_1,\dot{V}_2) \otimes \Theta(\pi_1^{[\chi_1]},\ddot{V}_2).$$ \end{lem}

\begin{proof} We have $\omega_{V_1,\dot{V}_2 \oplus \ddot{V}_2} \simeq \omega_{V_1,\dot{V}_2} \otimes \omega_{V_1,\ddot{V}_2}$. On the right-hand side we can consider the biggest $(\pi_1 \otimes \pi_1^{[\chi_1]})$-isotypic quotient:
$$\omega_{V_1,\dot{V}_2} \otimes \omega_{V_1,\ddot{V}_2} \twoheadrightarrow (\pi_1 \otimes  \Theta(\pi_1,\dot{V}_2)) \otimes (\pi_1^{[\chi_1]} \otimes \Theta(\pi_1^{[\chi_1]},\ddot{V}_2)).$$
and composing with the $\chi_1$-trace $\pi_1 \otimes \pi_1^{[\chi_1]} \to \chi_1$ we obtain a surjection:
$$\omega_{V_1,\dot{V}_2} \otimes \omega_{V_1,\ddot{V}_2} \twoheadrightarrow \Theta(\pi_1,\dot{V}_2) \otimes \Theta(\pi_1^{[\chi_1]},\ddot{V}_2).$$
As this map is $H_1^\triangle$-equivariant it factors through $\Theta(\chi_1,\dot{V}_2 \oplus \ddot{V}_2)$. \end{proof}

\paragraph{Seesaw identity.} The seesaw identity is deeply related to the generalised doubling method and the surjection we obtained. It is obtained by factorising in two different ways the coinvariants for $\chi_1 \otimes (\dot{\pi}_2 \otimes \ddot{\pi}_2)$ for $\dot{\pi}_2 \in \textup{Irr}_R(\dot{H}_2)$ and $\ddot{\pi}_2 \in \textup{Irr}_R(\ddot{H}_2)$.

\begin{lem} \label{seesaw_identity_isotypic_chi1_pi2_lem} $\Theta(\chi_1,\dot{V}_2 \oplus \ddot{V}_2)_{\dot{\pi}_2 \otimes \ddot{\pi}_2} = (\Theta(\dot{\pi}_2,V_1) \otimes \Theta(\ddot{\pi}_2,V_1))_{\chi_1} \otimes (\dot{\pi}_2 \otimes \ddot{\pi}_2)$. \end{lem}

We can rephrase it as $(\Theta(\chi_1,\dot{V}_2 \oplus \ddot{V}_2) \otimes \dot{\pi}_2^\vee \otimes \ddot{\pi}_2^\vee)_{U(\dot{V}_2) \times U(\ddot{V}_2)} = (\Theta(\dot{\pi}_2,V_1) \otimes \Theta(\ddot{\pi}_2,V_1))_{\chi_1}$, which is an equality of vector spaces controlling the multiplicities of $\dot{\pi}_2 \otimes \ddot{\pi}_2$ in each term.

\begin{rem} The literature usually consider homomorphism spaces instead of isotypic quotients \textit{e.g.} $\textup{Hom}_{\dot{H}_2 \times \ddot{H}_2}(\Theta(\chi_1,\dot{V}_2 \oplus \ddot{V}_2),\dot{\pi}_2 \otimes \ddot{\pi}_2)$ instead of $\Theta(\chi_1,\dot{V}_2 \oplus \ddot{V}_2)_{\dot{\pi}_2 \otimes \ddot{\pi}_2}$. Note that factorising the former leads to $\textup{Hom}_R((\Theta(\chi_1,\dot{V}_2 \oplus \ddot{V}_2) \otimes \dot{\pi}_2^\vee \otimes \ddot{\pi}_2^\vee)_{U(\dot{V}_2) \times U(\ddot{V}_2)}, 1)$ which is the algebraic dual of linear forms on $(\Theta(\chi_1,\dot{V}_2 \oplus \ddot{V}_2) \otimes \dot{\pi}_2^\vee \otimes \ddot{\pi}_2^\vee)_{U(\dot{V}_2) \times U(\ddot{V}_2)}$. Since these vector spaces are infinite dimensional, our way of writing these identities is more intrinsic. \end{rem}

\begin{rem} \label{doubling_method_from_minus_to_psi_inverse_rem} Let $(U(V_1),U(V_2))$ be a type I dual pair in $\textup{Sp}(W)$. We set $\dot{V}_2 = V_2^{m_2}$ and $\ddot{V}_2 = - V_2^{m_2'}$ and $V_2^\square = V_2^{m_2} \oplus (-V_2^{m_2'})$. We rewrite the seesaw diagram above as:
$$\xymatrix{ H_1 \times H_1^- \ar@{-}[dr] \ar@{-}[d] & H_2^\square \ar@{-}[d] \\
H_1^\triangle \ar@{-}[ur] & H_2 \times H_2^{-,m_2'}
}.$$
Using Remark \ref{pullback_for_-1_remark_annex} in the appendix, the pullback of the Weil representaion $\omega_{\psi,H_1^-,H_2^{-,m_2'}}$ via $\gamma_-$ is $\omega_{\psi^{-1},H_1,H_2^{m_2'}}$. Thanks to this pullback, we obtain the identity:
$$\Theta_\psi(\chi_1,V_2^\square)_{\pi_2 \otimes \pi_2'} = ( \Theta_\psi(\pi_2,V_1) \otimes \Theta_{\psi^{-1}}(\pi_2',V_1) )_{\chi_1} \otimes \pi_2 \otimes \pi_2'.$$
Note that $\pi_2'$ is a representation of $H_2^{m_2'}$ and the character is $\psi^{-1}$, whereas in the see-saw identity of Lemma \ref{seesaw_identity_isotypic_chi1_pi2_lem}, the group is $H_2^{-,m_2'}$ and the character is $\psi$. \end{rem}

\section{Consequences of Rallis' and Kudla's filtrations} \label{Rallis_Kudla_filtrations_consequences_sec}

We denote the Weil representation $\omega_{\psi,H_1^{m_1},H_2^{m_2}}$ in Witt towers by $\omega_{m_1,m_2}$ as a shorthand. Recall $(\omega_{m_1,m_2})_{\pi_1} \simeq \pi_1 \otimes \Theta_\psi(\pi_1,V_2^{m_2})$ for $\pi_1 \in \textup{Irr}_R(H_1^{m_1})$. In this section, we study the Weil representation and its theta lifts varying in Witt towers \textit{i.e.} we fix $m_1$ whereas $m_2$ varies. So we wil simply write $H_1$ instead of $H_1^{m_1}$ when the context is clear.

\subsection{First occurrence index}

Let $V_1$ be an $\epsilon_1$-hermitian space of dimension $n_1$ and Witt index $m_1$. Even though $m_1$ is fixed, it is sometimes useful to write $V_1 = V_1^{m_1}$ and $H_1 = H_1^{m_1}$. Let $V_2^0$ be an $\epsilon_2$-hermitian anisotropic space (possibly zero) and $V_2^{m_2}$ the $\epsilon_2$-hermitian space of Witt index $m_2$ in the same series as $V_2^0$. We suppose $\epsilon _1 \epsilon_2 = - 1$ and recall that $(U(V_1),U(V_2^{m_2}))$ is a dual pair in $\textup{Sp}(V_1 \otimes V_2^{m_2})$.

\begin{theo} \label{first_occurrence_index_thm} Let $\pi_1 \in \textup{Irr}_R(H_1)$. There exists an integer $m_2(\pi_1) \leq n_1$ such that:
$$\Theta(\pi_1,W_2^{m_2}) \neq 0 \Leftrightarrow m_2 \geq m_2(\pi_1).$$ \end{theo}

\begin{proof} Let $V_1 \otimes V_2^{m_2} = V_1 \otimes X_2^k + V_1 \otimes V_2^{m_2-k} + V_1 \otimes Y_2^k$. The mixed Schr\"odinger model is $\omega_{m_1,m_2} = C_c^\infty(\textup{Hom}(X_2^k,V_1)) \otimes \omega_{m_1,m_2-k}$. Consider the natural action of $U(V_1)$ on $\textup{Hom}(X_2^k,V_1)$. The orbit associated to an element $x \in \textup{Hom}(X_2^k,V_1)$ such that its image is non-degenerate is closed. Through restriction to $U(V_1) \cdot x$ we get an $H_1$-surjection:
$$\omega_{m_1,m_2} \twoheadrightarrow \textup{ind}_{\textup{St}_x}^{U(V_1)}(1) \otimes \omega_{m_1,m_2-k}.$$
Choosing $x=0$, we get an $H_1$-surjection $\omega_{m_1,m_2} \twoheadrightarrow \omega_{m_1,m_2-k}$. Therefore: 
$$(\omega_{m_1,m_2-k})_{\pi_1} \neq 0 \Rightarrow (\omega_{m_1,m_2})_{\pi_1} \neq 0.$$
It remains to show it is non-zero for some $m_2 \in \mathbb{N}$. Actually we have $(\omega_{m_1,n_1})_{\pi_1} \neq 0$ for all $\pi_1$. Indeed if $m_2 = n_1$, we can choose $x$ with image a hermitian basis of $V_1$ and take $k=m_2$, therefore the stabiliser of $x$ in $U(V_1)$ is trivial and we obtain an $H_1$-surjection: 
$$\omega_{m_1,n_1} \twoheadrightarrow \textup{ind}_1^{U(V_1)}(1) \otimes \omega_{m_1,0} \simeq \textup{ind}_{R^\times}^{H_1}(\omega_{m_1,0}|_1)$$
where $\omega_{m_1,0}|_1$ is simply a vector space. But the right-hand side is just copies of the regular representation of $H_1$ which admits any irreducible representation as a quotient. \end{proof}
 
\subsection{Cuspidal representations $\pi_1$}

In the cuspidal case, we can relate Jacquet functors of big theta lifts in a given tower and show the first occurrence index is a cuspidal representation. We recall that $s=n_2-n_1+\eta_1$ where $\eta_1$ is defined in Proposition \ref{modulus_character_parabolic_prop}. To simplify notations we set from now on:
$$\nu_{k,s} = \delta_{X_2^k,n_1} |\textup{det}_{X_2^k}|^{\frac{-s + k}{2}}.$$
This character also is $\nu_{X_2^k,n_1} |\textup{det}_{X_2^k}|^{\frac{-(n_2 - k - \eta_2)}{2}}$ \textit{i.e.} it is made of $\nu_{X_2^k,n_1}$ and some modulus character coming from normalisation as $\eta_2 = -\eta_1$.

\begin{lem} \label{cuspidal_lift_jacquet_functor_inductive_formula_lem} Set $s=n_2 - n_1 + \eta_1$. Let $\pi_1 \in \textup{Irr}_R^{\textup{cusp}}(H_1)$. The Jacquet functor $\mathfrak{r}_{X_2^k}$ induces a natural $(G_k \times H_2^{m_2-k})$-isomorphism:
$$\mathfrak{r}_{X_2^k}(\Theta(\pi_1,V_2^{m_2})) \simeq \nu_{k,s} \otimes \Theta(\pi_1,V_2^{m_2-k}).$$
In particular $\Theta(\pi_1,V_2^{m_2(\pi_1)})$ is cuspidal. \end{lem}

\begin{proof} First of all, we have $\mathfrak{r}_{X_2^k}((\omega_{m_1,m_2}))_{\pi_1} = \mathfrak{r}_{X_2^k}((\omega_{m_1,m_2})_{\pi_1})$ because the Jacquet functor $\mathfrak{r}_{X_2^k}$ is exact and the actions of $H_1$ and $H_2$ commmute. Considering the filtration $(R^t)$ of $R^0 = \mathfrak{r}_{X_2^k}(\omega_{m_1,m_2})$ with subquotients $(J_t)$ from Theorem \ref{kudla_filtration_thm}, we obtain: 
$$(J_t)_{\pi_1} = (\mathfrak{i}_{P_t^{m_1}}^{H_1} (-))_{\pi_1} = 0 \textup{ for all } t > 0.$$
So $(R^1)_{\pi_1} = 0$. Applying this to the exact sequence $0 \to R^1 \to R^0 \to J_0 \to 0$ the biggest $\pi_1$-isotypic quotient yields the result $(R^0)_{\pi_1} \simeq (J_0)_{\pi_1} = \nu_{k,s} \otimes (\omega_{m_1,m_2-k})_{\pi_1}$.

For the last part of the statement, note that it is enough to be killed by Jacquet functors of all proper maximal standard parabolic subgroups to be cuspidal. By the existence of the first occurrence index, we have $\Theta(\pi_1,V_2^{m_2(\pi_1)-k})=0$ for all $k > 0$. Therefore $\mathfrak{r}_{X_2^k}(\Theta(\pi_1,V_2^{m_2(\pi_1)})=0$ for all $k >0$. Because $P(X_2^k)$ runs over all maximal standard parabolic subgroups, the first occurrence is cuspidal. \end{proof}

In the following situation, the first occurrence index is irreducible and we can tell that $\Theta(\pi_1,V_2^{m_2})$ has no irreducible supercuspidal subquotients and no irreducible cuspidal subrepresentations unless $m_2=m_2(\pi_1)$. Normalising inductions in Rallis' filtration, the character appearing instead of $\nu_{{}_{t+\delta} \ddot{X}_2,s} $ is $\nu_{t+\delta,s}$. We recall the notation $\textup{Irr}_R^{p\&i}(G)$ from Section \ref{projective_vs_injective_objects_sec} for isomorphism classes of irreducible projective and injective representations.

In the notations of Section \ref{generalised-doubling-method-sec}, we will consider two opposite Witt towers $\dot{V}_2 = V_2^{m_2}$ and $\ddot{V}_2 = - V_2^{m_2'}$ along with the associated lifts of dual pairs $H_1$, $H_2^{m_2}$ and $H_2^{-,m_2'}$. In particular the identity map is an anti-isometry between $V_2^{m_2}$ and $-V_2^{m_2}$. In this case we write $\textup{Reg}^{\alpha_0^-,\xi_0^-}(H_2^{m_2},H_2^{-,m_2})$ from Theorem \ref{filtration_degenerate_principal_series_thm} is obtained with $\alpha_0$ the identity map, so we delete the reference to $\alpha_0^-$ as this map is canonical. Recall from Section \ref{generalised-regular-reps-sec} that $\textup{Reg}^{\xi_0^-}(H_2^{m_2},H_2^{-,m_2})$ has irreducible quotients of multiplicity one and they are of the form $\pi_2 \otimes \pi_2^{[\xi_0^-]}$ where $\pi_2 \in \textup{Irr}_R(H_2^{m_2})$, or equivalently, $(\pi_2^-)^{[\xi_0^-]} \otimes \pi_2^-$ where $\pi_2^- \in \textup{Irr}_R(H_2^{-,m_2})$. All our representations are assumed to be genuine.

\begin{prop} \label{proj_inj_theta_quotient_prop} Let $\pi_1 \in \textup{Irr}_R^{\textup{cusp}}(H_1)$.
\begin{enumerate}[label=\textup{\alph*)}]
\item If $\Theta(\pi_1,V_2^{m_2^c})$ has a subquotient $\pi_2 \in \textup{Irr}_R^{p\&i}(H_2^{m_2^c})$:
\begin{itemize}[label=$\bullet$]
\item $m_2^c \leq m_2^-(\pi_1^{[\chi_1]})$ ;
\item $\Theta(\pi_1^{[\chi_1]},-V_2^{m_2})$ is a subquotient of $\mathfrak{i}_{P_{\delta}^-}^{H_2^-} ( \nu_{\delta,s} \otimes \pi_2^{[\xi_0^-]})$ where $\delta = m_2 - m_2^c$.
\end{itemize}

\item If furthermore $\Theta(\pi_1^{[\chi_1]},-V_2^{m_2^{-,c}})$ has a subquotient $\pi_2^- \in \textup{Irr}_R^{p\&i}(H_2^{m_2^{-,c}})$:
\begin{itemize}[label=$\bullet$]
\item $m_2^c = m_2^{-,c} = m_2(\pi_1) = m_2^-(\pi_1^{[\chi_1]})$;
\item $\Theta(\pi_1,V_2^{m_2^c}) = \pi_2 = (\pi_2^-)^{[\xi_0^-]} = \Theta(\pi_1^{[\chi_1]},-V_2^{m_2^c})^{[\xi_0^-]}$;
\item $\Theta(\pi_1,V_2^{m_2})$ is a subquotient of $\mathfrak{i}_{P_{\delta}}^{H_2^{m_2}} ( \nu_{\delta,s} \otimes \pi_2)$ where $\delta = m_2 - m_2^c$;
\item in particular $\Theta(\pi_1,V_2^{m_2})$ has no supercuspidal constituents if $m_2 \neq m_2^c$;
\end{itemize}
\item If furthermore $\Theta(\pi_2^-,V_1)=\pi_1^{[\chi_1]}$:
\begin{itemize}[label=$\bullet$]
\item $\Theta(\pi_1,V_2^{m_2})$ is a submodule of $\mathfrak{i}_{P_{\delta}}^{H_2} ( \nu_{\delta,s} \otimes \pi_2)$ where $\delta = m_2 - m_2^c$;
\item in particular $\Theta(\pi_1,V_2^{m_2})$ has no irreducible cuspidal submodule if $m_2 \neq m_2^c$.
\end{itemize}  \end{enumerate} \end{prop}

\begin{proof} The proof of \cite[chap 2 IV.9 Cor]{mvw} is wrong since there is a confusion between contragredient with respect to different groups: if $H \subset G$ is a subgroup and $V \in \textup{Rep}_R(G)$, then it is not true that $(V^\vee)|_H$ is the contragredient with respect to $H$. We follow instead a strategy similar to \cite{kudla_invent}, except that, here, we have to carefully unwrap which properties about cuspidals or their first occurrences are being used.

\noindent a) The generalised doubling method from Lemma \ref{generalised_doubling_theta_lem} provides a surjection: 
$$\Theta(\chi_1,V_2^{m_2^c} \oplus (-V_2^{m_2})) \twoheadrightarrow \Theta(\pi_1,V_2^{m_2^c}) \otimes \Theta(\pi_1^{[\chi_1]},-V_2^{m_2}).$$
By hypothesis $\pi_2$ is projective-injective and occurs as a subquotient of $\Theta(\pi_1,V_2^{m_2^c})$, so it occurs as a quotient as well. Therefore:
$$\Theta(\chi_1,V_2^{m_2^c} \oplus (-V_2^{m_2}))_{\pi_2} \twoheadrightarrow \pi_2 \otimes \Theta(\pi_1^{[\chi_1]},-V_2^{m_2}).$$
In particular $\Theta(\chi_1,V_2^{m_2^c} \oplus (-V_2^{m_2}))_{\pi_2} \neq 0$ for all $m_2^-(\pi_1^{[\chi_1]}) \leq m_2$. So $I(n_1)_{\pi_2} \neq 0$ for all $m_2^-(\pi_1^{[\chi_1]}) \leq m_2$ as $\Theta(\chi_1,V_2^{m_2^c} \oplus (-V_2^{m_2}))_{\pi_2} \subset I(n_1)_{\pi_2}$ by exactness -- this also shows the statement about $\Theta(\pi_1^{[\chi_1]},-V_2^{m_2})$. Note that $I(n_1)_{\pi_2} = I_0(n_2)_{\pi_2}$, where $J_0(n_1)=I_0(n_1)$ is a subquotient of the filtration of $I(n_1)$ in Theorem \ref{filtration_degenerate_principal_series_thm}. Indeed, $\pi_2$ is injective, and for all other subquotients $J_t(n_1)$ in the filtration, we have $J_t(n_1)_{\pi_2} = 0$. Moreover, we have the equivalence $I_0(n_1)_{\pi_2} \neq 0 \Leftrightarrow m_2^c \leq m_2$ by cuspidality of $\pi_2$, because $I_0(n_1)$ is a proper parabolically induced representations exactly when $m_2^c >m_2$. As $I(n_1)_{\pi_2} \neq 0$ when $m_2^-(\pi_1^{[\chi_1]}) \leq m_2$, this implies the inequality. In general $\pi_2 \otimes \Theta(\pi_1^{[\chi_1]},-V_2^{m_2})$ is a quotient of $\Theta(\chi_1,V_2^{m_2^c} \oplus (-V_2^{m_2}))_{\pi_2}$, which is contained in $I(n_1)_{\pi_2}$. Therefore, by normalising the parabolic induction functor, it is a subquotient of: 
$$I(n_1)_{\pi_2} = I_0(n_1)_{\pi_2} = \pi_2 \otimes i_{P_\delta^-}^{H_2^-}(\nu_{\delta,s}' \otimes \pi_2^{[\xi_0^-]}).$$

\noindent b) By the generalised doubling method:
$$\Theta(\chi_1,V_2^{m_2^c} \oplus (-V_2^{m_2^{-,c}})) \twoheadrightarrow \Theta(\pi_1,V_2^{m_2^c}) \otimes \Theta(\pi_1^{[\chi_1]},-V_2^{m_2^{-,c}}) \twoheadrightarrow \pi_2 \otimes \pi_2^-.$$
Therefore $I(s)_{\pi_2 \otimes \pi_2^-} \neq 0$. By the previous point, we have $m_2^c = m_2^{-,c}$ and the inequalities:
$$m_2(\pi_1) \leq m_2^c \leq m_2^-(\pi_1^{[\chi_1]}) \textup{ and } m_2^-(\pi_1^{[\chi_1]}) \leq m_2^{-,c} \leq m_2(\pi_1),$$
so all these integers are equal. Because $\pi_2$ is injective: 
$$\Theta(\chi_1,V_2^{m_2^c} \oplus (-V_2^{m_2^c}))_{\pi_2} \subset I(s)_{\pi_2} = (I_0(s))_{\pi_2} = \pi_2 \otimes \pi_2^{[\xi_0^-]}.$$
Therefore $\pi_2^- = \pi_2^{[\xi_0^-]} = \Theta(\pi_1^{[\chi_1]},-V_2^{m_2^{-,c}})$. Similarly $\pi_2 = \Theta(\pi_1,V_2^{m_2^c})$. The fourth statement is a consequence of the third one, which follows from a) applied with $\pi_2^-$.

\noindent c) We want to show $\Theta(\chi_1,V_2^{m_2} \oplus (-V_2^{m_2^c}))_{\pi_2^-} \simeq \Theta(\pi_1,V_2^{m_2}) \otimes \pi_2^-$. We are going to build a section of the surjection:
$$\Theta(\chi_1,V_2^{m_2} \oplus (-V_2^{m_2^c}))_{\pi_2^-} \twoheadrightarrow \Theta(\pi_1,V_2^{m_2}) \otimes \pi_2^-.$$
The morphism: 
$$\omega_{m_1,m_2} \otimes \omega_{m_1,-m_2^c} \twoheadrightarrow \Theta(\chi_1,V_2^{m_2} \oplus ( - V_2^{m_2^c}))_{\pi_2^-}$$
is $(H_1^\triangle \times H_2^{m_2} \times H_2^{m_2^c})$-equivariant so it factors through: 
$$\omega_{m_1,m_2} \otimes (\Theta(\pi_2^-,V_1) \otimes \pi_2^-) = \omega_{m_1,m_2} \otimes (\pi_1^{[\chi_1]} \otimes \pi_2^-) \twoheadrightarrow \Theta(\chi_1,V_2^{m_2} \oplus ( - V_2^{m_2(\pi_1)}))_{\pi_2^-}.$$
We can factor it again using the $H_1^\triangle$-equivariance to obtain the desired section:
$$(\omega_{m_1,m_2} \otimes \pi_1^{[\chi_1]})_{\chi_1} \otimes \pi_2^- = \Theta(\pi_1,V_2^{m_2}) \otimes \pi_2^-  \twoheadrightarrow \Theta(\chi_1,V_2^{m_2} \oplus ( - V_2^{m_2^c}))_{\pi_2^-}.$$
Going back to the proof of b) we can now replace ``subquotient'' by ``submodule''. \end{proof}

\begin{rem} \label{cuspidal_quotient_for_theta_rem1} Even in the most favourable situation c) above, nothing \textit{a priori} prevents $\Theta(\pi_1,V_2^{m_2})$ from having cuspidal quotients when $m_2 > m_2(\pi_1)$. \end{rem}

\begin{rem} We could express the result in the terms of the theta lift for $\psi^{-1}$ following Remark \ref{doubling_method_from_minus_to_psi_inverse_rem} as $\Theta_{\psi^{-1}}(\pi_1^{[\chi_1]},V_2) \simeq \Theta(\pi_1^{[\chi_1]},-V_2)$ via pullback. \end{rem}

When the characteristic of the coefficients is good, we have a situation similar to the complex setting, but only up to a certain index $m_2(\textup{ban})$. We use the notation $\sigma \leq \tau$ when $\sigma$ is a subquotient of $\tau$. As long as we sit in the banal range most results of the classical setting hold:

\begin{cor} \label{lifts_of_cuspidals_banal_setting_cor} Let $\ell$ be banal with respect to $U(V_1)$ and define $m_2(\textup{ban})$ as the largest integer such that $\ell$ is banal with respect to $U(V_2^{m_2})$.  Let $\pi_1 \in \textup{Irr}_R^{\textup{cusp}}(H_1)$ and assume $m_2(\pi_1) \leq m_2(\textup{ban})$. Then:
\begin{enumerate}[label=\textup{\alph*)}]
\item $\Theta(\pi_1,V_2^{m_2})$ is irreducible for all $m_2(\pi_1) \leq m_2 \leq m_2(\textup{ban})$; 
\item $\Theta(\pi_1,V_2^{m_2})$ is a submodule of $\mathfrak{i}_{P_{\delta}}^{H_2} ( \nu_{\delta,s} \otimes \Theta(\pi_1,V_2^{m_2(\pi_1)}))$ where $\delta = m_2 - m_2^c$:
\begin{itemize}
\item[$\bullet$] in particular $\Theta(\pi_1,V_2^{m_2})$ is cuspidal only when $m_2 \leq m_2(\pi_1)$
\end{itemize}
\item for $\pi_1' \in \textup{Irr}_R(H_1)$ we have the equivalence:
$$\Theta(\pi_1,V_2^{m_2(\pi_1)}) \leq \Theta(\pi_1',V_2^{m_2(\pi_1)}) \Leftrightarrow \pi_1' \simeq \pi_1;$$
\item for $\pi_1' \in \textup{Irr}_R^{\textup{cusp}}(H_1)$ and $m_2 \geq m_2(\pi_1)$ we have the equivalence:
$$\Theta(\pi_1',V_2^{m_2}) \simeq \Theta(\pi_1,V_2^{m_2}) \Leftrightarrow \pi_1' \simeq \pi_1.$$
\end{enumerate} \end{cor}

\begin{proof} a) The proof is the same as in the complex case, but the only proper detailed reference we found are the notes of Kudla \cite[IV.3 Prop]{kudla_notes}. We recall quickly how this works. Because of our assumption on $m_2(\textup{ban})$, cuspidal representations are $p\&i$ for all $m_2 \leq m_2(\textup{ban})$, so we can apply Proposition \ref{proj_inj_theta_quotient_prop} to obtain that $\Theta(\pi_1,V_2^{m_2(\pi_1)})$ is irreducible. Then, a proof by induction, using faithfullness properties of Jacquet functors, gives the irreducibility of $\Theta(\pi_1,V_2^{m_2})$ as long as $m_2 \leq m_2(\textup{ban})$.

\noindent b) As a consequence of Lemma \ref{cuspidal_lift_jacquet_functor_inductive_formula_lem}, the representation $\Theta(\pi_1,V_2^{m_2})$ must contain at least one non-cuspidal irreducible subquotient if $m_2 > m_2(\pi_1)$. So the equivalence follows.

\noindent c) Let $\pi_2 = \Theta(\pi_1,V_2^{m_2(\pi_1)}) \in \textup{Irr}_R^{\textup{cusp}}(H_2^{m_2(\pi_1)})$. Then $\pi_2$ also occurs as a quotient of $\Theta(\pi_1',V_2^{m_2(\pi_1)})$. In a) we can reverse the roles of $H_1$ and $H_2^{m_2(\pi_1)}$ to obtain that $\Theta(\pi_2,V_1)$ is irreducible. As both $\pi_1$ and $\pi_1'$ occur as quotients of $\Theta(\pi_2,V_1)$, we must have $\pi_1' \simeq \pi_1$.

\noindent d) Applying Lemma \ref{cuspidal_lift_jacquet_functor_inductive_formula_lem}, this implies $\Theta(\pi_1',V_2^{m_2(\pi_1)}) \simeq \Theta(\pi_1,V_2^{m_2(\pi_1)})$ then use c). \end{proof}

\begin{rem} If we assume $n_1 \leq m_2(\textup{ban})$, the corollary is uniform in $\pi_1$ because $n_1$ is an upper bound for first occurrence indices by Theorem \ref{first_occurrence_index_thm}. \end{rem}

\begin{rem} We simply repeat Remark \ref{cuspidal_quotient_for_theta_rem1} to stress the fact that nothing \textit{a priori} prevents $\Theta(\pi_1,V_2^{m_2})$ from having an irreducible cuspidal quotient when $m_2 > m_2(\textup{ban})$. \end{rem}

\paragraph{Cuspidal support.} As an easy consequence of Lemma \ref{cuspidal_lift_jacquet_functor_inductive_formula_lem}, we have:

\begin{lem} Let $\pi_1 \in \textup{Irr}_R^{\textup{cusp}}(H_1)$ and assume $\pi_2 \in \textup{Irr}_R(H_2^{m_2})$ occurs a quotient of $\Theta(\pi_1,V_2^{m_2})$ and $\pi_2 \subset \mathfrak{i}_{P_k^{m_2}}^{H_2^{m_2}}(\rho_2 \otimes \sigma_2)$ with $\rho_2 \otimes \sigma_2 \in \textup{Irr}_R(M_k^{m_2})$. Then:
$$\rho_2 = \nu_{k,s} \ \textup{ and } \ \Theta(\pi_1,V_2^{m_2-k}) \twoheadrightarrow \sigma_2.$$ \end{lem}

If we apply the proposition with $\sigma_2$ cuspidal, we then get that the cuspidal support of $\pi_2$ is coming from $[\nu_{k,s} \otimes \sigma_2,M_k^{m_2}]$, so it is supported on the Levi $(G_1^{m_2})^k \times H_2^{m_2-k}$. If $\pi_2$ is a cuspidal quotient of $\Theta(\pi_1,V_2^{m_2})$, then either $m_2 = m_2(\pi_1)$ or $m_2 > m_2(\textup{ban})$. The lemma becomes tautological in this case as we must have $k = 0$, $\rho_2 = 1$ and $\sigma_2 = \pi_2$. If $m_2 \leq m_2(\textup{ban})$ the (super)cuspidal support is identical to the classical case as a consequence of a) and b) in Corollary \ref{lifts_of_cuspidals_banal_setting_cor} -- we refer to Corollary \ref{cuspidal_support_general_pi1_cor} for the explicit form of the support since it encompasses the cuspidal case.

\subsection{General representations $\pi_1$}

Let $\sigma_1 \in \textup{Irr}_R^{\textup{cusp}}(H_1^{m_1-k})$. 

\paragraph{Finite length from cuspidals.} It is enough to know finite generation of big theta lifts of cuspidals to deduce that all big theta lifts have finite length:

\begin{prop} \label{finite_length_from_cuspidals_prop} Let $m_2$ be an integer and assume that $\Theta(\sigma_1,V_2^{m_2^0})$ has finite length for all $m_2^0 \leq m_2$. Let $\rho_1 \in \textup{Irr}_R(G_k^{m_1})$, $\sigma_1 \in \textup{Irr}_R^{\textup{cusp}}(H_1^{m_1-k})$ and $\pi_1 \subset \mathfrak{i}_{P_k}^{H_1}(\rho_1 \otimes \sigma_1)$ be irreducible. Then $\Theta(\pi_1,V_2^{m_2})$ has finite length. \end{prop}

\begin{proof} Consider the surjection $\omega_{m_1,m_2} \twoheadrightarrow \pi_1 \otimes \Theta(\pi_1,V_2^{m_2})$ and apply the Jacquet functor $\mathfrak{r}_{X_1^k}$ to it. By Frobenius reciprocity, we have $\mathfrak{r}_{X_1^k}(\pi_1) \twoheadrightarrow \rho_1 \otimes \sigma_1$. Therefore:
$$\mathfrak{r}_{X_1^k}(\omega_{m_1,m_2}) \twoheadrightarrow \rho_1 \otimes \sigma_1 \otimes \Theta(\pi_1,V_2^{m_2}).$$
Looking at the subquotient $J_t$ of the filtration in Theorem \ref{kudla_filtration_thm}, we obtain:
$$(J_t)_{\sigma_1} = \mathfrak{i}_{Q_{{}^t X_1^k} \times {}^k H_1^{m_1} \times P_t^{m_2}}^{G_{X_1^k} \phantom{k} \times {}^k H_1^{m_1} \times H_2}(\delta_{{}^t X_1^k,n_2} |\textup{det}_{{}^t X_1^k}|^{\frac{s+k-t}{2}} \otimes \textup{Reg}^{\alpha_t,\xi_t}(G_1^t,G_2^t) \otimes \sigma_1 \otimes \Theta(\sigma_1,V_2^{m_2-t})).$$
A classical argument ensures that all three $(J_t)_{\sigma_1 \otimes \rho_1}$, $(\mathfrak{r}_{X_1^k}(\omega_{m_1,m_2}))_{\rho_1 \otimes \sigma_1}$ and $\Theta(\pi_1,V_2^{m_2})$ have finite length. We sketch this argument in Appendix \ref{end_proof_prop_finite_length_app}. \end{proof}

\paragraph{Cuspidal support.} Let $\upsilon_{k,s} = \delta_{X_1^k,n_2} |\textup{det}_{X_1^k}|^{\frac{s+k}{2}}$ be the genuine character of $G_{X_1^k}$.

\begin{lem} \label{dualising_or_erasing_lem} Let $\rho_1 \in \textup{Irr}_R^{\textup{cusp}}(G_k^{m_1})$, $\sigma_1 \in \textup{Irr}_R(H_1^{m_1-k})$ and $\pi_1 \subseteq \mathfrak{i}_{P_k}^{H_1}(\rho_1 \otimes \sigma_1)$.
\begin{enumerate}[label=\textup{\alph*)}]
\item If $\rho_1 \neq \upsilon_{k,s}$ then: 
$$\mathfrak{i}_{P_k}^{H_2}(\rho_1^{[\alpha_k,\xi_k]} \otimes \Theta(\sigma_1,V_2^{m_2-k})) \twoheadrightarrow \Theta(\pi_1,V_2^{m_2}).$$
\item If $\rho_1 = \upsilon_{k,s}$ then $k=1$ and for all irreducible quotients $\Theta(\pi_1,V_2^{m_2}) \twoheadrightarrow \pi_2$:
$$\mathfrak{i}_{P_1}^{H_2}(\upsilon_{1,s}^{[\alpha_1,\xi_1]} \otimes \Theta(\sigma_1,V_2^{m_2-1})) \twoheadrightarrow \pi_2 \textup{ or } \Theta(\sigma_1,V_2^{m_2}) \twoheadrightarrow \pi_2.$$
\end{enumerate}
\end{lem}

\begin{proof} We have $\mathfrak{r}_{X_1^k}(\omega_{m_1,m_2}) \twoheadrightarrow \rho_1 \otimes \sigma_1 \otimes \Theta(\pi_1,V_2^{m_2})$ by exactness of the Jacquet functor and the fact that $\mathfrak{r}_{X_1^k}(\pi_1) \twoheadrightarrow \rho_1 \otimes \sigma_1$ by Frobenius reciprocity. By cuspidality of $\rho_1$, only the subquotients $J_k$ and $J_0$ in the filtration of $\mathfrak{r}_{X_1^k}(\omega_{m_1,m_2})$ may have cuspidal quotients. We obtain an exact sequence:
$$(J_k)_{\rho_1 \otimes \sigma_1} \to (\mathfrak{r}_{X_1^k}(\omega_{m_1,m_2}))_{\rho_1 \otimes \sigma_1} \to (J_0)_{\rho_1 \otimes \sigma_1} \to 0$$
where: 
$$(J_k)_{\rho_1 \otimes \sigma_1} = \rho_1 \otimes \sigma_1 \otimes \mathfrak{i}_{P_k}^{H_2}(\rho_1^{[\alpha_k,\xi_k]} \otimes \Theta(\pi_1,V_2^{m_2-k}))$$
and:
$$(J_0)_{\rho_1 \otimes \sigma_1} = \left\{ \begin{array}{cc}
0 & \textup{if } \rho_1 \neq \upsilon_{k,s} \\
\rho_1 \otimes \sigma_1 \otimes \Theta(\sigma_1,V_2^{m_2}) & \textup{if } \rho_1 = \upsilon_{k,s} 
\end{array} \right.$$
When $\rho_1 = \upsilon_{k,s}$ it is cuspidal and a character, so $k$ is $1$ or $0$. The result easily follows. \end{proof}

If we consider an irreducible quotient $\Theta(\pi_1,V_2^{m_2}) \twoheadrightarrow \pi_2$, then the previous lemma tells us that we are in either of the two cases:
$$\mathfrak{i}_{P_k}^{H_2}(\rho_1^{[\alpha_k,\xi_k]} \otimes \Theta(\pi_1,V_2^{m_2-k})) \twoheadrightarrow \pi_2 \textup{ or } \Theta(\sigma_1,V_2^{m_2}) \twoheadrightarrow \pi_2.$$
We refer to the first situation as \textit{dualising} and the second as \textit{erasing}. In order for the erasing to happen, it is necessary that $\rho_1 = \upsilon_{1,s}$ \textit{i.e.} $k=1$. Note that dualising does not necessarily imply the existence of an irreducible quotient $\Theta(\sigma_1,V_2^{m_2-k}) \twoheadrightarrow \pi_2'$ such that: 
$$\mathfrak{i}_{P_k}^{H_2}(\rho_1^{[\alpha_k,\xi_k]} \otimes \pi_2') \twoheadrightarrow \pi_2$$
that is equivalent by Cassleman duality \cite[II.3.8]{vig_book} to:
$$\rho_1^{[\alpha_k,\xi_k]} \otimes \pi_2' \subseteq \mathfrak{r}_{X_1^k}^{\textup{op}}(\pi_2),$$
where we used the parabolic restriction with respect to the opposite parabolic of $P(X_1^k)$, obtained after choosing a minimal parabolic along with a Levi decomposition. It is valid for the metaplectic group too since $(P,K)$-stability techniques \cite[II.3.7]{vig_book} still apply.

The surjection $\mathfrak{i}_{P_k}^{H_2}(\rho_1^{[\alpha_k,\xi_k]} \otimes \pi_2') \twoheadrightarrow \pi_2$ guarantees that $\mathfrak{i}_{P_k}^{H_1}(\rho_1^{[\alpha_k,\xi_k]} \otimes \pi_2')$ and $\pi_2$ have same (super)cuspidal support only when $\ell$ is banal \textit{i.e.} $m_2 \leq m_2(\textup{ban})$. For this reason, we can only keep track of the support -- through successive dualising and erasing  -- in the banal setting.

\begin{cor} \label{cuspidal_support_general_pi1_cor} Assume $\ell$ is banal with respect to $H_1$ and let $m_2 \leq m_2(\textup{ban})$. Let $\pi_1 \otimes \pi_2$ be an irreducible quotient of $\omega_{m_1,m_2}$. Denote by $[\pi_1] = [\rho_1,\rho_2,\cdots, \rho_r ; \sigma_1]$ the supercuspidal support of $\pi_1$ associated to the Levi $G_{k_1} \times G_{k_2} \times \cdots \times G_{k_r} \times H_1^{m_1-\sum k_i}$. Then there exist indices $1 \leq i_1 < \cdots < i_t \leq r$ such that:
\begin{eqnarray*}
[\pi_2] & = & [\rho_{i_1}^{[\alpha_{k_{i_1}},\xi_{k_{i_1}}]}, \cdots, \rho_{i_t}^{[\alpha_{k_{i_t}},\xi_{k_{i_t}}]} ; [\Theta(\sigma_1,V_2^{m_2(\sigma_1)+j})]] \\
 & = & [\rho_{i_1}^{[\alpha_{k_{i_1}},\xi_{k_{i_1}}]}, \cdots, \rho_{i_t}^{[\alpha_{k_{i_t}},\xi_{k_{i_t}}]},\nu_{1,s}, \cdots, \nu_{1,s-2j};\Theta(\sigma_1,V_2^{m_2(\sigma_1)})]
 \end{eqnarray*}
where $m_2 = k_{i_1} + \cdots + k_{i_t} + j + m_2(\sigma_1)$. In addition we can assume that:
\begin{itemize}
\item[$\bullet$] $t=r$ and $j = m_2 - (\sum_{1 \leq i \leq r} k_i) - m_2(\sigma_1) \geq 0$ if $m_2 \geq (\sum_{1 \leq i \leq r} k_i) + m_2(\sigma_1)$;
\item[$\bullet$] $j=0$ and $t < r$ if $m_2 < (\sum_{1 \leq i \leq r} k_i )+ m_2(\sigma_1)$. 
\end{itemize} \end{cor}

\begin{proof} By induction, we obtain after dualising $t$ times and erasing $r-t$ times: 
$$[\pi_2] = [\rho_{i_1}^{[\alpha_{k_{i_1}},\xi_{k_{i_1}}]},\cdots,\rho_{i_t}^{[\alpha_{k_{i_t}},\xi_{k_{i_t}}]}; [ \Theta(\sigma_1,V_2^{m_2(\sigma_1) + \delta }]]$$
where $\Theta(\sigma_1,V_2^{m_2(\sigma_1) + \delta})$ is irreducible by Corollary \ref{lifts_of_cuspidals_banal_setting_cor}. If $\delta=0$, we obtain the result with $j=0$. In this case $m_2 \leq (\sum_{1 \leq i \leq r} k_i) + m_2(\sigma_1)$. If $\delta>0$, we have:
$$[\Theta(\sigma_1,V_2^{m_2(\sigma_1) + \delta })] = [\nu_{1,s'}, \cdots, \nu_{1,s'-2\delta}; \Theta(\sigma_1,V_2^{m_2(\sigma_1)})]$$ 
with $s' = n_2(\sigma_1) + 2 \delta - (n_1 - 2 \sum k_i) + \eta_1$ and $n_2(\sigma_1) = n_2^0 + 2m_2(\sigma_1) = \textup{dim}(V_2^{m_2(\sigma_1)})$. Note that $s= n_2 - n_1 + \eta_1 = n_2(\sigma_1) + 2 \delta + 2 (\sum k_{i_j}) - n_1 + \eta_1 = s' - 2 \sum k_{i_j'}$ where $i_j'$ is the complement of $i_j$ in $\{1, \cdots , r\}$. Since the indices $i_j'$ have been erased, we must have $k_{i_1'}=k_{i_2'}=\cdots = k_{i_{r-t}'} = 1$ and $\rho_{i_1'} = \upsilon_{1,s}$, $\rho_{i_2'} = \upsilon_{1,s+2}$, $\cdots$, $\rho_{i_{r-t}'} = \upsilon_{1,s+2(r-t-1)}$. Indeed, dualising preserves $s = n_2 - n_1 + \eta_1 = n_2 - 2k - (n_1 -2k) + \eta_1$ whereas erasing replaces it by $n_2 - (n_1-2) +\eta_1 = s+2$. We can rewrite:
$$[\Theta(\sigma_1,V_2^{m_2(\pi_1) + \delta })] = [\nu_{1,s+2(r-t)}, \cdots, \nu_{1,s+2(r-t)-2\delta}; \Theta(\sigma_1,V_2^{m_2(\pi_1)})].$$
If $m_2 \geq \sum k_i + m_2(\sigma_1)$, then $\delta \geq \sum k_{i_j'} = r-t$ because $m_2 = m_2(\sigma_1) + \sum k_{i_j} + \delta$. So $j=\delta - (r-t) \geq 0$. We can list again our characters as:
$$\nu_{1,s+2(r-t)}, \cdots,\nu_{1,s+2}, \nu_{1,s}, \cdots, \nu_{1,s-2j}.$$
But $(\upsilon_{1,u})^{[\alpha_1,\xi_1]}=\nu_{1,u} |\cdot|^1 = \nu_{1,u+2}$. After a small reordering, this list is also:
$$(\upsilon_{1,s})^{[\alpha_1,\xi_1]}, \cdots,(\upsilon_{1,s+2(r-t-1)})^{[\alpha_1,\xi_1]},\nu_{1,s}, \cdots, \nu_{1,s-2j}.$$
Therefore we can reintegrate the first $r-t$ characters in the support as they run over the $\rho_{i_j'}$'s we erased. The argument is similar when $\delta < \sum k_{i_j'}$, but this time we will not be able to reintegrate all the characters we erased. \end{proof}

\section{Modular theta correspondence} \label{proof_modular_theta_sec}

We begin by recalling some of the notations. We are interested in the non-quaternionic case \textit{i.e.} when $D=E$ and $[E:F] \leq 2$. Let $R$ be an algebraically closed field of characteristic $\ell \neq p$. We fix a choice $q^{1/2}$ of a square root of $q$ in $R$.

In this section, $V_1$ is an $\epsilon_1$-hermitian space over $E$ of dimension $n_1$ and Witt index $m_1$. Let $V_1^k \subset V_1$ be in the same Witt series as $V_1$ with Witt index $k$. We assume they are ordered for inclusion \textit{i.e.} $V_1^{k} \supseteq V_1^{k'}$ if $k \leq k'$. We remark that $V_1^{m_1} = V_1$ and $V_1^0$ is anisotropic. We use similar notations when $V_2$ is an $\epsilon_2$-hermitian space. We now assume that $\epsilon_1 \epsilon_2 = - 1$, so that $(U(V_1),U(V_2))$ is an irreducible dual pair in $\textup{Sp}(V_1 \otimes_E V_2)$. Set:
$$\eta_1 = \left\{ \begin{array}{ll}
\epsilon_1 & \textup{if } E=F; \\
0 & \textup{if } [E:F]=2.
\end{array} \right.$$
Up to switching $V_1$ and $V_2$, we can always assume $n_2 \leq n_1 -\eta_1$. Set $s = n_2-n_1+\eta_1 \leq 0$. We fix a complete isotropic flag $\{X_1^k\}$ in $V_1$, together with a dual isotropic flag $\{Y_1^k\}$. This defines $V_1^0$ as the orthogonal complement of $X_1^{m_1} \oplus Y_1^{m_1}$ and $V_1^k = X_1^k + V_1^0 + Y_1^k$ satisfying $V_1^k \subset V_1^{k'}$ if $k \leq k'$. We do similar choices and use similar notations for $V_2$.

This section is generalising the strategy of \cite{gt} to the modular setting. We will denote by $\vartheta(\pi_1,V_2) \in \mathbb{N} \cup \{ +\infty \}$ the supremum of the lengths of semisimple quotients of $\Theta(\pi_1,V_2)$. If $\Theta(\pi_1,V_2)$ has finite length, then $\vartheta(\pi_1,V_2)$ is exactly the length of the cosocle $\theta(\pi_1,V_2)$. In particular it is finite. However, for representations of infinite length, or when the cosocle is not well-defined, we can have $\vartheta(\pi_1,V_2) = 0$ even though $\Theta(\pi_1,V_2) \neq 0$ and this means $\Theta(\pi_1,V_2)$ does not admit an irreducible quotient.

We also assume in this section that parabolic restriction functors preserve finite length representations. In the modular case, this holds for all reductive groups \cite[II.5.13]{vig_book}, but it is unknown for the metaplectic group in general, except when $p \neq 2$ or $\ell$ is large enough compared to the metaplectic group we study. We refer to Section \ref{covering-groups-sec} for a more detailed discussion about this point.

Moreover, we want to use the MVW involution, but there are some restrictions again in the metaplectic case. Indeed, in Theorem \ref{MVW_involution_metaplectic_thm}, we were able to prove its existence except when the characteristic of $F$ is positive and less than half the dimension of the symplectic space. Therefore, when the metaplectic group appears and $F$ has positive characteristic, we need to assume the characteristic is large enough.

\subsection{Boundary of $\mathfrak{I}(u)$} \label{boundary_of_I_u_sec}

Let $I(n_2)$ be the representation of the doubling method from Section \ref{Rallis_filtration_sec} with $V_1 \oplus (-V_1)$ and $V_2$. In particular $V_1$ and $-V_1$ are anti-isometric via $\alpha^- = \textup{id}_{V_1}$. Because $\alpha^-$ is so canonical, we make it implicit to lighten notations. Using Proposition \ref{modulus_character_parabolic_prop}, we normalise $I(n_2)$ to obtain:
$$\mathfrak{I}(s) = \mathfrak{i}_{P_1^\square}^{H_1^\square}(| \cdot |^{\frac{-2 n_1+n_1+\eta_1}{2}}\nu_{X_1^\square,n_2}) = \mathfrak{i}_{P_1^\square}^{H_1^\square}(| \cdot |^{\frac{s}{2}}\delta_{X_1^\square,n_2}).$$
We also identify $H_1^-$ and $H_1$ via $\gamma_-$ so that the filtration of $\mathfrak{I}(s)$ has subquotients: 
$$\mathfrak{I}_t(s) = \mathfrak{i}_{P_t \times P_t}^{H_1 \times H_1} \big( \upsilon_{t,s} \otimes \alpha \upsilon_{t,s} \otimes \textup{Reg}^{\xi_t^-}(U_t, U_t) \big)$$
where $\upsilon_{t,s} = \delta_{X_1^t,n_2} |\textup{det}_{X_1^t}|^{\frac{s+t}{2}}$ and $\alpha(m) = (-1,\textup{det}_F(m))^{n_2}$ for $m \in \textup{GL}_t(E)$.

We will need to consider other parameters than $s$, so for $u \in \mathbb{Z}$ we set:
$$\mathfrak{I}(u) = \mathfrak{i}_{P_1^\square}^{H_1^\square}(| \cdot |^{\frac{u}{2}}\delta_{X_1^\square,n_2}).$$
The filtration of $\mathfrak{I}(u)$ is obtained by replacing $s$ by $u$ and has subquotients $\mathfrak{I}_t(u)$. We consider $\mathfrak{I}(u) \in \textup{Rep}_R(H_1 \times H_1)$ as a representation of $H_1$ by restriction to the first factor.

\begin{defi} Let $u \in \mathbb{Z}$. We say $\pi_1 \in \textup{Irr}_R(H_1)$ occurs at $t$ in the filtration of $\mathfrak{I}(u)$ if $\mathfrak{I}_{t}(u)_{\pi_1} \neq 0$, or equivalently, $\textup{Hom}_{H_1}(\mathfrak{I}_{t}(u),\pi_1) \neq 0$. If there exists $t > 0$ such that $\pi_1$ occurs at $t$, we say $\pi_1$ appears on the boundary of $\mathfrak{I}(u)$. \end{defi}

\begin{rem} Note that $\{ \pi_1 \ | \ \mathfrak{I}_{0}(u)_{\pi_1} \neq 0\} = \textup{Irr}_R(H_1)$, so all irreducible representations occur at $0$. The representations that only occur at $0$ are said to be outside the boundary. They include all cuspidal representations, but they are not the only ones. We will see in the next section that representations outside the boundary are easier to deal with. \end{rem}

There are other equivalent useful definitions for the boundary:

\begin{prop} \label{equivalences_boundary_prop} Let $\pi_1 \in \textup{Irr}_R(H_1)$. The following assertions are equivalent:
\begin{enumerate}[label=\textup{\alph*)}]
\item $\pi_1$ appears on the boundary of $\mathfrak{I}(u)$;
\item there exists $0 < t \leq m_1$ such that $\textup{Hom}_{G_t^{m_1}}(\upsilon_{t,u},\mathfrak{r}_{X_1^t}^{\textup{op}}(\pi_1)) \neq 0$; \end{enumerate}
If parabolic restriction preserves finite length representations, they are equivalent to:
\begin{enumerate}
\item[\textup{c)}] there exist $0 < t \leq m_1$ and $\sigma_1 \in \textup{Irr}_R(H_1^{m_1 - t })$ such that $\textup{Hom}_{H_1}(\pi_1 , \alpha \upsilon_{t,u}^\vee \rtimes \sigma_1) \neq 0$.
\end{enumerate} \end{prop}

\begin{proof} Let $\pi_1$ be on the boundary of $\mathfrak{I}(u)$ at $t$.

\noindent a) $\Leftrightarrow$ b) As $\pi_1$ is admissible, second adjunction is Casselman duality $\mathfrak{r}_{X_1^t}(\pi_1^\vee)^\vee \simeq \mathfrak{r}_{X_1^t}^{\textup{op}}(\pi_1)$. Therefore:
$$\textup{Hom}_{H_1}(\mathfrak{I}_{t}(u),\pi_1) \simeq \textup{Hom}_{M_t}(\upsilon_{t,u} \otimes \textup{Reg}^{\xi_t^-}(U_t,U_t),\mathfrak{r}_{X_1^t}^{\textup{op}}(\pi_1))$$
and the latter is non-zero if and only if $\textup{Hom}_{G_t}(\upsilon_{t,u},\mathfrak{r}_{X_1^t}^{\textup{op}}(\pi_1)) \neq 0$.

\noindent c) $\Leftrightarrow$ b) Applying the MVW involution to $\pi_1 \hookrightarrow  \alpha \upsilon_{t,u}^\vee \rtimes \sigma_1$ gives $\pi_1^\vee \hookrightarrow \upsilon_{t,u}^\vee \rtimes \sigma_1'$ for some $\sigma_1' \in \textup{Irr}_R(H_1^{m_1 - t })$ \textit{i.e.} $\textup{Hom}_{H_1}(\pi_1^\vee , \upsilon_{t,u}^\vee \rtimes \sigma_1') \neq 0$. By Frobenius reciprocity:
$$\textup{Hom}_{M_t}(\mathfrak{r}_{X_1^t}(\pi_1^\vee) , \upsilon_{t,u}^\vee \otimes  \sigma_1') \neq 0.$$
By duality, this is $\textup{Hom}_{M_t}(\upsilon_{t,u} \otimes  \sigma_1'',\mathfrak{r}_{X_1^t}(\pi_1^\vee)^\vee) \simeq \textup{Hom}_{M_t}(\upsilon_{t,u} \otimes  \sigma_1'',\mathfrak{r}_{X_1^t}^{\textup{op}}(\pi_1)) \neq 0$ for some $\sigma_1'' \in \textup{Irr}_R(H_1^{m_1 - t })$. The converse requires parabolic restriction to preserve finite length representations, which is known \cite[II.5.13]{vig_book} for all reductive groups in the modular setting. In the metaplectic case, it is an assumption we made in this section. \end{proof}


%
	
\subsection{Representations outside the boundary}

In this section, we show the Howe correspondence is true for representations outside the boundary of $\mathfrak{I}(-s)$ provided a certain assumption related to the generalised doubling method holds. Since we will very often refer to this hypothesis, we define it here:
\[
\mathfrak{I}(-s) \twoheadrightarrow \Theta(\chi_2,V_1 \oplus (-V_1)) \textup{ in } \textup{Rep}_R(H_1 \times H_1) \label{hypothesis_H} \tag{H}
\]

\begin{theo} \label{rep_outside_boundary_thm} Assume \textup{(H)} holds. Then, if $\pi_1$ is outside the boundary of $\mathfrak{I}(-s)$, we have one of the following two situations:
\begin{itemize}
\item $\Theta(\pi_1,V_2)$ has no irreducible quotient;
\item $\Theta(\pi_1,V_2)$ has a unique irreducible quotient, \textit{i.e.} its cosocle $\theta(\pi_1,V_2)$ is irreducible. Moreover, for $\pi_1' \in \textup{Irr}_R(H_1)$, we have $\Theta(\pi_1',V_2) \twoheadrightarrow \theta(\pi_1,V_2) \Leftrightarrow \pi_1' \simeq \pi_1$.
\end{itemize}
\end{theo}

\begin{proof} By assumption:
$$\mathfrak{I}(-s) \twoheadrightarrow \Theta(\chi_2,V_1 \oplus (-V_1)).$$
Because $\pi_1$ does not occur on the boundary of $\mathfrak{I}(-s)$, we have: 
$$\mathfrak{I}_0(-s)_{\pi_1 \otimes (\pi_1')^{[\xi_0^-]}} \twoheadrightarrow \mathfrak{I}(-s)_{\pi_1 \otimes (\pi_1')^{[\xi_0^-]}}$$
where $\mathfrak{I}_0(-s)_{\pi_1 \otimes (\pi_1')^{[\xi_0^-]}}$ is $\pi_1 \otimes (\pi_1)^{[\xi_0^-]}$ if $\pi_1' \simeq \pi_1$ and $0$ otherwise. Therefore:
$$\mathfrak{I}_0(-s)_{\pi_1 \otimes (\pi_1')^{[\xi_0^-]}} \twoheadrightarrow \Theta(\chi_2,V_1 \oplus (-V_1))_{\pi_1 \otimes (\pi_1')^{[\xi_0^-]}}.$$
By Lemma \ref{seesaw_identity_isotypic_chi1_pi2_lem} and Remark \ref{doubling_method_from_minus_to_psi_inverse_rem}, we have:
$$(\mathfrak{I}_0(-s) \otimes \pi_1^\vee \otimes ((\pi_1')^{[\xi_0^-]})^\vee)_{U(V_1) \times U(V_1)} \twoheadrightarrow (\Theta_\psi(\pi_1,V_2) \otimes \Theta_{\psi^{-1}}((\pi_1')^{[\xi_0^-]},V_2))_{\chi_2}.$$
where the left-hand side is $R$ when $\pi_1' \simeq \pi_1$ and $0$ otherwise.

Assume the first situation does not hold, so $\Theta(\pi_1,V_2)$ admits an irreducible quotient. From Appendix \ref{dual_pairs_weil_rep_appendix}, if we choose MVW involutions $\delta_1$ and $\delta_2$ of respectively $V_1$ and $V_2$, the morphism $\delta = \delta_1 \otimes \delta_2$ is in $\textup{GSp}(W)$ and has similitude factor $a=-1$. Therefore $\omega_{\psi,H_1,H_2}^\delta \simeq \omega_{\psi^{-1},H_1,H_2} \circ (\delta_1,\delta_2)$ implies:
$$\Theta_{\psi^{-1}} ((\pi_1')^{[\xi_0^-]},V_2)^{\delta_2} \simeq \Theta_{\psi} (((\pi_1')^{[\xi_0^-]})^{\delta_1},V_2) = \Theta(\pi_1',V_2)$$
\textit{i.e.} $\Theta_{\psi^{-1}}((\pi_1')^{[\xi_0^-]},V_2) \simeq \Theta(\pi_1',V_2)^{\delta_2}$. The dimension of $(\Theta(\pi_1,V_2) \otimes \Theta(\pi_1',V_2)^{\delta_2})_{\chi_2}$ measures common factors in the cosocles counted with multiplicities. It has dimension  $1$ when $\pi_1' \simeq \pi_1$, so $\theta(\pi_1,V_2)$ is irreducible. It has dimension $0$ otherwise \textit{i.e.} $\textup{Hom}_{H_2}(\Theta(\pi_1',V_2),\theta(\pi_1,V_2)) = 0$  when $\pi_1'$ and $\pi_1$ are not isomorphic.  \end{proof}

Conversely, if two theta lifts share a common irreducible quotient, they must appear simultaneously on the boundary:

\begin{prop} \label{violating_theta_at_the_same_t_prop} Assume $\mathfrak{I}(-s) \twoheadrightarrow \Theta(\chi_2,V_1 \oplus (-V_1))$ in $\textup{Rep}_R(H_1 \times H_1)$. Let $\pi_1 \neq \pi_1'$ in $\textup{Irr}_R(H_1)$ and suppose $\Theta(\pi_1,V_2)$ and $\Theta(\pi_1',V_2)$ have a common irreducible quotient. If parabolic restriction preserves finite length representations, there exists $t >0$ such that $\pi_1$ and $\pi_1'$ occur at $t$ on the boundary of $\mathfrak{I}(-s)$. \end{prop}

\begin{proof} As in the previous proof, we use the see-saw identity to show $\mathfrak{I}(-s) \twoheadrightarrow \pi_1 \otimes (\pi_1')^{\delta_1}$. Therefore there exists $t$ such that $\upsilon_{t,-s} \rtimes \sigma_1 \twoheadrightarrow \pi_1$ and $\alpha \upsilon_{t,-s} \rtimes \sigma_1' \twoheadrightarrow (\pi_1')^{\delta_1}$. The latter implies $\upsilon_{t,-s} \rtimes \sigma_1' \twoheadrightarrow \pi_1'$ as in Proposition \ref{equivalences_boundary_prop}. Note that $\pi_1 \neq \pi_1'$ implies $t>0$ because $\mathfrak{I}_0(-s)_{\pi_1} = \pi_1 \otimes (\pi_1)^{\delta_1}$. \end{proof}

\subsection{Representations on the boundary}

\paragraph{Induced representations with irreducible socle.} We first need the following lemma, generalising \cite[5.2]{gt} to the modular setting. We drop the indices in the dual pair and simply write $H$ to mean $H_1$ or $H_2$. The quadratic character $\alpha(m) = (-1, \textup{det}(m))_F^n$ appearing below is due to the conjugation by the Weyl element realising the transpose-inverse-conjugate morphism on a genuine representation as it involves a twist by $\alpha$ in the identification with the conjugate-contragredient \textit{i.e.} ${}^w \rho \simeq \alpha \cdot {}^c \rho^\vee$.

\begin{lem} \label{irred_socle_criterion_lemma} Let $\rho \in \textup{Irr}_R^{\textup{scusp}}(G_r^{m_1})$ and $\sigma \in \textup{Irr}_R(H^{m - r a})$ with $a \in \mathbb{N}$ and set: 
$$\sigma_{\rho,a} = \rho^{\times a} \rtimes \sigma \in \textup{Rep}_R(H).$$
We assume $q_E^r \neq 1 \ \textup{mod} \ \ell$ and:
\begin{enumerate}[label=\textup{\alph*)}]
\item ${}^c \rho^\vee \not\simeq \alpha \rho$;
\item $\sigma \nsubseteq \rho \ltimes \sigma_0$ for all $\sigma_0$.
\end{enumerate}
Then, if parabolic restriction preserves finite length representations:
$$\mathfrak{r}_{X^{ra}}(\sigma_{\rho,a}) \simeq (\rho^{\times a} \otimes \sigma) \oplus T$$
where $T$ has no non-trivial subquotient of the form $\rho^{\times a} \otimes \sigma'$. In particular, the socle of $\sigma_{\rho,a}$ is irreducible \textit{i.e.} it contains a unique irreducible subrepresentation. \end{lem}

\begin{proof} First, the hypothesis $q_E^r \neq 1 \ \textup{mod} \ \ell$ implies $\rho^{\times a}$ is irreducible by \cite[V.3]{vigneras_induced_R_reps}. Second, as a consequence of Frobenius reciprocity, the induced representation $\mathfrak{i}_P^G(\tau)$ with $\tau \in \textup{Irr}_R(M)$ has irreducible socle if the semisimplification of $\mathfrak{r}_G^P(\mathfrak{i}_P^G(\tau))$ contains $\tau$ with multiplicity $1$. Therefore we simply need to prove the result about the $\rho^{\times a}$-isotypic part.

We use \cite{tadic,hanzer_muic} to interpret the geometric lemma. Because $\rho$ is cuspidal, irreducible subquotients of $\mathfrak{r}_{X^{ra}}(\sigma_{\rho,a})$ are of the form $\delta \otimes \sigma'$ where $\delta$ is a subquotient of $\rho^{\times k_3} \times (\alpha {}^c \rho^\vee)^{\times k_1} \times \tau$ with $k_1 + k_2 + k_3 = a$ and $\tau \otimes \sigma'$ is an irreducible subquotient of $\mathfrak{r}_{X^{rk_2}}(\sigma)$. Here we have used the fact that transpose-inverse realises the contragredient for irreducible representations of general linear groups. This fact is valid in the modular setting by the discussion at the end of Section \ref{invariant_distrib_modular_sec}.

Suppose $\delta \simeq \rho^{\times a}$. Then, by uniqueness of the supercuspidal support \cite[V.4]{vigneras_induced_R_reps}, we must have $k_1 = 0$ by assumption a) and $\tau \simeq \rho^{\times k_2}$. We now prove $k_2=0$ as a consequence of assumption b). Indeed, if $\rho^{\times k_2} \otimes \sigma'$ is a subquotient of $\mathfrak{r}_{X^{rk_2}}(\sigma)$ then $\mathfrak{r}_{X^{rk_2}}(\sigma)_{\rho^{\times k_2}} \neq 0$ because $\mathfrak{r}_{X^{rk_2}}(\sigma) \in \textup{Rep}_R(M)$ has finite length and the finite length category for $\textup{GL}_{rk_2}(E)$ is decomposed by supercuspidal support \cite{drevon_secherre}. But this would contradict assumption b), so $k_2 = 0$. Therefore we must have $k_1 = k_2 = 0$ and the subquotient corresponding to $(0,0,a)$ in the geometric lemma is associated to the closed Bruhat cell $P(X^{ra})$ \textit{i.e.} it is the ``evaluation at $1$ map'' $\mathfrak{r}_{X^{ra}}(\sigma_{\rho,a}) \twoheadrightarrow \rho^{\times a} \otimes \sigma$. This surjective map has a section as a consequence of the finite length category decomposition since there is no other $\rho^{\times a} \otimes \sigma'$ subquotient. \end{proof}

We have a nice uniqueness statement for the irreducible socle:

\begin{cor} \label{max_a_and_b_cor} We use the same notations and hypotheses as in the previous lemma. Let $\pi$ be the socle of $\sigma_{\rho,a}$. Let $b \geq a$ and $\delta \in \textup{Irr}_R(H^{m - rb})$ such that $\pi \subseteq \delta_{\rho,b} = \rho^{\times b} \rtimes \delta$. Then $b=a$ and $\delta \simeq \sigma$.  \end{cor}

\begin{proof} If we have $\pi \hookrightarrow \rho^{\times b} \rtimes \delta$, we obtain by Frobenius reciprocity a non-zero map $\mathfrak{r}_{X^{ra}}(\pi) \to \rho^{\times a} \otimes (\rho^{\times (b-a)} \rtimes \delta)$. By the previous lemma this map must factor through $\rho^{\times a} \otimes \sigma \hookrightarrow \rho^{\times a} \otimes (\rho^{\times (b-a)} \rtimes \delta)$. Therefore $b-a=0$ by assumption b) on $\sigma$ and $\delta \simeq \sigma$. \end{proof}

\begin{defi} Let $\pi \in \textup{Irr}_R(H)$ and $\rho \in \textup{Irr}_R^{\textup{scusp}}(G_r^{m_1})$ with $\alpha \cdot {}^c \rho^\vee \not\simeq \rho$ (\textit{i.e.} ${}^w \rho \not\simeq \rho$) and assume $q_E^r \neq 1 \ \textup{mod} \ \ell$. Corollary \ref{max_a_and_b_cor} tells us that the set of pairs $(a,\delta)$ such that $\pi \subseteq \delta_{\rho,a}$ has a maximal element in the following sense. We set $a_\rho(\pi) = \textup{max} \{ a \ | \ \exists \ \delta \textup{ s.t. } \pi \subseteq \delta_{\rho,a} \}$. There is a unique pair $(a_\rho(\pi),\sigma)$ such that $\pi \subseteq \sigma_{\rho,a_\rho(\pi)}$. \end{defi}

\paragraph{Theta lifts of irreducible socles.}  If $\pi_1$ occurs on the boundary of $\mathfrak{I}(-s)$, there exist $t>0$ and $\sigma_1 \in \textup{Irr}_R(H_1^{m_1-t})$ such that $\pi_1 \subseteq \alpha \upsilon_{t,-s}^\vee \rtimes \sigma_1$. Moreover:
$$1_{\textup{GL}_r(E)} \hookrightarrow | \cdot |^{-\frac{t-1}{2}} \times | \cdot |^{-\frac{t-3}{2}} \times \cdots \times | \cdot |^{\frac{t-1}{2}}$$
therefore:
$$\alpha \upsilon_{t,-s}^\vee  = \upsilon_{t,s-2t}	 \hookrightarrow  \upsilon_{1,s-2t} \times \upsilon_{1,s-2(t-1)} \times \cdots \times \upsilon_{1,s-2}.$$ 
We set $\rho_t = \upsilon_{1,s-2t} \in \textup{Irr}_R^{\textup{scusp}}(G_1^{m_1})$. Therefore $a_{\rho_t}(\pi_1) > 0$ if $\pi_1$ occurs at $t > 0$.

We will be able to deal with these representations living on the boundary at $t$ if the characteristic $\ell$ is good enough. We introduce the following definition. Let $0 < t \leq m_1$. We say $\ell$ is $(E,s,t)$-banal if the three conditions below hold:
\begin{enumerate}[label=\textup{\alph*)}]
\item $(q_E)^{s-2t+1} \neq 1 \ \textup{mod} \ \ell$;
\item $(q_E)^t \neq 1 \ \textup{mod} \ \ell$;
\item $(q_E)^{s-t} \neq 1 \ \textup{mod} \ \ell$.
\end{enumerate}

\begin{rem} In the complex setting, the order $q_E$ is infinite and the three conditions always hold since $s \leq 0$. In the modular setting, if we first fix the data $V_1$ and $V_2$, there are only finitely many $\ell$ such that one of the three conditions may fail to hold. \end{rem}

\begin{prop} \label{theta_lifts_irreducible_socles_prop} Let $0 < t \leq m_1$ and assume $\ell$ is $(E,s,t)$-banal. Let $\pi_1 \in \textup{Irr}_R(H_1)$ such that $(\omega_{V_1,V_2})_{\pi_1} \neq 0$ and $a_{\rho_t}(\pi_1) >0$. We set $\varrho_t = \nu_{1,-s+2t}$. Then for all irreducible quotients $\pi_2$ of $\Theta(\pi_1,V_2)$ we have:
\begin{itemize}[label=$\bullet$]
\item $a_{\rho_t}(\pi_1) = a_{\varrho_t}(\pi_2)$ and we denote it by $a$ for simplicity;
\item the unique embeddings $\pi_1 \subseteq (\sigma_1)_{\rho_t,a}$ and $\pi_2 \subseteq (\sigma_2)_{\varrho_t,a}$ induce:
$$0 \neq \textup{Hom}_{H_1 \times H_2}(\omega_{V_1,V_2},\pi_1 \otimes \pi_2) \hookrightarrow \textup{Hom}_{H_1^{m_1-a} \times H_2^{m_2-a}}(\omega_{V_1^{m_1-a},V_2^{m_2-a}},\sigma_1 \otimes \sigma_2).$$
\end{itemize}
\end{prop}

\begin{proof} None of the three assumptions can be realised if $q_E = 1 \ \textup{mod} \ \ell$, so we may assume $q_E \neq 1 \ \textup{mod} \ \ell$. We have $\alpha {}^c \rho_t^\vee = \rho_t$ if and only if $| \cdot |^{s-2t+1} = 1$, which is impossible by assumption a), and Lemma \ref{irred_socle_criterion_lemma} applies. Let $a=a_{\rho_t}(\pi_1)$ and $\pi \subseteq (\sigma_1)_{\rho_t,a}$. The latter embedding, together with Frobenius reciprocity and restriction to $J_a \hookrightarrow \mathfrak{r}_{X_1^a}(\omega_{V_1,V_2})$ in the filtration, gives a map:
$$\textup{Hom}_{H_1 \times H_2}(\omega_{V_1,V_2},\pi_1 \otimes \pi_2) \to \textup{Hom}_{M_a^{m_1} \times H_2}(J_a,\rho_t^{\times a} \otimes \sigma_1 \otimes \pi_2).$$
We claim this map is injective.

In order to prove it, we show that $\textup{Hom}_{M_a^{m_1} \times H_2}(J_k,\rho_t^{\times a} \otimes \sigma_1 \otimes \pi_2) = 0$ if $k \neq a$. As a representation of $G_a^{m_1}$ the subquotient $J_k$ is $\mathfrak{i}_{Q_{a-k}^a}^{G_a^{m_1}}(\upsilon_{a-k,s} \otimes \textup{Reg}^{\alpha_k,\xi_k}(G_1^k,G_2^k))$-isotypic. By uniqueness of supercuspidal support \cite[V.4]{vigneras_induced_R_reps} and since $q_E \neq 1 \ \textup{mod} \ \ell$, we obtain: 
$$\textup{Hom}_{G_a^{m_1}}(\mathfrak{i}_{Q_{a-k}^a}^{G_a^{m_1}}(\upsilon_{a-k,s} \otimes \textup{Reg}^{\alpha_k,\xi_k}(G_1^k,G_2^k)),\rho_t^{\times a}) = 0 \textup{ if } a-k > 1.$$
This homomorphism space may be non-zero when $k=a-1$. Actually, it is non-zero if and only if $\upsilon_{1,s} = \rho_t$ \textit{i.e.} $| \cdot |^t  = 1$. But assumption b) prevents it from happening, so $\textup{Hom}_{M_a^{m_1} \times H_2}(J_k,\rho_t^{\times a} \otimes \sigma_1 \otimes \pi_2) = 0$ if $k \neq a$.

By the existence of the previous injection, we get $\textup{Hom}_{M_a^{m_1} \times H_2}(J_a,\rho_t^{\times a} \otimes \sigma_1 \otimes \pi_2) \neq 0$. In particular $(J_a)_{\rho_t^{\times a} \otimes \sigma_1} \neq 0$. Recall $J_a = \mathfrak{i}_{P_a^{m_2}}^{H_2}( \textup{Reg}^{\alpha_a,\xi_a}(G_1^a,G_2^a) \otimes \omega_{V_1^{m_1-a},V_2^{m_2-a}})$, so:
$$(J_a)_{\rho_t^{\times a} \otimes \sigma_1} \simeq \rho_t^{\times a} \otimes \sigma_1 \otimes \mathfrak{i}_{P_a^{m_2}}^{H_2}((\rho_t^{\times a})^{[\alpha_a,\xi_a]} \otimes \Theta(\sigma_1,V_2^{m_2-a})).$$
We deduce $\mathfrak{i}_{P_a^{m_2}}^{H_2}((\rho_t^{\times a})^{[\alpha_a,\xi_a]} \otimes \Theta(\sigma_1,V_2^{m_2-a})) \twoheadrightarrow \pi_2$ where $(\rho_t^{\times a})^{[\alpha_a,\xi_a]} \simeq (\rho_t^{[\alpha_1,\xi_1]})^{\times a}$ and: 
$$\rho_t^{[\alpha_1,\xi_1]} = (\upsilon_{1,s-2t})^{[\alpha_1,\xi_1]} = \nu_{1,s-2t+2}.$$
We do not necessarily know that $\Theta(\sigma_1,V_2^{m_2-a})$ has finite length, but Casselman duality and the preservation of finite length by Jacquet functors ensure the existence of $\delta \in \textup{Irr}_R(H_2^{m_2-a})$ such that $(\nu_{1,s-2t+2})^{\times a} \rtimes \delta \twoheadrightarrow \pi_2$. As in the proof of Proposition \ref{equivalences_boundary_prop}, we obtain $\pi_2 \hookrightarrow \varrho_t^{\times a} \rtimes \delta$ with $\varrho_t = \alpha \nu_{1,s-2t+2}^\vee = \nu_{1,-s+2t}$. In particular $a_{\varrho_t}(\pi_2) \geq a$.

We can also apply the argument the other way round, starting with $b = a_{\varrho_t}(\pi_2)$ and $\pi_2 \subseteq (\sigma_2)_{\rho_t,b}$. This time the condition for $\textup{Hom}_{G_b^{m_2}}(\mathfrak{i}_{Q_{b-k}^b}^{G_b^{m_2}}(\nu_{b-k,s} \otimes \textup{Reg} \cdots),\varrho_t^{\times b}) \neq 0$ leads to $b-k=1$ and $| \cdot |^{s-t} \neq 1$, but assumption c) makes it impossible. Therefore $\pi_1 \subseteq \delta_{\rho_t,b}$ where $b \geq a_{\rho_t}(\pi_1)$, so we must have $b=a$ and $\delta \simeq \sigma_1$ by Corollary \ref{max_a_and_b_cor}.

To finish the proof, it is easy to see that the conclusions of Lemma \ref{irred_socle_criterion_lemma} still hold with $\alpha \cdot {}^c \varrho^\vee \otimes \sigma$ if we use the opposite parabolic restriction, and using Casselman duality: 
\begin{eqnarray*}
\textup{Hom}_{M_a^{m_1} \times H_2}(J_a,\rho_t^{\times a} \otimes \sigma_1 \otimes \pi_2) & \simeq & \textup{Hom}_{M_a^{m_2}}((\rho_t^{\times a})^{[\alpha_t,\xi_t]} \otimes \Theta(\sigma_1,V_2),\mathfrak{r}_{X_2^a}^{\textup{op}}(\pi_2)) \\
& \simeq &  \textup{Hom}_{H_2^{m_2-a}}(\Theta(\sigma_1,V_2),\sigma_2) \\
& \simeq & \textup{Hom}_{H_1^{m_1-a} \times H_2^{m_2-a}}(\omega_{V_1^{m_1-a},V_2^{m_2-a}}, \sigma_1 \otimes \sigma_2).
\end{eqnarray*}
This proves the proposition. \end{proof}

\subsection{Heredity of irreducibility and uniqueness}

We can now prove heredity for irreducibility and uniqueness.

\paragraph{Heredity of irreducibility.} We can use an induction argument for representations on the boundary thanks to Proposition \ref{theta_lifts_irreducible_socles_prop}.

\begin{theo} Assume $\pi_1 \in \textup{Irr}_R(H_1)$ occurs at $t >0$ on the boundary of $\mathfrak{I}(-s)$ and $\ell$ is $(E,s,t)$-banal. Set $a = a_{\rho_t}(\pi_1) > 0$ and consider the unique $\sigma_1 \in \textup{Irr}_R(H_1^{m_1-a})$ such that $\pi_1 \subset (\sigma_1)_{\rho_t,a}$. Then the maximal length of semisimple quotients satisfy:
$$\vartheta(\pi_1,V_2) \leq \vartheta(\sigma_1,V_2^{m_2-a}).$$ \end{theo}

\begin{proof} Let $\pi_2 \in \textup{Irr}_R(H_2)$ such that $\omega_{V_1,V_2} \twoheadrightarrow \pi_1 \otimes \pi_2$ \textit{i.e.} $\Theta(\pi_1,V_2) \twoheadrightarrow \pi_2$. Then, by Proposition \ref{theta_lifts_irreducible_socles_prop}, there exists an irreducible quotient $\sigma_2$ of $\Theta(\sigma_1,V_2^{m_2-a})$ such that $\pi_2 \subset (\sigma_2)_{\rho_t,a}$ and $a = a_{\rho_t}(\pi_2)$. Moreover, such a $\sigma_2$ is unique by Corollary \ref{max_a_and_b_cor} and, for another $\Theta(\pi_1,V_2) \twoheadrightarrow \pi_2'$, we have $\sigma_2' \simeq \sigma_2$ if and only if $\pi_2' \simeq \pi_2$. Note that the inclusion:
$$\textup{Hom}_{H_1 \times H_2}(\omega_{V_1,V_2},\pi_1 \otimes \pi_2) \hookrightarrow \textup{Hom}_{H_1^{m_1-a} \times H_2^{m_2-a}}(\omega_{V_1^{m_1-a},V_2^{m_2-a}},\sigma_1 \otimes \sigma_2)$$
tells us that the multiplicity of $\pi_2$ is always less than that of $\sigma_2$, so this implies the inequality $\vartheta(\pi_1,V_2) \leq \vartheta(\sigma_1,V_2^{m_2-a})$ by summing over isomorphism classes of such $\pi_2$. \end{proof}

As a direct consequence of this theorem and its proof, we obtain

\begin{cor} Take the same notations and hypotheses as in the previous theorem. Assume $\vartheta(\sigma_1,V_2^{m_2-a})=1$, then $\vartheta(\pi_1,V_2) \leq 1$. Moreover, when $\vartheta(\pi_1,V_2)=1$, we have: 
$$a = a_{\varrho_t}(\theta(\pi_1,V_2)) \textup{ and } \theta(\pi_1,V_2) \hookrightarrow \varrho_t^{\times a} \rtimes \theta(\sigma_1,V_2^{m_2-a}).$$ \end{cor}

\paragraph{Heredity of uniqueness.}

\begin{prop} \label{heredity_of_uniquenes_prop} Assume \textup{(H)} holds and $\ell$ is $(E,s,t)$-banal for all $m_1 \geq t >0$. Suppose $\Theta(\pi_1,V_2)$ and $\Theta(\pi_1',V_2)$ have a common irreducible quotient $\pi_2$. Then:
\begin{itemize}[label=$\bullet$]
\item there exists $t > 0$ such that $a_{\rho_t}(\pi_1) = a_{\rho_t}(\pi_1') = a_{\varrho_t}(\pi_2) >0$ and we denote it by $a$;
\item $\Theta(\sigma_1,V_2^{m_2-a})$ and $\Theta(\sigma_1',V_2^{m_2-a})$ have $\sigma_2$ as a common irreducible quotient where $\pi_1 \subseteq (\sigma_1)_{\rho_t,a}$, $\pi_1' \subseteq (\sigma_1')_{\rho_t,a}$ and $\pi_2 \subseteq (\sigma_2)_{\rho_t,a}$.
\end{itemize} \end{prop}

\begin{proof} By Proposition \ref{violating_theta_at_the_same_t_prop}, we can assume there exists $t >0$ such that $a_{\rho_t}(\pi_1) > 0$ and $a_{\rho_t}(\pi_1') > 0$. Proposition \ref{theta_lifts_irreducible_socles_prop} shows, on the one hand, that these two indices are equal to $a_{\rho_t}(\pi_2)$, and on the other hand, the claim on $\sigma_2$ as a common quotient of $\Theta(\sigma_1,V_2^{m_2-a})$ and $\Theta(\sigma_1',V_2^{m_2-a})$. \end{proof}

\begin{cor} In the context of Proposition \ref{heredity_of_uniquenes_prop}, if $\sigma_1$ is unique, \textit{i.e.} $\sigma_1 \simeq \sigma_1'$ when $\Theta(\sigma_1,V_2^{m_2-a})$ and $\Theta(\sigma_1',V_2^{m_2-a})$ have a common irreducible quotient, then $\pi_1 \simeq \pi_1'$. \end{cor}

\subsection{Proof of the correspondence in the banal setting} \label{proof_banal_theta_sec}

It is clear that all but finitely many $\ell$ are $(E,s,t)$-banal for all $m_1 \geq t >0$. Moreover, if the order $o_R(q_E)$ of $q_E$ in $R$ satisfies: 
$$o_R(q_E)> \underset{0<t \leq m_1}{\textup{max}}(-s+2t-1,t,-s+t)=-s+2m_1-1,$$
then $\ell$ is $(E,s,t)$-banal for all $m_1 \geq t >0$. Clearly, there are only finitely many $\ell$ such that $o_R(q_E) \leq -s+2m_1-1$. Note that $-s +2m_1 -1 \leq 2n_1$. We say $\ell$ is strongly banal with respect to $(U(V_1),U(V_2))$ if $o_R(q_E) > 2n_1$. In particular this condition guarantees that $\ell$ is banal with respect to $(U(V_1),U(V_2))$ \textit{i.e.} it does not divide the pro-orders of these groups.

\paragraph{Conditional banal theta correspondence.} We first prove the theta correspondence holds provided the assumption (H) holds and $\ell$ is large enough. Even though this first statement is conditional on (H), we will explain later in which cases we can relax it.

\begin{theo} \label{conditional_banal_theta_thm} Let $(U(V_1),U(V_2))$ be an irreducible dual pair of type I.  Assume: 
\begin{itemize}[label=$\bullet$]
\item $\ell$ is strongly banal with respect to the dual pair;
\item \textup{(H)} holds for all $(U(V_1^{m_1-a}),U(V_2^{m_2-a}))$ with $a \geq 0$.
\end{itemize}
Then:
\begin{enumerate}[label=\textup{\alph*)}]
\item $\Theta(\pi_1,V_2)$ has finite length for all $\pi_1 \in \textup{Irr}_R(H_1)$;
\item its cosocle $\theta(\pi_1,V_2)$ has length at most $1$;
\item $0 \neq \theta(\pi_1,V_2) \simeq \theta(\pi_1',V_2)$ if and only if $\pi_1' \simeq \pi_1$.
\end{enumerate}
Conversely, $\Theta(\pi_2,V_1)$ for $\pi_2 \in \textup{Irr}_R(H_2)$ satisfies \textup{a)-b)-c)}.
\end{theo}

\begin{proof} Finite length follows from Corollary \ref{lifts_of_cuspidals_banal_setting_cor} and Proposition \ref{finite_length_from_cuspidals_prop} because $\ell$ is strongly banal, so it is banal with respect to $U(V_1)$ and $U(V_2)$ because $s \leq 0$. Theorem \ref{rep_outside_boundary_thm} proves b) and c) when $\pi_1$ is outside the boundary of $\mathfrak{I}(-s)$.

We use an induction argument for representations on the boundary. We assume the theorem holds for all $(U(V_1^{m_1-a}),U(V_2^{m_2-a}))$ with $a>0$. Given the assumptions on (H) and $\ell$, the heredity corollaries for irreducibility and uniqueness ensure b) and c). 

We still have to prove the initial case of our induction argument \textit{i.e.} $a = \textup{min}(m_1,m_2)$. We claim that no irreducible representations $\pi_1$ on the boundary of $\mathfrak{I}(-s)$ can occur as a quotient of the Weil representation. Indeed, at least one of the two groups in the dual pair has no proper parabolic, therefore the condition $a_{\rho_t}(\pi_1) > 0 $ never happens as Proposition \ref{theta_lifts_irreducible_socles_prop} can't be applied.

Conversely, a) holds for $H_2$ because $\ell$ is banal with respect to $U(V_1)$ and $U(V_2)$. Moreover b)-c) can be obtained from b)-c) for $H_1$. \end{proof}

\paragraph{When does (H) hold?} We first prove that the hypothesis (H) is implied by the irreducibility of some theta lift.

\begin{lem} \label{H_is_implied_by_irred_Theta_lem} Assume $\Theta(\chi_2,V_1 \oplus (-V_1)) \in \textup{Rep}_R(H_1^\square)$ is irreducible. Then \textup{(H)} holds. \end{lem}

\begin{proof} We have $\Theta(\chi_2,V_1 \oplus (-V_1)) \subseteq \mathfrak{I}(s)$ by Rallis' argument. By applying MVW and contragredient to this embedding, we obtain the desired surjection. \end{proof}

Therefore it is sufficient to prove the irreducibility of $\Theta(\chi_2,V_1 \oplus (-V_1)) \in \textup{Rep}_R(H_1^\square)$. In the complex setting, since $s \leq 0$, the irreducibility follows from a series of papers \cite{kudla_rallis,kudla_sweet,sweet,ban_jantzen,yamana} where the authors study constituents of certain degenerate principal series such as $\mathfrak{I}(-s)$. However, this literature only seems, according to the forthcoming book \cite{theta_book}, to prove the irreducibility when $\textup{char}(F)=0$ because of the use of conservation relations, so we will also make this assumption. In the strategy to prove irreducibility, there are obvious obstacles such as the methods of \cite{li}, which rely crucially on some argument involving unitarity -- see \cite[Prop 3.1]{kudla_rallis} for instance. As these tools are not available in the modular setting, proving the irreducibility would require completely new methods, and therefore would require substantial developments we do not intend to pursue here. We use a different method instead, by an argument of reduction modulo $\ell$, but this method has a drawback: it is not explicit. It will only ensure (H) holds for all but finitely many $\ell$.

\begin{prop} \label{H_holds_for_almost_all_l_prop} Assume $\textup{char}(F) = 0$. Let $(U(V_1),U(V_2))$ be a type I commutative dual pair. For all but finitely many $\ell$, the hypothesis \textup{(H)} holds. \end{prop}

\begin{proof} We first start by proving a result we call generic irreducibility:
\begin{lem} Let $L \in \textup{Rep}_{\mathbb{Z}[1/p]}(G)$ a $\mathbb{Z}[1/p][G]$-lattice in the sense of \textup{\cite[I.9.1]{vig_book}} \textit{i.e.} the $\mathbb{Z}[1/p]$-module $L^K$ is finite free for all compact open subgroups $K$. Assume $L \otimes \mathbb{C}$ is irreducible, then $L \otimes \overline{\mathbb{F}}_\ell$ is irreducible for all but finitely many primes $\ell$. \end{lem} 

\begin{proof} The function $\textup{len} : \mathcal{P} \mapsto \textup{length}(L \otimes \overline{k_\mathcal{P}})$ on $\textup{Spec}(\mathbb{Z}[1/p])$ is upper semi-continuous in the Zariski topology, where $\overline{k_\mathcal{P}}$ is the algebraic closure of the residue field $k_\mathcal{P}$ at $\mathcal{P}$. Moreover $L \otimes \overline{k}_\mathcal{P} \neq 0$ for all $\mathcal{P}$, so the inverse image of $\{1\}$ by $\textup{len}$ is Zariski open. It is non-empty because it must contain the generic point by assumption. Therefore it is open and dense \textit{i.e.} it is the complement of a finite set. \end{proof}

In our situation, we want to consider a generalisation of the previous lemma for the ring $\mathcal{A} = \mathbb{Z}[1/p,\zeta_{p^\infty}]$ appearing in \cite{trias_theta2}. Its fraction field $\mathcal{K}$ is a Galois extension of $\mathbb{Q}$ that is unramified ouside $p$. Above a prime number $\ell \neq p$, there are actually finitely many prime ideals. The decomposition number $r_\ell$ for each $\ell$ is finite and can be described as follows. Let $a_\ell$ be the order of $\ell$ modulo $p$ and let $k_\ell$ be the $p$-adic valuation of $\ell^{a_\ell}-1$, then the number of prime ideals above $\ell$ is:
$$r_\ell = \frac{p-1}{a_\ell} p^{k_\ell-1} = \frac{\varphi(p^{k_\ell})}{a_\ell}.$$
Non-zero prime ideals of $\mathcal{A}$ coincide with maximal ideals and they are noetherian. As the ring $\mathcal{A}$ is also a Prüfer domain, it is Dedekind if and only if it is noetherian. Since all maximal ideals of $\mathcal{A}$ are noetherian, the localisation at a maximal ideal is a noetherian valuation ring \textit{i.e.} a discrete valuation ring. As a result, the ring $\mathcal{A}$ is a Dedekind domain.

We want to use the previous lemma, with $\mathcal{A}$ instead of $\mathbb{Z}[1/p]$, to prove a generic irreducibility result for $\Theta^R = \Theta(\chi_2,V_1 \oplus (-V_1))$. To apply the lemma, we first need to show that $\Theta^{\mathcal{K}}$ admits an integral model over $\mathcal{A}$, also known as an $\mathcal{A}[H_1^\square]$-lattice.

Recall that $\Theta^R$ is obtained in Theorem \ref{Rallis_argument_thm} as a space of functions from the evaluation at $0$ for functions in $\omega_{H_1^\square,H_2}^R \simeq C_c^\infty(Y_1 \otimes V_2,R)$. More precisely, the $U(V_2)$-coinvariants are given by the map: 
$$\phi^R : f \mapsto (h_1 \mapsto (h_1 \cdot f)(0)).$$
There is an integral version of it, as we now explain. 

The Weil representation over $\mathcal{A}$ is $\omega_{H_1^\square,H_2}^A \simeq C_c^\infty(Y_1 \otimes V_2,\mathcal{A})$. The map $\phi^\mathcal{A}$ is defined the exact same way and $U(V_2)$ acts trivially on its image, therefore $\phi^\mathcal{A}$ factors through:
$$(\omega_{H_1^\square,H_2}^\mathcal{A})_{U(V_2)} \twoheadrightarrow \textup{Im}(\phi^\mathcal{A}).$$
But this map must be an isomorphism, since it induces, after tensoring with the residue field at $\mathcal{P}$ for any prime ideal in $\mathcal{A}$, the isomorphism factorising $\phi^{k_\mathcal{P}}$. 

We set $\Theta^\mathcal{A} = (\omega_{H_1^\square,H_2}^\mathcal{A})_{U(V_2)}$. Then:
\begin{itemize}[label=$\bullet$]
\item $\Theta^\mathcal{A} = \textup{Im}(\phi^\mathcal{A}) \subset \Theta^\mathcal{K}$ where $\mathcal{K}$ is the fraction field of $\mathcal{A}$;
\item $\Theta^\mathcal{A} \otimes k_\mathcal{P} = \Theta^{k_\mathcal{P}}$ by compatibility of $\phi^\mathcal{A}$ to scalar extension.
\end{itemize}
Because $\Theta^\mathbb{C} = \Theta^\mathcal{K} \otimes \mathbb{C}$ is irreducible, it is admissible, so $\Theta^\mathcal{K}$ is admissible and absolutely irreducible. As $\mathcal{A}$ is noetherian, we deduce that $\Theta^\mathcal{A}$ is admissible and contains a $\mathcal{K}$-basis of $\Theta^{\mathcal{K}}$. So $\Theta^\mathcal{A}$ is an $\mathcal{A}[H_1^\square]$-lattice in the sense of \cite[I.9.1]{vig_book}.

Now we can apply the generalised version of the lemma to $\Theta^\mathcal{A}$. It tells us that there is a Zariski open dense subset of $\textup{Spec}(\mathcal{A})$ on which $\Theta^{k_\mathcal{P}}$ is irreducible. Its complement has finite image\footnote{This complement is actually finite because $\mathcal{A}$ is a Dedekind domain. When $\mathcal{A}$ is a Prüfer domain, the image would still be finite even though the complement may fail to be finite.} under the map $\textup{Spec}(\mathcal{A}) \to \textup{Spec}(\mathbb{Z}[1/p])$. In particular, for all but finitely many $\ell$, we have $\Theta^{\overline{k}_\mathcal{P}} = \Theta^{\overline{\mathbb{F}_\ell}}$ is irreducible. The compatibility of the map $\phi$ with arbitrary field extensions yields the result for all $R$, as $\Theta^R = \Theta^{\overline{\mathbb{F}_\ell}} \otimes R$ and $\Theta^{\overline{\mathbb{F}_\ell}}$ is absolutely irreducible. \end{proof}

\paragraph{Unconditional theta correspondence.}  Building on Lemma \ref{H_is_implied_by_irred_Theta_lem} and Proposition \ref{H_holds_for_almost_all_l_prop}, we obtain: 
 
\begin{theo} \label{non_explicit_banal_theta_thm} Assume $\textup{char}(F) = 0$. Let $(U(V_1),U(V_2))$ be a type I commutative dual pair. For all but finitely many $\ell$, the consequences of Theorem \ref{conditional_banal_theta_thm} hold. \end{theo}

In order to make this result more explicit, and as we have already mentioned, it would require to conduct, in the modular setting, a thorough study of reducibility points and composition factors of degenerate principal series similar to the series of papers \cite{kudla_rallis,kudla_sweet,ban_jantzen,yamana}. We conjecture the theorem holds if $\ell$ is strongly banal.

\subsection{A counter-example in the non-banal setting} \label{counter_example_sec}

We give a counter-example to the irreducibility of the co-socle of $\Theta(\pi_1,V_2)$ when $\ell$ is non-banal. Let $(U(V_1),U(V_2))$ be a type I dual where:
\begin{itemize}
\item $V_1$ is symplectic of dimension $2$;
\item $V_2$ is split orthogonal of even dimension. 
\end{itemize} 
We will find $\ell$ such that $\Theta(\chi_2,V_1)$ is semisimple and has length $2$. First, note that $H_1 = \textup{Sp}(V_1) \times R^\times$ because $n_2$ is even. Let $s = n_2 - m_1 -1 = n_2 -3$ and recall: 
$$\mathfrak{I}^R(s) = \mathfrak{i}_{P_1}^{H_1}(| \cdot |^{\frac{-n_1+m_1+\eta_1}{2}}\nu_{X_1,n_2}) = \mathfrak{i}_{P_1}^{H_1}(| \cdot |^{\frac{s}{2}}\delta_{X_1,n_2}) = \mathfrak{i}_{P_1}^{H_1}(| \cdot |^{\frac{s}{2}}) \in \textup{Rep}_R(H_1).$$
Though the last representation is in $\textup{Rep}_R(\textup{Sp}(V_1))$, it is identified with a representation of $H_1$ in an obvious way. In the so-called stable range, this representation is irreducible:

\begin{prop}[\cite{kudla_rallis}] If $|s| > \frac{m_1+1}{2} = 1$, \textit{i.e.} $m_2 \geq 2$, then $\mathfrak{I}^\mathbb{C}(s)$ is irreducible. In particular $\Theta^\mathbb{C}(\chi_2,V_1) = \mathfrak{I}^\mathbb{C}(s)$. \end{prop}

\noindent  From now on we suppose we are in the stable range \textit{i.e.} $m_2 \geq 2$. In the proof of Proposition \ref{H_holds_for_almost_all_l_prop}, we have shown $\Theta^\mathbb{C} = \Theta^\mathbb{C}(\chi_2,V_1)$ has an integral model $\Theta^\mathcal{A}$ where $\mathcal{A} = \mathbb{Z}[1/p,\zeta_{p^\infty}]$. By completing at a maximal ideal $\mathcal{P}$ above $\ell$, we obtain a model over the Witt vectors since $\hat{\mathcal{A}} \simeq  W(\mathbb{F}_\ell(\zeta_{p^\infty}))$. Write $\hat{\mathcal{K}}$ for the field of fractions of $\hat{\mathcal{A}}$.

Moreover, Rallis' argument showed $\Theta^R \subset \mathfrak{I}^R(s)$. Since we know $\Theta^{\hat{\mathcal{K}}} = \mathfrak{I}^{\hat{\mathcal{K}}}(s)$ by the proposition, the Brauer-Nesbitt principle \cite[I.9.6]{vig_book} ensures that the representations $\Theta^{\mathbb{F}_\ell(\zeta_{p^\infty})}(\chi_2,V_1)$ and $\mathfrak{I}^{\mathbb{F}_\ell(\zeta_{p^\infty})}(s)$ have same length, which is finite, so they are equal by Rallis' argument. This shows:

\begin{prop} If $m_2 \geq 2$, then $\Theta^R(\chi_2,V_1) = \mathfrak{I}^R(s)$. \end{prop}

We can study the reducibility points and the structure of $\mathfrak{I}^R(s) \in \textup{Rep}_R(\textup{Sp}(V_1)$ with respect to $\ell$ following the example at the end of \cite{dat} because $\textup{Sp}(V_1) \simeq \textup{SL}_2(F)$. As $|\varpi_F|^{\frac{s}{2}} = q^{\frac{1}{2}} q^{m_1-2}$, we have:
\begin{enumerate}
\item if $q = -1 \mod \ell$, then $\mathfrak{I}^R(s)$ is irreducible as $|\varpi_F|^{\frac{s}{2}} \neq -1$;
\item if $q = 1 \mod \ell$, then $\mathfrak{I}^R(s)$ is reducible semisimple as $|\varpi_F|^{\frac{s}{2}} \in \{ \pm 1\}$;
\item if $q \neq \pm 1 \mod \ell$, then $\mathfrak{I}^R(s)$ is irreducible as $|\varpi_F|^{\frac{s}{2}} \notin \{ -1,q,q^{-1}\}$.
\end{enumerate}
Note that these techniques could prove (H) in the stable range, but it requires to know reducibility points and composition series of these degenerate principal series in the modular setting. Finally, our counter-example comes from the second case:
$$\Theta^R(\chi_2,V_1) \textup{ has length $2$ and is semisimple if } q = 1 \mod \ell.$$
Furthermore, we can describe its composition factor in more details: 
\begin{enumerate}
\item if $\ell = 2$, then $|\varpi_F|^{\frac{s}{2}}=1$ and the trivial representation appear in $\mathfrak{i}_B^{\textup{SL}_2(F)}(1)$, so it is of the form $1_{\textup{SL}_2(F)} \oplus \pi$ where $\pi$ is the so-called special representation, which is infinite dimensional;
\item if $\ell \neq 2$, then $|\varpi_F|^{\frac{s}{2}}=-1$ and it consists of two inequivalent infinite dimensional representations $\pi \oplus \pi'$ as the geometric lemma ensures that the trivial does not appear and $\textup{End}_{\textup{SL}_2(F)}(\mathfrak{i}_B^{\textup{SL}_2(F)}(\chi))$ has dimension at most $2$ for all characters $\chi$.
\end{enumerate}

\appendix

\section{Compatibilities for the Weil representation} \label{compatibility_Weil_rep_app}

The lemmas below come from straightforward and elementary calculations related to the Weil representation and the metaplectic group, so their proofs are omitted.

\subsection{Replacing $W$ by $W^a$}

Let $(W,\langle \ , \ \rangle)$ be a symplectic space over $F$. For $a \in F^\times$, we define the symplectic space $W^a$ whose underlying vector space is $W$ and whose symplectic product is $a \langle \ , \ \rangle$. We let $H^a$ be the Heisenberg group associated to $W^a$. It is isomorphic to $H$ via:
$$\begin{array}{cccc}
\iota^a : & H & \to & H^a \\
& (w,t) & \mapsto & (w,at)
\end{array}.$$
We denote by $S_{\psi,X}^a$ the Heisenberg representation of $H^a$ \textit{i.e.} smooth functions on $H^a$ such that $f((x,t) h) = \psi(t) f(h)$ for all $x \in X$, $t \in F$ and $h \in H^a$. This is compatible with the Heisenberg representation of $H$ in a sense we now explain. Let $\psi^a : t \mapsto \psi(a^{-1} t)$. Then the following map is an $\iota^a$-equivariant isomorphism:
$$\begin{array}{cccc}
\phi_{\iota^a} : & S_{\psi,X}^a & \to & S_{\psi^a,X} \\ 
& f & \mapsto & f \circ \iota^a
\end{array}$$
\textit{i.e.} $\phi_{\iota^a} \circ \rho_{\psi,X}^a(\iota^a(h)) = \rho_{\psi^a,X}(h) \circ \phi_{\iota^a}$ for all $h \in H$. 

\begin{lem} Let $\textup{Mp}^{c_X^a}(W^a)$ and $\textup{Mp}^{c_X}(W)$ be the metaplectic groups associated to the Lagrangian $X$. The canonical isomorphism of central extensions $\gamma_a$ makes the following diagram commute:
$$\xymatrix{
		\textup{Mp}^{c_X^a}(W^a) \ar@{->}[r]^{\omega_{\psi,X}^a} \ar@{->}[d]^{\gamma_a} & \ar@{->}[d]^{\phi_{\iota^a} (-) \phi_{\iota^a}^{-1}} \textup{GL}(S_{\psi,X}^a) \\
		 \textup{Mp}^{c_X}(W) \ar@{->}[r]^{\omega_{\psi^a,X}} & \textup{GL}(S_{\psi^a,X}) 
		}$$
where $\gamma_a : (g,\lambda) \mapsto (g,\lambda \lambda_a(g))$ and $\lambda_a(g) \in R^\times$ is characterised by the relation:
$$\phi_{\iota^a} \circ \omega_{\psi,X}^a(g,1) = \lambda_a(g) \times \omega_{\psi^a,X}(g,1) \circ \phi_{\iota^a}.$$ \end{lem}

\begin{rem} In particular the Weil representations $(\omega_{\psi,X}^-,S_{\psi,X}^-) \in \textup{Rep}_R(\textup{Mp}^{c_X^-}(W^-))$ and $(\omega_{\psi^{-1},X},S_{\psi^{-1},X}) \in \textup{Rep}_R(\textup{Mp}^{c_X}(W))$ can be identified via $\gamma_{-1}$. In other words:
$$(\omega_{\psi,X}^- \circ \gamma_{-1} , S_{\psi,X}^- ) \simeq (\omega_{\psi,X},S_{\psi,X}).$$ \end{rem}

\begin{rem} From a computational perspective, the constant $\lambda_a(g)$ is almost impossible to determine without further assumptions on $g \in \textup{Sp}(W)$. However, in a number of cases, the specific $g$'s we consider are nice enough to carry out the computation. \end{rem}

\subsection{Action of $\textup{GSp}(W)$}

Let $\delta \in \textup{GSp}(W)$ and denote by $a \in F^\times$ its similitude factor \textit{i.e.} $\langle \delta w , \delta w' \rangle = a \langle w , w' \rangle$ for all $w, w' \in W$.  In particular $\delta : w \mapsto \delta w$ realises an isometry $W^a \to W$.

\subsubsection{Metaplectic group}

Let $X$ be a Lagrangian in $W$ and let $\textup{Mp}^{c_X}(W)$ be the metaplectic group associated to it. There exists a unique automorphism of the metaplectic group lifting $g \mapsto \delta g \delta^{-1}$ and we denote it $\gamma_\delta$. To be explicit:
$$\begin{array}{cccc}
\gamma_\delta : & \textup{Mp}^{c_X}(W) & \to & \textup{Mp}^{c_X}(W) \\
 & (g,\lambda) & \mapsto & (\delta g \delta^{-1}, \lambda \times \lambda_\delta(g))
\end{array}$$
and $\lambda_\delta$ is uniquely characterised by $\partial \lambda_\delta \cdot c_X = c_X^\delta$ that is:
$$\lambda_\delta(g) \lambda_\delta(g') c_X(\delta g \delta^{-1}, \delta g' \delta^{-1}) = \lambda_\delta(g g') c_X(g,g').$$
By unicity, $\gamma_\delta \circ \gamma_{\delta'} = \gamma_{\delta \delta'}$, for $\delta, \delta' \in \textup{GSp}(W)$, and we deduce for $\lambda$ the relation:
$$\lambda_{\delta \delta'}(g) = \lambda_{\delta'}(g) \lambda_\delta(\delta' g \delta'^{-1}).$$
If $w \in \textup{Sp}(W)$, then $\gamma_w(g,\lambda) = (w,1)(g,\lambda)(w,1)^{-1}$ is the conjugation in the metaplectic group. Therefore $\lambda_w(g) = c_X(w,w^{-1}) c_X(w,g) c_X(wg,w^{-1})$. Moreover if $p \in P(X)$, we have $\gamma_p(g,\lambda) = (p,1) (g,1) (p,1)^{-1} = (p g p^{-1},\lambda)$ \textit{i.e.} $\lambda_p(g)=1$. Alternatively $c_X$ is invariant by $P(X)$ \textit{i.e.} $c_X(pgp^{-1},pg'p^{-1})=c_X(g,g')$. As a result:
$$\lambda_{p \delta} (g) = \lambda_\delta(g) \textup{ and } \lambda_{\delta p}(g) = \lambda_\delta(pg p^{-1}).$$
The morphism $\gamma_\delta$ can be factored as the composition of two morphisms as follows. Let $X' = \delta(X)$ and consider $\textup{Mp}^{c_{X'}^a}(W^a)$. Then $(g,\lambda) \mapsto (\delta g \delta^{-1},\lambda)$ is an isomorphism:
$$\textup{Mp}^{c_X}(W) \to \textup{Mp}^{c_{X'}^a}(W^a)$$
\textit{i.e.} $c_X(\delta g \delta^{-1} , \delta g' \delta^{-1}) = c_{X'}^a(g,g')$. Now, its composition with the canonical morphism of central extensions:
$$\begin{array}{ccc}
 \textup{Mp}^{c_{X'}^a}(W^a) & \to & \textup{Mp}^{c_X}(W) \\
  (g,\lambda) & \mapsto & (g, \lambda \times \lambda_{X,X'}^a(g)) 
\end{array}$$
gives $\gamma_\delta$. In particular $\lambda_{X,X'}^a(\delta g \delta^{-1}) = \lambda_\delta(g)$.

\subsubsection{Weil representation}

It is easy to check that the following map is an isomorphism of the Heisenberg group:
$$\begin{array}{cccc}
\iota_\delta : & H & \to & H \\
 & (w,t) & \mapsto & (\delta w , at)
\end{array}$$
Let $\Phi_{X,\delta^{-1}(X)} : S_{\psi,\delta^{-1}(X)} \to S_{\psi,X}$ be the change of model for the Heisenberg representation. Similarly to the previous paragraph, there is an $\iota_\delta$-equivariant isomorphism:
$$\begin{array}{cccc}
\phi_{\iota_\delta} : & S_{\psi^a,X} & \to & S_{\psi,X} \\
 & f & \mapsto & \Phi_{X,\delta^{-1}(X)} (f \circ \iota_{\delta})
\end{array}$$
\textit{i.e.}  $\phi_{\iota_\delta} \circ \rho_{\psi^a,X}(\iota_\delta(h)) = \rho_{\psi,X}(h) \circ \phi_{\iota_\delta}$ for all $h \in H$. This leads to the following compatibility:

\begin{lem} The canonical isomorphism $\gamma_\delta$ of the metaplectic group, lifting conjugation by $\delta$ on $\textup{Sp}(W)$, makes the diagram below commute:
$$\xymatrix{
		\textup{Mp}^{c_X}(W) \ar@{->}[r]^{\omega_{\psi^a,X}} \ar@{->}[d]^{\gamma_\delta} & \ar@{->}[d]^{\phi_{\iota_\delta} (-) \phi_{\iota_\delta}^{-1}} \textup{GL}(S_{\psi^a,X}) \\
		 \textup{Mp}^{c_X}(W) \ar@{->}[r]^{\omega_{\psi,X}} & \textup{GL}(S_{\psi,X}) 
		}$$
where $\gamma_\delta : (g,\lambda) \mapsto (\delta^{-1} g \delta,\lambda \lambda_\delta(g))$ and $\lambda_\delta(g) \in R^\times$ is characterised by the relation:
$$\phi_{\iota_\delta} \circ \omega_{\psi^a,X}(g,1) = \lambda_\delta(g) \times \omega_{\psi,X}(\delta^{-1} g \delta,1) \circ \phi_{\iota_\delta},$$
where $a$ is the similitude factor of $\delta$. \end{lem}

\begin{rem} We often write $\omega_{\psi,X}^\delta$ to mean the pullback along the canonical map $\gamma_\delta$. In this case, this gives the compatibility:
$$(\omega_{\psi,X}^\delta,S_{\psi,X}) \simeq (\omega_{\psi^a,X},S_{\psi^a,X}).$$ \end{rem}

\subsection{Dual pairs} \label{dual_pairs_weil_rep_appendix}

Let $W=V_1 \otimes V_2$ together with anisotropic decompositions $V_1 = X_1 + V_1^0 + Y_1$ and $V_2 = X_2 + V_2^0 + Y_2$. Let $\iota_{\psi,H_1}$ be the embedding of $H_1=U(V_1) \times_{c_1} R^\times$ into $\textup{Mp}^{c_{X^W}}(W)$ where $X^W = X^0 \oplus X^\mathbb{H}$ with $X^0 \in \Omega(V_1 \otimes V_2^0)$ and $X^\mathbb{H} = V_1 \otimes X_2 \in \Omega(V_1 \otimes V_2^\mathbb{H})$. Recall that $c_1$ is the $2$-cocyle induced by $c_{X^0}$ on $U(V_1)$ by restriction. We consider the Weil representation $(\omega_{\psi,X^W} \circ \iota_{\psi,X^W}) \in \textup{Rep}_R(H_1)$ and denote its isomorphism class by $\omega_{\psi,H_1}$. Therefore the realisation of $\omega_{\psi,H_1}$ on another model $S_{\psi,X'}$ is the composition:
$$H_1 \overset{\iota_{\psi,H_1}}{\to} \textup{Mp}^{c_{X^W}}(W) \overset{\gamma_{X',X}}{\to} Mp^{c_{X'}}(W) \overset{\omega_{\psi,X'}}{\to} \textup{GL}(S_{\psi,X'}).$$
When we modify $W$ to $W^a$ we get $\omega_{\psi,H_1^a}^a$ given by $(\omega_{\psi,X}^a \circ \iota_{\psi,H_1^a}^a,S_{\psi,X}^a)$ where $H_1^a  = U(V_1) \times_{c_1^a} R^\times$ and $c_1^a$ is induced by $c_{X^0}^a$.

\begin{lem} The canonical isomorphism of central extensions $\gamma_a$ descends to $H_1$ and $H_1^a$ by inducing an isomorphism of central extensions still denoted $\gamma_a$. To be more precise, the following diagram commutes :
$$\xymatrix{
	H_1^a \ar@{->}[r]^-{\iota_{\psi,H_1^a}} \ar@{->}[d]^{\gamma_a} &	\textup{Mp}^{c_{X^W}^a}(W^a) \ar@{->}[d]^{\gamma_a} \\
		H_1 \ar@{->}[r]^-{\iota_{\psi^a,H_1}} & \textup{Mp}^{c_{X^W}}(W)
		}$$
where $\gamma_a : (u_1,\lambda_1) \mapsto (u_1,\lambda_1 \lambda_a(u_1)))$ and $\lambda_a(u_1) \in R^\times$ is characterised by the relation:
$$\phi_{\iota^a} \circ \omega_{\psi,H_1^a}^a(u_1,1) = \lambda_a(u_1) \times \omega_{\psi^a,H_1}(u_1,1) \circ \phi_{\iota^a}.$$  \end{lem}

\begin{rem} \label{pullback_for_-1_remark_annex} This means $\omega_{\psi,H_1^-}^-$ can be identified with $\omega_{\psi^{-1},H_1}$ thanks to $\gamma_{-1}$. \end{rem}

\begin{lem} Let $\delta \in \textup{GSp}(W)$ and assume $\delta$ preserves $U(V_1)$ by conjugation \textit{i.e.} $u_1 \mapsto \delta^{-1} (u_1 \otimes \textup{id}_{V_2})\delta$ induces an automorphism of $U(V_1)$ denoted $u_1 \mapsto u_1^\delta$. Then the canonical isomorphism of central extensions $\gamma_\delta$ descends to an automorphism of $H_1$ still denoted $\gamma_\delta$ such that the diagram below commutes:
$$\xymatrix{
	H_1 \ar@{->}[r]^-{\iota_{\psi^a,H_1}} \ar@{->}[d]^{\gamma_\delta} &	\textup{Mp}^{c_{X^W}}(W) \ar@{->}[d]^{\gamma_\delta} \\
		H_1 \ar@{->}[r]^-{\iota_{\psi,H_1}} & \textup{Mp}^{c_{X^W}}(W)
		}$$
where $\gamma_\delta : (u_1,\lambda_1) \mapsto (u_1^\delta,\lambda_1 \lambda_\delta(u_1))$ and $\lambda_\delta(u_1) \in R^\times$ is characterised by the relation:
$$\phi_{\iota_\delta} \circ \omega_{\psi^a,H_1}(u_1,1) = \lambda_\delta(u_1) \times \omega_{\psi,H_1}(u_1^\delta,1) \circ \phi_{\iota_\delta},$$
where $a$ is the similitude factor of $\delta$.
\end{lem}

\begin{rem} This means $\omega_{\psi,H_1}^\delta$ is isomorphic to $\omega_{\psi^a,H_1}$. \end{rem}

\begin{rem} Note that $\delta$ preserves $U(V_1)$ if and only if $\delta$ preserves $U(V_2)$ as they are mutual centralisers in the symplectic group. \end{rem}

\section{On finite length of theta lifts} \label{finite_length_app}

Our proofs in this section will be based on finiteness properties, involving finite generation and admissibility, to deduce finite length of theta lifts. The following criterion for finite length is often very helpful. By \cite[II.5.10]{vig_book} it works for any reductive group $G$ with a compact open subgroup of invertible pro-order.

\begin{lem} Let $V \in \textup{Rep}_R(G)$. If $V$ is finitely generated and admissible over $R$, then $V$ has finite length. \end{lem}

\noindent As a consequence of the lemma, our two main strategies revolve around the corollary:

\begin{cor} \label{finite_length_criterion_A_cor} Let $A$ be a commutative $R$-algebra and $V \in \textup{Rep}_A(G)$. If $V$ is finitely generated and admissible over $A$, then for all characters $\eta : A \to R$, the representation $V_\eta = V \otimes_{A,\eta} R \in \textup{Rep}_R(G)$ has finite length. \end{cor}

Even though the Weil representation is not admissible over $R$, it turns out to be admissible over a a bigger commutative $R$-algebra $A$. Moreover, the largest $\pi$-isotypic quotient $V \twoheadrightarrow V_\pi$, where $\pi \in \textup{Irr}_R(G)$ and $V \in \textup{Rep}_A(G)$, will factor through $V_{\eta_\pi}$ where $\eta_\pi : A \to R$ is the character associated to $\pi$ by Schur's lemma. When the corollary above applies, the largest isotypic quotients have finite length as a consequence of the factorisation $V_{\eta_\pi} \twoheadrightarrow V_\pi$. This is at the heart of our first strategy. The second strategy is very similar, except that it uses the generalised doubling method and the induced representation $I(s)$ instead of the Weil representation.

However, both arguments require the so-called finiteness of Hecke algebras, due to Bernstein in the complex setting and \cite{dhkm-finiteness} in the modular setting:

\begin{theo}[Finiteness of Hecke algebras] Let $G$ be a reductive group over $F$. Then for all compact open subgroups $K$ in $G$, we have:
\begin{itemize}[$\bullet$]
\item the relative Hecke algebra $\mathcal{H}(G,K)$ is finite over its centre $\mathfrak{z}(G,K)$;
\item the centre $\mathfrak{z}(G,K)$ is a finitely generated $R$-algebra.
\end{itemize}
\end{theo}

In particular $\mathfrak{z}(G,K)$ is noetherian and $\mathcal{H}(G,K)$ is noetherian. Together with the depth decomposition, it implies that $\textup{Rep}_R(G)$ is noetherian and all finitely generated representation are $\mathfrak{z}(G)$-admissible. This theorem is obtained by very different strategies in the modular setting and in the complex/banal one. Indeed, in the classical theory, noetherianity of the category and second adjunction appear rather early as consequences of faithfulness of Jacquet functors \cite[Lem 29]{bernstein_notes}. Then building explicit progenerators show the theorem above.

In the modular setting, the classical proofs of noetherianity and second adjunction break down, because Jacquet functors are no longer faithful -- it is related to the fact that the cuspidal part is not necessarily a direct factor. The strategy actually goes the other way round: the first step consists in proving the finiteness theorem and then use it to deduce second adjunction and other results. For modular representations -- or more generally, representations in families -- the proof of the theorem uses profound results from Fargues-Scholze about excursion operators.

We assumed $G$ to be reductive in the previous theorem, which means the theorem is \textit{a priori} not known for the metaplectic group. In the complex setting, the author is not aware of an actual reference for the depth decomposition of $\textup{Rep}_\mathbb{C}(\textup{Mp}(W))$, even though it may be widely accepted among experts as the theory should work identically. The same remark applies for the finiteness of Hecke algebras. In the modular case, the depth decomposition should still hold, though it would need to be carefully written. However, the excursion operators of Fargues-Scholze have a very geometric flavour and are not available for the metaplectic group as it escapes the realm of reductive algebraic groups. This is the reason why we exclude the metaplectic case below, even though all our arguments would work equally well if both the depth decomposition and the finiteness result were fully known.

\subsection{Finite length when $p\neq 2$}

We assume $H_1$ is not the metaplectic group. For $\pi_1 \in \textup{Irr}_R(H_1)$ and $\eta_{\pi_1} : \mathfrak{z}_{H_1} \to R$ the associated character induced by Schur's lemma, the biggest $\pi_1$-isotypic quotient always factors through the biggest $\eta_{\pi_1}$-isotypic quotient \textit{i.e.} $\omega_{\eta_{\pi_1}} \twoheadrightarrow \omega_{\pi_1}$. We want to apply the previous key corollary to the Weil representation $\omega$ to obtain that $\omega_{\eta_{\pi_1}}$ has finite length, which implies $\omega_{\pi_1}$ has finite length. As we make use of the lattice model in our argument, we have to avoid even residual characteristic.

\paragraph{Finiteness properties of $\omega_{m_1,m_2}$.} We prove the following finiteness properties about the Weil representation.

\begin{lem} \label{weil_rep_is_adm_over_centre_and_fin_gen_lem} Assume $p \neq 2$. Then:
\begin{itemize}[$\bullet$]
\item the Weil representation $\omega_{m_1,m_2}$ is admissible over $\mathfrak{z}_{H_1}$;
\item the depth $k$ factor $e_k^1 \omega_{m_1,m_2} \in \textup{Rep}_R(H_1 \times H_2)$ is finitely generated for all $k$.
\end{itemize} \end{lem}

\begin{proof} Let $K_2$ be a compact open subgroup of $H_2$. To be on the safe side, we may assume our open compact subgroups are small enough so that their pro-order is invertible in $R$. Then by the computation of invariants in \cite[III.1 Th]{wald}, that we can still carry out in the modular setting assuming $K_2$ is small enough, there exists a finite-dimensional $R$-vector space $V$ in $\omega_{m_1,m_2}$ such that:
$$\omega_{m_1,m_2}^{K_2} = \mathcal{H}(H_1) \cdot V.$$
Here this computation is obtained in the lattice model, assuming $p \neq 2$. Therefore there exists a compact open subgroup $K_1$ in $H_1$ fixing $V$ \textit{i.e.} $V \subseteq \omega_{m_1,m_2}^{K_1 \times K_2}$, and in particular $\omega_{m_1,m_2}^{K_1 \times K_2} = \mathcal{H}(H_1,K_1) \cdot V$. Because $V$ has finite dimension, the right-hand side is finitely generated over the Hecke algebra $\mathcal{H}(H_1,K_1)$. We conclude by the finiteness of Hecke algebras over their centres which claims that $\mathcal{H}(H_1,K_1)$ is a finite $\mathfrak{z}_{H_1}$-algebra. 

Regarding the second claim, choose a compact open subgroup $K_1$ in $H_1$ such that all depth $k$ representations are generated by their $K_1$-fixed vectors. As in the previous paragraph there exists a finite dimensional $V$ in $\omega_{m_1,m_2}$ such that:
$$\omega_{m_1,m_2}^{K_1} = \mathcal{H}(H_2) \cdot V.$$
But $e_k^1 \omega_{m_1,m_2} = e_k^1 \mathcal{H}(H_1) \omega_{m_1,m_2}^{K_1} = \mathcal{H}(H_1) e_k^1(K_1) \omega_{m_1,m_2}^{K_1} \in \textup{Rep}_R(H_1 \times H_2)$ is a direct factor of $\mathcal{H}(H_1) \mathcal{H}(H_2) \cdot V$ and the latter is finitely generated.  \end{proof}

\begin{rem} We can always invert the roles of $H_1$ and $H_2$, provided $H_2$ is not the metaplectic group, so the result is equally true with $\mathfrak{z}_{H_2}$ in place of $\mathfrak{z}_{H_1}$. \end{rem}

\paragraph{Finitess of $\Theta(\pi_1,V_2^{m_2})$ when $p \neq 2$.} We finally obtain finiteness of the big theta lift.

\begin{theo} \label{finite_length_theta_p_not_2_thm} Assume $p \neq 2$. Let $\pi_1 \in \textup{Irr}_R(H_1)$. Then $\Theta(\pi_1,V_2^{m_2})$ has finite length. \end{theo}

\begin{proof} The biggest $\pi_1$-isotypic quotient $\omega_{\pi_1} \simeq \pi_1 \otimes \Theta(\pi_1,V_2^{m_2})$ factors through $e_k^1 \omega$ where $k$ is the depth of $\pi_1$. Thanks to Lemma \ref{weil_rep_is_adm_over_centre_and_fin_gen_lem}, we can apply Corollary \ref{finite_length_criterion_A_cor} to the Weil representation, setting $V=e_k^1 \omega$ and $A=\mathfrak{z}_{H_1}$.  Therefore $\omega_{\eta_{\pi_1}} \in \textup{Rep}_R(H_1 \times H_2)$ has finite length, and because it surjects on $\omega_{\pi_1}$, we obtain that $\Theta(\pi_1,V_2^{m_2})$ has finite length. \end{proof}

\subsection{Finite length from any index} \label{finite_length_from_any_index_sec}

We assume that neither of the groups $H_1$ or $H_2$ is the metaplectic group \textit{i.e.} neither $V_1$ nor $V_2$ is odd orthogonal. We are going to show that $\Theta(\pi_1,V_2^{m_2})$ has finite length for all $m_2$ if it has finite length for some $m_2 \geq m_2(\pi_1)$. Actually our proposition is a bit stronger as we only need to assume that $\Theta(\pi_1,V_2^{m_2})$ admits an irreducible quotient for some $m_2$. There is also an intermediate variation with finite generation since: 
$$\textup{non-zero finite length} \Rightarrow \textup{non-zero finitely generated} \Rightarrow \textup{admits an irreducible quotient}.$$

\paragraph{Finiteness properties of $I(s)$.} We use all the notations from Section \ref{Rallis_filtration_sec}, considering the representation $I(s)$ as a representation of $\dot{H}_2\times \ddot{H}_2$.

\begin{lem} The $(\dot{H}_2 \times \ddot{H}_2)$-representation $I(s)$ is compatible with depth decomposition in the sense that $\dot{e}_k I(s) = \ddot{e}_k I(s)$ for the central idempotents defining the depth $k$ categories. \end{lem}

\begin{proof} It is enough to prove the lemma on each sub-quotient of the filtration of $I(s)$. Parabolic induction functors are compatible with depth \textit{i.e.} $\dot{e}_k \cdot \mathfrak{i}_{\dot{P}_t}^{\dot{H}_2} = \mathfrak{i}_{\dot{P}_t}^{\dot{H}_2} \circ \mathfrak{e}_k$ where $\mathfrak{e}_k$ is the the depth $k$ idempotent of the Levi of $\dot{P}_t$. Because the regular representation is compatible with depth and the characters $\chi$ have depth 0, we can equally consider that $\mathfrak{e}_k$ is the idempotent associated to the Levi of $\ddot{P}_{t+\delta}$. Using again the compatibilty of depth with parabolic induction $\mathfrak{i}_{\ddot{P}_{t+\delta}}^{\ddot{H}_2} \circ \mathfrak{e}_k = \ddot{e}_k \cdot \mathfrak{i}_{\ddot{P}_{t+\delta}}^{\ddot{H}_2}$ yields the result. \end{proof}

\begin{lem} $I(s)$ is locally finitely generated \textit{i.e.} its depth $k$ factor is finitely generated as an $(\dot{H}_2 \times \ddot{H}_2)$-representation. \end{lem}

\begin{proof} Again it is enough to prove it on each sub-quotient of the filtration of $I(s)$. As parabolic induction preserves depth \cite[II.5.12]{vig_book} and finite generation \cite[Cor 1.5]{dhkm-finiteness}, it is enough to show the result for the representation of $\dot{P}_t \times \ddot{P}_{t+\delta}$ we are inducing from. Therefore it is equivalent to ask whether the regular representation is locally finitely generated as a bi-module, which is true as a consequence of the depth decomposition and the finiteness of Hecke algebras. \end{proof}

\begin{lem} \label{degenerate_principal_series_admissible_centre_lem} $I(s)$ is admissible over the centre of any of the two groups \textit{i.e.} for all compact open subgroups $K \subset \dot{H}_2\times \ddot{H}_2$ the module $I(s)^K$ is finitely generated both as a $\mathfrak{z}_{\dot{H}_2}$-module and as a $\mathfrak{z}_{\ddot{H}_2}$-module. \end{lem} 

\begin{proof} We can assume that $K$ is small enough so that the functor of $K$-invariants is an exact functor. Again it is enough to prove the lemma on each sub-quotient of the filtration of $I(s)$. As parabolic induction preserves depth and admissibility \cite[II.2.1 \& II.5.12]{vig_book}, it is enough to show the result for the representation of $\dot{P}_t \times \ddot{P}_{t+\delta}$ we are inducing from. Therefore it is equivalent to ask whether the regular representation is admissible over the Bernstein centre, which is true by the finiteness of Hecke algebras over their centres. \end{proof}

We finally show the main result of this section:

\begin{prop} \label{finite_length_from_any_index_prop} Assume $\Theta(\pi_1,V_2^{m_2})$ admits an irreducible quotient for some $m_2$. Then $\Theta(\pi_1,V_2^{m_2})$ has finite length for all $m_2$. \end{prop}

\begin{proof} Let $\pi_2$ be an irreducible quotient of $\Theta(\pi_1,V_2^{m_2})$. Then using the surjection of Lemma \ref{generalised_doubling_theta_lem} we obtain: 
$$\Theta(\chi_1,V_2^{m_2} \oplus (-V_2^{m_2'}))_{\pi_2} \twoheadrightarrow \pi_2 \otimes \Theta(\pi_1^{[\chi_1]},-V_2^{m_2'}).$$
As a submodule of $I(s)$, the representation $\Theta(\chi_1,V_2^{m_2} \oplus (-V_2^{m_2}))$ inherits all the good properties from the previous paragraph, and in particular, it is compatible with depth, it is locally finitely generated and admissible over either of the Bernstein centres -- to transfer this information we crucially rely on the finiteness of Hecke algebras. Similarly to Theorem \ref{finite_length_theta_p_not_2_thm} we obtain that $\Theta(\chi_1,V_2^{m_2} \oplus (-V_2^{m_2'}))_{\pi_2}$ is finitely generated and admissible over $R$, therefore it has finite length. So $\Theta(\pi_1^{[\chi_1]},-V_2^{m_2'})$ has finite length for all $m_2'$. We apply the same argument again, starting with $\Theta(\pi_1^{[\chi_1]},-V_2^{m_2'})$ and reversing the roles of the two spaces, to obtain that $\Theta(\pi_1,V_2^{m_2})$ has finite length for all $m_2$. \end{proof}

\subsection{End of the proof of Proposition \ref{finite_length_from_cuspidals_prop}} \label{end_proof_prop_finite_length_app}

We continue with the notation of Proposition \ref{finite_length_from_cuspidals_prop}. We consider:
$$V=\mathfrak{i}_{Q_{{}^t X_1^k} \times {}^k H_1^{m_1} \times P_t^{m_2}}^{G_{X_1^k} \phantom{k} \times {}^k H_1^{m_1} \times H_2}(\delta_{{}^t X_1^k,n_2} |\textup{det}_{{}^t X_1^k}|^{\frac{s+k-t}{2}} \otimes \textup{Reg}^{\alpha_t,\xi_t}(G_1^t,G_2^t) \otimes \sigma_1 \otimes \Theta(\sigma_1,V_2^{m_2-t})).$$
Similarly to Appendix \ref{finite_length_from_any_index_sec}, the representation above is compatible with depth, locally finitely generated and admissible over the Bernstein centre $\mathfrak{z}_{G_1^t}$, because the regular representation enjoys these properties and the other representations have finite length. Moreover, the Harish-Chandra morphism is finite \cite[Lem 4.2 \& Th 4.3]{dhkm-finiteness} and it factors, in our case, through $\mathfrak{z}_{G_{X_1^k}} \to \mathfrak{z}_{G_1^t}$. Therefore $V$ is admissible over $\mathfrak{z}_{G_{X_1^k}}$. 

We conclude that $V_{\rho_1}$ has finite length for all $\rho_1 \in \textup{Irr}_R(G_{X_1^k})$ by using Corollary \ref{finite_length_criterion_A_cor}. This conclusion is also valid for any subquotient of $V$ by the finiteness of Hecke algberas.

\bibliographystyle{alpha}
\bibliography{lesrefer}

\end{document}